\documentclass{birkmult}
\usepackage{amsmath}
\usepackage{amsfonts}
\usepackage{amssymb}
\usepackage{amsthm}
\usepackage{ragged2e}
\usepackage[mathscr]{eucal}
\usepackage{xspace}
\usepackage{enumerate}
\usefont{T1}{cmr}{m}{n}
\usepackage[russian,english]{babel}
 \newtheorem{thm}{Theorem}[section]
 \newtheorem*{nthm}{Theorem}
 
 \newtheorem{lem}[thm]{Lemma}
 
 \theoremstyle{definition}
 \newtheorem{defn}[thm]{Definition}
 \theoremstyle{remark}
 \newtheorem{rem}[thm]{Remark}
 
\DeclareSymbolFont{rsfs}{U}{rsfs}{m}{n}
\DeclareSymbolFontAlphabet{\mathscrn}{rsfs}

\DeclareMathOperator{\E}{\mathit{E}}

\DeclareMathOperator{\Mm}{\text{\raisebox{0.3ex}{\(\scriptscriptstyle\circ\)}}}
\DeclareMathOperator{\dist}{\textup{dist}}
\DeclareMathOperator{\Psii}{\Psi}
\renewcommand{\theequation}{\thesection.\arabic{equation}}
 \numberwithin{equation}{section}
 \renewcommand{\labelenumi}{\textup{\theenumi}.}
\renewcommand{\labelenumii}{\textup{\theenumii)}.}
\begin{document}
\title{On  H.\,Weyl and J.\,Steiner
Polynomials}
\author{Victor Katsnelson}
\address{Department of Mathematics\\
 the Weizmann Institute\\
  Rehovot 76100\\
   Israel}
\email{victor.katsnelson@weizmann.ac.il,
victorkatsnelson@gmail.com}
\date{}

\begin{abstract}
The paper deals with root location problems for two classes
of univariate polynomials both of geometric origin.
The first class  discussed, the class of Steiner polynomial,
consists of polynomials, each associated with a compact convex set \(V\subset\mathbb{R}^n\).
A polynomial of this class describes the volume of the set \(V+tB^n\) as a function of \(t\),
where \(t\) is a positive number and \(B^n\) denotes the unit ball in \(\mathbb{R}^n\).
The second class, the class of Weyl polynomials, consists of polynomials, each associated
with a Riemannian manifold \(\mathscr{M}\), where \(\mathscr{M}\) is isometrically
 embedded with positive codimension in \(\mathbb{R}^n\).
A Weyl polynomial describes the volume of a tubular neighborhood of its associated
\(\mathscr{M}\) as a function of the tube's radius.
These polynomials are calculated explicitly in a number of natural examples
such as balls, cubes, squeezed cylinders. Furthermore, we examine how the above mentioned
polynomials are related to one another and how they depend on the standard embedding of
\(\mathbb{R}^n\) into \(\mathbb{R}^m\) for \(m>n\). We find that in some cases the real part
of any Steiner polynomial root will be negative. In certain other cases, a Steiner
polynomial will have only real negative roots. In all of this cases, it can be shown
that all of a Weyl polynomial's roots are simple and, furthermore, that they lie on the
imaginary axis. At the same time, in certain cases the above pattern does not hold.
\end{abstract}
\maketitle{}

\hfill
\begin{minipage}[l]{38.0ex}
 \footnotesize{ Erasmus Darwin, the nephew of the great scientist Charles Darwin, believed
 that sometimes one should perform the most unusual experiments.
 They usually yield no results but when they do~...\,.
 So once he played trumpet in front of tulips for the whole day. The experiment yielded
 no results.}
\end{minipage}
\newpage
\centerline{\Large\textbf{Table of contents.}}
\begin{enumerate}
\item
H.\,Weyl and J.\,Steiner polynomials.
\item
Formulation of main results.
\item
The explicit expressions for the Steiner and Weyl
polynomials associated with the `regular' convex sets.
\item
Weyl and Steiner polynomials of `regular' convex
sets as renormalized Jensen polynomials.
\item
Entire functions of the Hurwitz
and of the Laguerre-P\'olya class.
Multipliers preserving location of roots.
\item
Properties of entire functions
generating Steiner and Weyl polynomials
of `regular' convex sets and their surfaces.
\item
The Hermite-Biehler Theorem and its application.
\item
Properties of Steiner polynomials
\item
The Routh-Hurwitz Criterion.
\item
The case of low dimension:
proofs of Theorems \ref{LDC} and \ref{LDCW}.
\item
Extending the ambient
space.
\item
The Steiner polynomial of the Cartesian
 product\newline of convex sets.
\item
Properties of entire functions generating
  the Steiner and Weyl
Polynomials for the degenerate convex sets
\(\boldsymbol{B^{n+1}\times{}0^q}\).
\item
Concluding remarks.\\[-2.0ex]
\item[]
References.
\end{enumerate}

%%%%%%%%%%%%%%%%%%%%%%%%%%%%%%%%%%%%%%%%%%%%%%%%%%%%%%
\section{H.\,Weyl and J.\,Steiner polynomials.}
%%%%%%%%%%%%%%%%%%%%%%%%%%%%%%%%%%%%%%%%%%%%%%%%%%%%%%%%
Let \(\mathscr{M}\) be a smooth manifold,
\[\dim \mathscr{M}=n,\]
which is embedded injectively into the Euclidean space of a higher
dimension, say \(n+p\), \(p>0\). We identify \(\mathscr{M}\) with
the image of this embedding, so
\[\mathscr{M}\subset\mathbb{R}^{n+p}.\]
For \(x\in\mathscr{M}\), let \(\mathscr{N}_x\) be the normal
subspace to \(\mathscr{M}\) at the point \(x\). \(\mathscr{N}_x\)
is an affine subspace of the ambient space \(\mathbb{R}^{n+p}\),
\[\dim\mathscr{N}_x=p.\]
For \(t>0\), let
\begin{equation}%
\label{Ndisk}%
 D_x(t)=\{y\in\mathscr{N}_x:\,\dist(y,x)\leq{}t\},
\end{equation}%
 where
\(\dist(y,x)\) is the Euclidean distance between \(x\) and \(y\).
If the manifold \(\mathscr{M}\) is compact, and \(t>0\) is small
enough, then
\begin{equation}
\label{NInt}
D_{x_1}(t)\cap{}D_{x_2}(t)=\emptyset\quad\textup{for}\ \
x_1\in\mathscr{M}, x_2\in\mathscr{M}, x_1\not=x_2.
\end{equation}
\begin{defn}
\label{DTN} The set
\begin{equation}
\label{deftub}
\mathfrak{T}_{\,\,\mathscr{M}}^{\mathbb{R}^{n+p}}(t)
\stackrel{\textup{\tiny{}def}}{=}\bigcup_{x\in\mathscr{M}}D_x(t)
\end{equation}
is said to be \textsf{the tube neighborhood} of the manifold
\(\mathscr{M}\), or \textsf{the tube} around \(\mathscr{M}\). The
number \(t\) is said to be \textsf{the radius} of this tube.
\vspace{2.0ex}
\end{defn}%
\noindent
Is it clear that for manifolds \(\mathscr{M}\) \textit{without
boundary},
\begin{equation}
\label{DeTu}%
\mathfrak{T}_{\,\,\mathscr{M}}^{\mathbb{R}^{n+p}}(t)
=\{x\in\mathbb{R}^{n+p}:\,\textup{dist}(x,\mathscr{M})\leq t\},
\end{equation}
where \(\dist(x,\mathscr{M})\) is the Euclidean distance from
\(x\) to \(\mathscr{M}\). Thus, for manifolds without boundary,
the equality  \eqref{DeTu} could also be taken as a definition of
the tube \(\mathfrak{T}_\mathscr{M}(t)\). However, for manifolds
\(\mathscr{M}\) \textit{with boundary} the sets
\(\mathfrak{T}_{\,\,\,\mathscr{M}}^{\mathbb{R}^{n+p}}(t)\) defined
by \eqref{deftub} and \eqref{DeTu}
 do not coincide. In this, more general,
case the tube around \(\mathscr{M}\) should be defined by
\eqref{deftub}, but not by \eqref{DeTu}. Hermann Weyl,
\cite{Wey1}, obtained the following result, which is the starting
point of our work:\\

\begin{nthm}[H.Weyl] \ %
 \textit{Let \(\mathscr{M}\), \(\dim{}\mathscr{M}=n \), be a smooth compact manifold, with or without boundary, which is embedded in the Euclidean space} \(\mathbb{R}^{n+p},\
p\geq{}1\).
\begin{enumerate}
\item[I.]%
\begin{itshape}If \(t>0\) is small enough%
\footnote{\label{Smt}If the condition \eqref{NInt} is satisfied.}, %
then the \((n+p)\)\,-\,dimensional volume %
%\(\textup{Vol}_{n+p}(\mathfrak{T}_{\,\,\,\mathscr{M}}^{\mathbb{R}^{n+p}}(t))\)
of the tube
\(\mathfrak{T}_{\,\,\,\mathscr{M}}^{\mathbb{R}^{n+p}}(t)\) around
\(\mathscr{M}\), considered as a function of the radius \(t\) of
this tube%
 , is a polynomial of the form
\begin{equation}
\label{TuVo}
\textup{Vol}_{n+p}(\mathfrak{T}_{\,\,\,\mathscr{M}}^{\mathbb{R}^{n+p}}(t))=
\omega_p\,t^{p}\Big(
\sum\limits_{l=0}^{[\frac{n}{2}]}u_{2l,p}(\mathscr{M})\cdot{}t^{2l}\Big),
\end{equation}
where
\begin{equation}%
\label{Vpdb}%
 \omega_p=\frac{\pi^{p/2}}{\Gamma(\frac{p}{2}+1)}
\end{equation}%
is the \(p\)-dimensional volume of the unit
\(p\)\,-\,dimensional ball.
\end{itshape}\\
\item[II.]
\begin{itshape}
 The coefficients \(u_{2l,p}(\mathscr{M})\) depend on \(p\) as follows:
 \begin{equation}
 \label{DeOnDi}
u_{2l,p}(\mathscr{M})=%
\frac{2^{-l}\,\Gamma(\frac{p}{2}+1)}{\Gamma(\frac{p}{2}+l+1)}%
\cdot{}w_{2l}(\mathscr{M})\,,
 \quad{} 0\leq{}l\leq{}\left[\textstyle\frac{n}{2}\right]\,,
 \end{equation}
where \emph{\textsf{the values
\(w_{2l}(\mathscr{M}),\,0\leq{}l\leq[\frac{n}{2}]\), may be
expressed only in terms of the intrinsic metric%
\footnote{That is, the metric which is induced on
manifold \ \(\mathscr{M}\) from the ambient space \(R^{n+p}\).}
 of the manifold \ \(\mathscr{M}\).
}} In particular, the constant term \(
u_{0,p}(\mathscr{M})=w_{0}(\mathscr{M})\) is the \(n\)-dimensional
volume of \(\mathscr{M}\):
 \begin{equation}%
 \label{CTV}%
 w_0(\mathscr{M})=\textup{Vol}_n\,(\mathscr{M}).
 \end{equation}%
 \end{itshape}
\end{enumerate}
\end{nthm}
\noindent
H.\,Weyl, \textup{\cite{Wey1}}, expressed the coefficients
\(w_{2l}(\mathscr{M})\) as integrals of certain rather complicated
curvature functions of the manifold \(\mathscr{M}\).
\begin{rem}%
\label{EuC} In the case when \(\mathscr{M}\) is compact without
boundary and even dimensional, say \(n=2m\), the highest coefficient
\(w_{2m}(\mathscr{M})\) is especially interesting:
\begin{equation}
\label{ToCo}%
 w_{2m}(\mathscr{M})=(2\pi)^{m}\chi(\mathscr{M}),
\end{equation}
where \(\chi(\mathscr{M})\) is the Euler characteristic of
\(\mathscr{M}\). (See \cite[Section  1.1]{Gra}.)
\end{rem}%

\begin{defn} %
\label{DEfWP}%
 Let \(\mathscr{M}\), \(\dim{}\mathscr{M}=n\), be a smooth manifold,
 with or without boundary, \(\mathscr{M}\subset\mathbb{R}^{n+p},\,p\geq{}1\).
 Let \(\mathfrak{T}_{\,\,\,\mathscr{M}}^{\mathbb{R}^{n+p}}(t)\)
be the tube of radius \(t\) around \(\mathscr{M}\), see \eqref{DeTu}.\\[1.0ex]
\hspace*{2.0ex} The polynomial \(W_{\,\mathscr{M}}^{\,p}(t)\),
which appears in the expression \eqref{TuVo} :
\begin{equation}%
\label{Soo}%
\mathrm{Vol}_{n+p}\,(\mathfrak{T}_{\,\,\,\mathscr{M}}^{\mathbb{R}^{n+p}}(t))=
\omega_pt^p\cdot{}W_{\,\mathscr{M}}^{\,p}(t)\quad\textup{ for
small positive \(t\),}
\end{equation} %
is said to be \textsf{the H. Weyl
polynomial with the index \(p\) for the manifold \(\mathscr{M}\).}
\end{defn}
\vspace{1.0ex}
\begin{rem}
The radius \(t\) of the tube is a positive number, so the formula
\eqref{Soo} is meaningful for positive \(t\) only. However the
polynomial \(W_{\mathscr{M}}^{\,p}\) is determined uniquely by  its
restriction on any fixed interval \([0,\varepsilon]\),
\(\varepsilon>0\), and we may and will consider this polynomial
for \textsf{every complex} \(t\).
\end{rem}

\begin{defn}
\label{DeWC}%
 Let \(\mathscr{M}\) be a smooth manifold,
 \(\dim{}\mathscr{M}=n \),
which is embedded in the Euclidean space \(\mathbb{R}^{n+p},\
p\geq{}1\), and let \(W_{\mathscr{M}}^{\,p}\) be the Weyl polynomial
of \(\mathscr{M}\) \textup{(}defined by \eqref{NInt},\,\eqref{Soo}\,\textup{)}.
 The coefficients
\(w_{2l}(\mathscr{M}),\,0\leq{}l\leq{}[n/2]\),  which are
\textsf{defined} in terms of  the Weyl polynomial
\(W_{\,\mathscr{M}}^{\,p}\) by the equality
\begin{equation}%
\label{DefWC}%
W_{\,\mathscr{M}}^{\,p}(t)\stackrel{\textup{\tiny
def}}{=}\sum\limits_{l=0}^{[\frac{n}{2}]}%
\frac{2^{-l}\,\Gamma(\frac{p}{2}+1)}{\Gamma(\frac{p}{2}+l+1)}\,%
w_{2l}(\mathscr{M}) \cdot{}t^{2l}\,,
\end{equation}%
are said to be the Weyl coefficients of the manifold
\(\mathscr{M}\).
\end{defn}
\begin{rem}
\label{DecFo}%
 Often, the factor in \eqref{DefWC} appears in a
expanded form:
\begin{equation}
\label{DecFor}
\frac{2^{-l}\,\Gamma(\frac{p}{2}+1)}{\Gamma(\frac{p}{2}+l+1)}=
\frac{1}{(p+2)(p+4)\,\cdots\,\,(p+2l)}\,.
\end{equation}
\end{rem}
\begin{rem}
\label{RieMe} In defining the Weyl polynomials
\(W_{\,\mathscr{M}}^{\,p}\) of the manifold \(\mathscr{M}\) by
\eqref{Soo}, we assumed that \(\mathscr{M}\) is already embedded
into \(\mathbb{R}^{n+p}\). The tube around \(\mathscr{M}\) and its
volume are of primary importance in this definition, so that we, in fact,
define the notion of the Weyl polynomial not for the manifold \(\mathscr{M}\)
\textsf{itself} but for manifold \(\mathscr{M}\), which is already
\textsf{embedded} in an ambient space. Moreover, we assume
implicitly that from the outset the manifold
\(\mathscr{M}\) carries a `natural' Riemannian metric and that
this `original' Riemannian metric coincides with the metric on
\(\mathscr{M}\) induced  by the ambient space
\(\mathbb{R}^{n+p}.\) (In other words, we assume that the
imbedding is isometrical.) However, in this approach the
`original' metric  does not play
 an `explicit' role in the definition \eqref{DTN}-\eqref{Soo}-\eqref{DefWC}
of the Weyl polynomial \(W_{\,\mathscr{M}}^{\,p}\) and the Weyl
coefficients \(w_{2l}(\mathscr{M})\).

 There is another approach to defining the Weyl coefficients
 and the Weyl polynomials which does not require an
actual embedding \(\mathscr{M}\) into the ambient space.
 Starting from the given Riemannian metric on
\(\mathscr{M}\), the Weyl coefficients \(w_{2l}(\mathscr{M})\) can
be introduced  formally, by means of the Hermann Weyl expressions
for \(w_{2l}(\mathscr{M})\) in terms of the given metric on
\(\mathscr{M}\). Then the Weyl polynomials
\(W_{\,\mathscr{M}}^{\,p}(t)\) can be defined by means of the
expression \eqref{DefWC}. In this approach, the intrinsic metric
of \(\mathscr{M}\) is of primary importance, but not the tubes around
\(\mathscr{M}\) and their volumes.
\end{rem}
If the codimension \(p\) of
\(\mathscr{M}\)
 equals one\footnote{In other words,
\(\mathscr{M}\) is a hypersurface in \(\mathbb{R}^{n+1}\).} and
\,\,\(\dim{}\mathscr{M}=n,\)  the
Weyl polynomial is of the form:
\begin{equation}
\label{PW}
\textup{Vol}_{n+1}(\mathfrak{T}_{\,\,\,\mathscr{M}}^{\mathbb{R}^{n+1}}(t))=
2t\cdot{}W_{\mathscr{M}}^{1}(t)\,,\quad
W_{\mathscr{M}}^{1}(t)=
\sum\limits_{l=0}^{[\frac{n}{2}]}u_{2l}(\mathscr{M})\cdot{}t^{2l},
\end{equation}
where
\begin{equation}%
\label{WKRel}%
u _{2l}(\mathscr{M})=
\frac{2^{-l}\Gamma(\frac{1}{2})}{\Gamma(\frac{1}{2}+l+1)}\,w_{2l}(\mathscr{M})\,,\quad
 {}0\leq{}l\leq{}[\textstyle\frac{n}{2}]\,.
\end{equation}%
In \eqref{PW} the `shortened' notation is used:
\(u_{2l}(\mathscr{M})\) instead of \(u_{2l,1}(\mathscr{M})\). The
factor \(2t\) is the one-dimensional volume of the one-dimensional
ball of radius \(t\), that is, the length of the interval
\([-t,t]\).

If the \textit{hypersurface} \(\mathscr{M}\) is
orientable\,\footnote{The orientation of the hypersurface
\(\mathscr{M}\) can be specified by means of the continuous vector
field of unit normals on \(\mathscr{M}\). The half-tubes
\(\mathfrak{T}_\mathscr{M}^{+}(t)\) and
\(\mathfrak{T}_\mathscr{M}^{-}(t)\) are the parts of the tube
\(\mathfrak{T}_\mathscr{M}(t)\) corresponding to the `positive' and
`negative', respectively, directions of these normals.}, then  the tube
\(\mathfrak{T}_\mathscr{M}(t)\) can be decomposed into the union
of two \textit{half-tubes}, say,
\(\mathfrak{T}_\mathscr{M}^{+}(t)\) and
\(\mathfrak{T}_\mathscr{M}^{-}(t)\). The half-tubes
\(\mathfrak{T}_\mathscr{M}^{+}(t)\) and
\(\mathfrak{T}_\mathscr{M}^{-}(t)\) are the parts of the tube
\(\mathfrak{T}_\mathscr{M}(t)\) which are situated  on the
distinct sides of \(\mathscr{M}\).
 In particular, if the hypersurface \(\mathscr{M}\) is the boundary of a set
\(V:\,\mathscr{M}=\partial{}V\), then
\begin{equation}
\mathfrak{T}_\mathscr{M}^{+}(t)=\mathfrak{T}_\mathscr{M}(t)\setminus
V,\quad{}\mathfrak{T}_\mathscr{M}^{-}(t)=\mathfrak{T}_\mathscr{M}(t)\cap{}V\,.
\end{equation}
The \((n+1)\)-\,dimensional volumes
\(\textup{Vol}_{n+1}(\mathfrak{T}_\mathscr{M}^{\,+}(t))\) and
\(\textup{Vol}_{n+1}(\mathfrak{T}_\mathscr{M}^{\,-}(t))\) of the
half-tubes  are also polynomials of \(t\). These polynomials are of
the form\,%
\footnote{The equalities \eqref{WhT}, \eqref{WPST} are some of the
results of the theory of tubes around manifolds. See
\cite{Gra},\,\cite{BeGo},\cite{AdTa}}\,:
\begin{equation}
\label{WhT}
\textup{Vol}_{n+1}(\mathfrak{T}_\mathscr{M}^{\,+}(t))=t\,W_\mathscr{M}^{\,+}(t),\quad
\textup{Vol}_{n+1}(\mathfrak{T}_\mathscr{M}^{\,-}(t))=t\,W_\mathscr{M}^{\,-}(t)\,,
\end{equation}
where:
\begin{subequations}%
 \label{WPST}%
\begin{alignat}{2}
W_\mathscr{M}^{\,+}(t)&=\sum\limits_{l=0}^{[\frac{n}{2}]}u_{2l}(\mathscr{M})\cdot{}t^{2l}&&+
t\,\sum\limits_{l=0}^{[\frac{n+1}{2}]-1}u_{2l+1}(\mathscr{M})\cdot{}t^{2l},\label{WPSTa}\\
W_\mathscr{M}^{\,-}(t)&=\sum\limits_{l=0}^{[\frac{n}{2}]}u_{2l}(\mathscr{M})\cdot{}t^{2l}&&-
t\,\sum\limits_{l=0}^{[\frac{n+1}{2}]-1}u_{2l+1}(\mathscr{M})\cdot{}t^{2l},\label{WPSTb}
\end{alignat}
\end{subequations}
and the coefficients \(u_{2l}(\mathscr{M})\) are the same as those
in \eqref{PW}-\eqref{WKRel}. Unlike the coefficients \(u_{2l}(\mathscr{M})\),
the coefficients \(u_{2l+1}(\mathscr{M})\) depend not only on the
`intrinsic' metric of the manifold \(\mathscr{M}\), but also on
how \(\mathscr{M}\) is embedded into \(\mathbb{R}^{n+1}\). It is
remarkable that when the volumes of the half-tubes are summed:
\begin{equation*}
\label{IEPp}
2\,W_\mathscr{M}(t)=W_\mathscr{M}^{\,+}(t)+W_\mathscr{M}^{\,-}(t),
\end{equation*}
the dependence on how \(\mathscr{M}\) is embedded disappears.
As it is seen from \eqref{WPST},
\(W_\mathscr{M}^{\,-}(t)=W_\mathscr{M}^{\,+}(-t)\), hence
\begin{equation}
\label{IEP}
2\,W_\mathscr{M}(t)=W_\mathscr{M}^{\,+}(t)+W_\mathscr{M}^{\,+}(-t).
\end{equation}
We remark also that the volumes of the half-tubes can be expressed only in the terms
of the polynomial \(W_\mathscr{M}^{\,+}\):
\begin{subequations}
\label{SeSa}
\begin{align}
\label{SeSa+}
\textup{Vol}_{n+1}(\mathfrak{T}_\mathscr{M}^{\,+}(t))&=t\,W_\mathscr{M}^{\,+}(\,\,t\,)
\ \ \textup{\,for small positive \(t\)\,.}\\[1.0ex]
\label{SeSa-}
\textup{Vol}_{n+1}(\mathfrak{T}_\mathscr{M}^{\,-}(t))&=t\,W_\mathscr{M}^{\,+}(-t)
\ \ \textup{for small positive \(t\)\,.}
\end{align}
\end{subequations}
The theory of the tubes around manifolds is presented in
\cite{Gra}, and to some extent in \cite{BeGo}, Chapter 6, and in
\cite{AdTa},\,Chapter 10. The comments of V.Arnold  \cite{Arn} to
the Russian translations of the paper \cite{Wey1} by H.Weyl are
very rich in content.

In the event that the hypersurface \(\mathscr{M}\) is the boundary
of a convex set \(V\): \(\mathscr{M}=\partial V\), the Weyl
polynomial \(W_{\mathscr{M}}^{1}\) can be expressed in terms of
polynomials considered in the theory of convex sets.

In the theory of convex sets the following fact, which was
discovered by Hermann Minkowski, \cite{Min1, Min2}, is of
principal importance: \textit{Let \(V_1\) and \(V_2\) be compact
convex sets in \(\mathbb{R}^n\). For positive numbers \(t_1,t_2\),
let us form the `linear combination' \(t_1V_1+t_2V_2\) of the sets
\(V_1\) and \(V_2\). \textup{(}That is,
\(t_1V_1+t_2V_2=\{t_1x_1+t_2x_2:
x_1\in{}V_1,\,x_2\in{}V_2\}\).\textup{)}
 Then  the \(n\)-dimensional Euclidean volume
\(\textup{Vol}_n(t_1V_1+t_2V_2)\) of this linear combination,
considered as a function of the variables \(t_1,t_2\), is a
homogeneous polynomial of  degree~\(n.\) \textup{(}It may vanish
identically.\textup{)}}\\

Choosing \(V\) as \(V_1\) and the unit ball \(B^n\)
of \(\mathbb{R}^n\) as \(V_2\), we obtain the following:\\[1.5ex]
\textit{Let \(V\) be a compact convex set in \(\mathbb{R}^n\),
\(B^n\) be the unit ball in \(\mathbb{R}^n\). Then the
\(n\)-dimensional volume \(\textup{Vol}_n(V+tB^n)\), considered as
a function of the variable \(t\in[0,\infty)\), is a polynomial of
degree \(n\).}
\begin{defn}
\label{DeMiPo}
 Let \(V\subset{}\mathbb{R}^n\) be a
compact convex set. The polynomial which expresses the
\(n\)-dimensional volume of the linear combination \(V+tB^n\) as a
function of the variable \(t\in[0,\infty)\) is said to be
\textsf{the Steiner polynomial of the set \(V\)} and is denoted
by \(S_{\,V}^{\mathbb{R}{n}}(t)\):\\[-1.5ex]
\begin{equation}
\label{DMP} S_{V}^{\mathbb{R}^n}(t)=%
\textup{Vol}_n(V+tB^n)\,,\quad
(t\in[0,\infty)).
\end{equation}
The coefficients of a Steiner polynomial are denoted by
\(s_{\,k}^{\mathbb{R}^n}(V)\):
\begin{equation}
\label{MiP} S_{\,V}^{\mathbb{R}^{n}}(t)=\sum\limits_{0\leq k\leq
n}s_{\,k}^{\mathbb{R}^n}(V)t^k.
\end{equation}
If there is no need to emphasize that the ambient space is
\(\mathbb{R}^n\), then the shortened notation \(S_V(t)\), \(s_k(V)\)
for the Steiner polynomial and its coefficients,respectively, will be used.
\end{defn}
\vspace{1.5ex}
\noindent
Of course,
\[S_{\,V}^{\mathbb{R}^{n}}(t)=\textup{Vol}_n(\mathfrak{V}_{\,V}^{\,\mathbb{R}^{n}}(t)),\]
where \(\mathfrak{V}_{\,V}^{\mathbb{\,R}^{n}}(t)\) is the
\(t\)-neighborhood of the set \(V\) with respect to
\(\mathbb{R}^n\):
\begin{equation}
\label{DeNeCo}%
\mathfrak{V}_{\,V}^{\,\mathbb{R}^{n}}(t)=\{x\in\mathbb{R}^n:\,\textup{dist}(x,{V})\leq
t\}.
\end{equation}
It is evident that
\begin{equation}
\label{EvEq}%
 s_0(V)=\text{Vol}_n(V), \ \ \text{and}\ \
s_n(V)=\text{Vol}_n(B^n).
\end{equation}
 If the boundary \(\partial V\) of a convex set \(V\) is smooth,
 then the \((n-1)\)-dimensional volume (`the area') of the boundary
\(\partial V\) can be expressed as
\begin{equation}
\label{Cau} s_{1}(V)=\text{Vol}_{\,n-1}(\partial V)\,.
\end{equation}
For a convex set \(V\), whose boundary \(\partial V\) may be
non-smooth, the formula \eqref{Cau} serves as a
\textit{definition} of the \textit{`area'} of \(\partial V\). (See
\cite[\textbf{31}]{BoFe}, \cite[\S\,24]{Min1},~ %
\cite[\textbf{6.4}]{Web}.)
Let us emphasize that the Steiner polynomial is defined for an
 \textit{arbitrary} compact convex set \(V\), without any extra
 assumptions. The boundary of \(V\) may be non-smooth, and the
 interior of \(V\) may be empty. In particular, the Steiner
 polynomial is defined for any convex polytope.
\begin{defn}
\label{solid} Let \(\,V\subset\mathbb{R}^n\) be a convex set.
\(V\) is said to be \textsf{solid} if the interior of \(V\) is not
empty, and \textsf{non-solid} if the interior of \(V\) is
empty.
\end{defn}
\begin{defn}
A set \(\mathscr{M}\subset\mathbb{R}^{n+1}\) is called an
\textsf{\(n\)-dimensional closed convex surface} if there exists a
solid compact convex set \(V\subset\mathbb{R}^{n+1}\), such that
\begin{equation}
\label{DeCoS}%
\mathscr{M}=\partial{}V\,.
\end{equation}
The set \(V\) is said to be the \textsf{generating set for the
surface \(\mathscr{M}\)}.
\end{defn}
 \begin{lem}
\label{MPWP}
  \textit{If the closed \(n\)\,-\,dimensional convex surface \(\mathscr{M}\)
  is also a smooth manifold,
    then the Weyl polynomial \(W_{\mathscr{M}}^{\,1}\)
 of the surface \(\mathscr{M}\) and the Steiner polynomial
  \(S_{\,\,V}^{\mathbb{R}^{n+1}}\)
 of its generating set \(V\) are related in the following way}:
 \begin{equation}
 \label{WMP}
 2t\,W_{\mathscr{M}}^{\,1}(t)=
 S_{\,\,V}^{\mathbb{R}^{n+1}}(t)-S_{\,\,V}^{\mathbb{R}^{n+1}}(-t).
\end{equation}
\end{lem}
\noindent
 \begin{proof}[Proof of Lemma \ref{MPWP}] We assign the positive orientation to the vector field
 of exterior normals on \(\partial V\). Let
 \(\mathfrak{T}_{\partial{}V}^{+}(t)\) be
 the `exterior' half-tube around \(\partial V\).
 For  positive \(t\),
\begin{equation*}
V+tB^{n+1}=V\cup{}\mathfrak{T}_{\partial{}V}^{+}(t),
\end{equation*}
Moreover the set \(V\) and \(\mathfrak{T}_{\partial{}V}^{+}(t)\)
do not  intersect. Therefore,
\[\textup{Vol}_{n+1}(V+tB^{n+1})=
\textup{Vol}_{n+1}(V)+\textup{Vol}_{n+1}(\mathfrak{T}_{\partial{}V}^{+}(t)).\]
Hence,
\begin{equation*}
S_V(t)=S_V(0)+t\,W_\mathscr{M}^{\,+}(t),\ \ \ %
{\mathscr{M}=\partial V}, %
\end{equation*}
where \(W_\mathscr{M}^{+}\) is a polynomial of the form
\eqref{WhT} (with \(n\) replaced by \(n+1\): now \(\dim
V=n+1\)). It follows, furthermore, that:
\begin{equation*}
S_V(-t)=S_V(0)-t\,W_\mathscr{M}^{\,+}(-t).
\end{equation*}
The equality \eqref{WMP} follows from the latter equality and from
\eqref{IEP}.
\end{proof}

 Since the Steiner polynomial is defined for an
arbitrary compact convex set, the formula \eqref{WMP} can serve as
a \textit{definition} for the Weyl polynomial of an
\textit{arbitrary} closed convex surface, smooth or non-smooth.
Furthermore, we can define the Weyl polynomial for the `improper
convex surface \(\partial{}V\)', where \(V\) is a non-solid
compact convex set.
\begin{defn}
\label{Impr} Let \(V,\,V\subset\mathbb{R}^{n+1},\) be a compact
convex set. The boundary \(\partial{}V\) of the set \(V\) is said
to be the \textsf{boundary surface} of \(V\). The boundary surface
of \(V\) is said to be \textsf{proper} if \(V\) is solid, and
\textsf{improper} if \(V\) is non-solid.
\end{defn}

The following improper closed convex surface plays a role in what
follow:
\begin{defn}
\label{SqCyl} Let  \(V\subset\mathbb{R}^{n}\) be a compact
convex set, which is solid \emph{with respect to
\(\mathbb{R}^{n}\).} We identify \(\mathbb{R}^{n}\) with its image
\(\mathbb{R}^{n}\times{}0\) using the `canonical' embedding%
\footnote{The point \(x\in\mathbb{R}^{n}\) is identified with the
point \((x,0)\in\mathbb{R}^{n+1}\).} %
 \(\mathbb{R}^{n}\) into \(\mathbb{R}^{n+1}\) and the set \(V\)
 with the set \(V\times{}0\), considered as a subset of
 \(\mathbb{R}^{n+1}\): \(V\times{}0\subset{}\mathbb{R}^{n+1}\).
 The set \(V\times{}0\),  considered as a subset of
 \(\mathbb{R}^{n+1}\), is said to be \textsf{the squeezed cylinder with
 the base \(V\).}
\end{defn}
\vspace{0.5ex}
\begin{rem}
\label{IntSqCy}%
The set \(V\times{}0\) can be interpreted as a ` cylinder of
height' zero, whose `lateral surface' is the Cartesian product
\(\partial{}V\times[0,0]\) and whose bases, lower and upper, are given
the sets \(V\times{}(-0)\) and  \(V\times{}(+0)\), respectively\textup{:}
\begin{equation}
\label{SuArIn}
\partial(V\times{}0)=\big((\partial{}V)\times{}[0,0]\big)\cup\big({V\times{}(-0)}\big)
\cup\big({V\times{}(+0)}\big)\,.
\end{equation}
In other words, the boundary surface \(\partial{}(V\times{}0)\)
 can be considered as `the doubly covered' set \(V\). In particular,
\begin{equation}%
\label{}%
\dim\partial{}(V\times{}0)=n\,.
\end{equation}%
and the number
\(\textup{Vol}_n(V\times(-0))+\textup{Vol}_n(V\times(+0))=2\,\textup{Vol}_n(V)\)
can be naturally interpreted as the `\(n\)-\,dimensional area' of
the \(n\)-\,dimensional convex surface (improper)
\(\partial(V\times{}0)\):
\begin{equation}
\label{SuAr}
\textup{Vol}_n(\partial(V\times{}0))=2\,\textup{Vol}_n(V)\,.
\end{equation}
\end{rem}

On the other hand, the equality \eqref{Cau}, in which the squeezed
cylinder \(V\times{}0\subset\mathbb{R}^{n+1}\) plays the role of
the set \(V\subset{}\mathbb{R}^{n}\), takes the form
\begin{equation}
\label{OOH}
\textup{Vol}_n(\partial(V\times{}0))=s_{\,1}^{\mathbb{R}^{n+1}}(V\times{}0)\,,
\end{equation}
where
\(s_{\,k}^{\mathbb{R}^{n+1}}(V\times{}0),\,\,k=0,1,\,\ldots\,,\,n+1,\)
are the coefficients of the Steiner polynomial
\(S_{\,V\times{}0}^{\mathbb{R}^{n+1}}(t)\) of the squeezed
cylinder \(V\times{}0\) with respect to the ambient space
\(\mathbb{R}^{n+1}\). (See \eqref{MiP}.)

 In section \ref{ExInSp} we prove the following statement, which there
 appears as Lemma \ref{IPRom}:
\begin{lem}
\label{IPR}%
Let \(V\) be a compact convex set in \(\mathbb{R}^n\), and
\begin{equation}%
\label{IRP1}%
 S_{\,V}^{\mathbb{R}^n}(t)=\sum\limits_{0\leq{}k\leq{}n}s_{\,k}^{\mathbb{R}^n}(V)t^k
\end{equation}%
be the Steiner polynomial with respect to the ambient space
\(\mathbb{R}^n\). Then the Steiner polynomial
\(S_{\,V\times{}0}^{\mathbb{R}^{n+1}}(t)\) with respect to the
ambient space \(\mathbb{R}^{n+1}\) is equal to:
\begin{equation}%
\label{MPn1}%
S_{\,V\times{}0}^{\mathbb{R}^{n+1}}(t)=%
t\!\!\sum\limits_{0\leq{}k\leq{}n}\frac{\Gamma(\frac{1}{2})
\Gamma(\frac{k}{2}+1)}{\Gamma(\frac{k+1}{2}+1)}\,s_{\,k}^{\mathbb{R}^n}(V)%
\, t^k\,.
\end{equation}%
\end{lem}
\vspace{2.5ex}
\noindent
So,%
 \[s_{\,0}^{\mathbb{R}^{n+1}}(V\times{}0)=0,\quad{}
s_{k+1}^{\mathbb{R}^{n+1}}(V\times{}0)=
\frac{\Gamma(\frac{1}{2})\Gamma(\frac{k}{2}+1)}{\Gamma(\frac{k+1}{2}+1)}\,
s_{\,k}^{\mathbb{R}^n}(V),\,k=0,\,\ldots,\,n\,.\] %
In particular, \(s_{\,1}^{\mathbb{R}^{n+1}}(V\times{}0)
=2s_{\,0}^{\mathbb{R}^n}(V).\) Since
\(s_{\,0}^{\mathbb{R}^n}(V)=\textup{Vol}_n(V)\), see \eqref{EvEq},
\begin{equation}
\label{AgMP}
s_{\,1}^{\mathbb{R}^{n+1}}(V\times{}0)=2\,\textup{Vol}_n(V)\,.
\end{equation}
The equalities \eqref{SuAr}, \eqref{OOH}  and \eqref{AgMP} agree.
%\hfill\framebox[0.45em]{ } %

\begin{rem}%
\label{AprC1} Any non-solid compact convex set \(V\) can be
presented as the limit \textup{(}in the Hausdorff
\,metric\textup{)} of a monotonic\,%
%%%%%%%%%%%%%%%%%%%%%%%%%%%%%%%%%%%%%%%%%%%%%%%%%
\footnote{\label{monot}The monotonicity means that
\(V_{\varepsilon^\prime}\supset{}V_{\varepsilon^{\prime\prime}}\supset{}V\)
for \(\varepsilon^{\prime}>\varepsilon^{\prime\prime}>0\).}
%%%%%%%%%%%%%%%%%%%%%%%%%%%%%%%%%%%%%%%%%%%%%%%%%
 family \(\{V_{\varepsilon}\}_{\varepsilon>0}\)
of solid convex sets \(V_{\varepsilon}\)~\textup{:}
\[V=\lim_{\varepsilon\to+0}V_{\varepsilon}.\]
Moreover, the approximating family
\(\{V_{\varepsilon}\}_{\varepsilon>0}\) of convex sets
 can be chosen  so that the boundary
\(\partial{}(V_{\varepsilon})\) of each set \(V_{\varepsilon}\) is
a smooth surface. Thus, the improper convex surface
\(\partial{}V\) may be presented as the limit of proper convex
smooth surfaces \(\partial{}(V_{\varepsilon})\) which shrink to
\(\partial{}V\):
\[\partial{}V=\lim_{\varepsilon\to+0}\partial{}(V_{\varepsilon}).\]
\end{rem}%
\begin{defn}
\label{DeWPDC} Let \(V,\,V\subset\mathbb{R}^{n+1},\) be an
arbitrary compact convex set. The \textsf{Weyl polynomial
\(W_{\partial{}V}^{\,1}(t)\)} of the convex surface
\(\mathscr{M}=\partial{}V\),
 proper or improper, is \textsf{defined} by the formula \eqref{WMP}.
 In other words, the Weyl polynomial \(t\,W_{\partial{}V}^{\,1}\)
 is defined as the odd part of the Steiner polynomial
  \(S_{\,\,V}^{\mathbb{R}^{n+1}}\):
  \begin{equation}%
  \label{MPOP}
  t\cdot{}W_{\partial{}V}^{\,1}(t)={}^{\mathscr{O}}\!S_{\,\,V}^{\mathbb{R}^{n+1}}(t),
  \end{equation}%
  where the even part \({}^{\mathscr{E}}\!P\) and
  the odd part \({}^{\mathscr{O}}\!P\) of an arbitrary polynomial \(P\) are
  defined as  \({}^{\mathscr{E}}\!P(t)=\frac{1}{2}(P(t)+P(-t))\),
  \({}^{\mathscr{O}}\!P(t)=\frac{1}{2}(P(t)-P(-t))\), respectively. (See
  \textup{Definition \ref{DefRIPa}}.)
\end{defn}
\begin{rem}
\label{agDe}%
 In the case when the set \(V\) is solid and its
boundary \(\partial{}V\) is smooth, both definitions,
\textup{Definition~\ref{DeWPDC}} \,and
\textup{Definition~\ref{DEfWP}} \,of the Weyl polynomial
\(W_{\,\partial{}V}^{\,1}\), are applicable to \(\partial{}V\). In
this case both definitions agree.
\end{rem}

\begin{rem}%
\label{why}%
\textup{Why  would it  be useful to consider improper convex
surfaces and
their Weyl polynomials?}\\
As it was mentioned earlier, \textup{(Remark \ref{AprC1})}, every improper
convex surface \(\partial{}V\) is a limiting object for a family
of proper smooth convex surfaces \(\partial{}(V_{\varepsilon})\).
It turns out that the Weyl polynomial for this improper surface is
the limit of the Weyl polynomials for this `approximating' family
\(\{V_{\varepsilon}\}_{\varepsilon>0}\) of smooth proper
surfaces.\\
The Weyl polynomials  for the improper surface \(\partial{}V\)
may, therefore, be useful for studying  the limiting behavior of the family
of the Weyl polynomials \(\) for the proper surfaces
\(\partial{}(V_{\varepsilon})\) shrinking to the improper surface
\(\partial{}V\). In particular, see \textup{Theorem \ref{NMR}}
formulated at the end of \textup{Section \ref{FMR}}, and its proof,
presented at the end of \textup{Section \ref{ExInSp}.}
\end{rem}%
\vspace{1.0ex}
Let \(\mathscr{M}\) be an \(n\)\,-\,dimensional closed convex
surface, which is not assumed to be smooth, and \(V\)  the
generating convex set for \(\mathscr{M}\):
\(\mathscr{M}=\partial{}V\). Let \(S_{V}^{\mathbb{R}^{n+1}}\) be
the Steiner polynomial for \(V\), defined by Definition
\ref{DeMiPo}. According to Definition \ref{DeWPDC}, the Weyl
polynomial \(W_{\mathscr{M}}^{\,1}\) is equal to
\begin{equation}%
\label{WPHC}%
W_{\mathscr{M}}^{\,1}(t)=
\sum\limits_{0\leq{}l\leq{}\left[\frac{n}{2}\right]}s_{2l+1}(V)t^{2l},
\end{equation}%
or, alternatively,
\begin{equation}%
\label{IOW}%
u_{2l}(\mathscr{M})=s_{2l+1}(V),\quad
0\leq{}l\leq{}[\textstyle\frac{n}{2}],
\end{equation}%
 where \(u_{2l}(\mathscr{M})\) are the coefficients of the Weyl polynomial
\(W_{\mathscr{M}}^{1}\), \eqref{PW}, of the \(n\)-\,dimensional
surface \(\mathscr{M}\) with respect to the ambient space
\(\mathbb{R}^{n+1}\) and
 \(s_k(V),\,k=2l+1,\) are the coefficients of
the Steiner polynomial~~\(S_{\,\,\,V}^{\mathbb{R}^{n+1}}\):
\begin{equation}
\label{MPGs}
S_{\,\,\,V}^{\mathbb{R}^{n+1}}(t)=\textup{Vol}_{n+1}(V+tB^{n+1}),\quad
S_{\,\,\,V}^{\mathbb{R}^{n+1}}(t)=
\sum\limits_{0\leq{}k\leq{}n+1}s_{k}(V)t^{k}\,.
\end{equation}
\begin{defn}
\label{DNWP}%
 Given a closed \(n\)-\,dimensional convex
surface \(\mathscr{M}\), proper or not,
\(\mathscr{M}=\partial{}V\), the numbers
 \(w_{2l}(\mathscr{M}),\,\,0\leq{}l\leq{}[\frac{n}{2}]\), are \textsf{defined} as
\begin{equation}%
\label{NWP}%
w_{2l}(\mathscr{M})=2^l\frac{\Gamma(l+\frac{1}{2}+1)}{\Gamma(\frac{1}{2}+1)}
\,s_{\,2l+1}^{\mathbb{R}^{n+1}}(V),
 \end{equation}
 where \( s_{\,k}^{\mathbb{R}^{n+1}}(V), \,k=2l+1,\,\) are the coefficients of the
 Steiner polynomial \(S_{\,\,V}^{\mathbb{R}^{n+1}}\) for the
  generating set \(V\), \eqref{MPGs}.
 The numbers
 \(w_{2l}(\partial{}V),\,\,0\leq{}l\leq{}\left[\frac{n}{2}\right]\),
 are said to be \textsf{the Weyl coefficients for the surface
 \(\mathscr{M}\)}.
\end{defn}%
\begin{rem}
According to \textup{Lemma \ref{IPR}}, in the event that the
(improper) convex surface \(\mathscr{M}, \dim \mathscr{M}=n,\) is
the boundary of the squeezed cylinder (see Definition
\ref{SqCyl}), that is, if \(\mathscr{M}=\partial{}(V\times{}0),\)
where \(V\subset\mathbb{R}^n\), then the Weyl coefficients
\(w_{2l}(\mathscr{M}),\,\,0\leq{}l\leq{}[\frac{n}{2}]\), are:
\begin{equation}
\label{WCSqC}
w_{2l}(\mathscr{M})=2^{l+1}\,\Gamma(l+1)\,s_{2l}^{\mathbb{R}^{n}}(V)\,,
\end{equation}
where \( s_{\,k}^{{\mathbb{R}^{n}}}(V), \,k=2l,\,\) are the
coefficients of the
 Steiner polynomial \(S_{\,\,V}^{\mathbb{R}^{n}}\) for
  the base   \(V\) of the squeezed cylinder \(\partial{}(V\times{}0)\).
\end{rem}
\begin{rem}
In the case when the convex surface \(\mathscr{M}\),
\(\mathscr{M}=\partial{}V\), is smooth and `proper', that is, the
set \(V\) generating the surface \(\mathscr{M}\) is solid, both
definitions, Definition~\ref{DNWP} \,and Definition~\ref{DeWC}
\,of the Weyl coefficients \(w_{2l}(\mathscr{M})\) are applicable.
In this case, accordingly to \ \eqref{PW}-\eqref{WKRel} and \eqref{IOW}-\eqref{NWP},%
\footnote{Actually, the equalities \eqref{WKRel}, \eqref{IOW}
served as a motivation for Definition \ref{DNWP}.}
 both definitions agree.
\end{rem}
Note, that according to \eqref{Cau}, (see also Remark
\ref{IntSqCy}),
\begin{equation}
\label{WCA} w_0(\mathscr{M})=\textup{Vol}_{n}(\mathscr{M})
\end{equation}
for every \(n\)\,-\,dimensional closed convex surface
\(\mathscr{M}\).

\begin{lem}
\label{Posi}%
{\ }\textup{\textsf{I}}. Let \(V\), \(V\subset\mathbb{R}^n\), be a
 solid (with respect to \(\mathbb{R}^n\)) compact convex set. Then
the coefficients \(s_k^{{\mathbb{R}}^n}(V),\,0\leq{}k\leq{}n,\) of
its Steiner
polynomials%
\footnote{See \eqref{DMP}, \eqref{MiP}.} %
are strictly positive:
 \(s_k^{{\mathbb{R}}^n}(V)>0,\,\,\,0\leq{}k\leq{}n\,.\)\\
\hspace*{2.0ex} \textup{\textsf{II}}. Let \(\mathscr{M}\) be a
proper compact convex surface, \(\dim \mathscr{M}=n.\) Then all
its Weyl  coefficients \(w_{2l}(\mathscr{M})\) are strictly
positive\,\textup{:}\,\,\,
\(w_{2l}(\mathscr{M})>0,\,\,0\leq{}l\leq{}[\frac{n}{2}]\)\,.\\
\hspace*{1.5ex} \textup{\textsf{III}}. Let \(\mathscr{M}\) be the
boundary surface%
\footnote{See Definition \ref{SqCyl} and Remark \ref{IntSqCy}.}%
 of a squeezed cylinder, whose base \(V\), \(\dim V=n,\)
is a compact convex set which is solid with respect to
\(\mathbb{R}^n\). Then all its Weyl coefficients
\(w_{2l}(\mathscr{M})\) are strictly positive\,\textup{:}\,\,\,
\(w_{2l}(\mathscr{M})>0,\,\,0\leq{}l\leq{}[\frac{n}{2}]\)\,.
\end{lem}

Statement \textsf{I} of Lemma \ref{Posi} is a consequence of a
more general statement related to the monotonicity properties of
the mixed volumes. This will be discussed later, in Section
\ref{ChPrMiPo}. Statements \textsf{II} and \textsf{III} of
Lemma \ref{Posi} are consequences of the statement \textsf{I} and
\eqref{NWP}, \eqref{WCSqC}.

\begin{defn}
\label{DGWP} Given a closed \(n\)\,-\,dimensional convex surface
\(\mathscr{M}\),
 the \textsf{Weyl polynomial \(W_{\mathscr{M}}^{\,p}\) for \(\mathscr{M}\)}
  with the index \(p,\,\,p=1,\,2,\,3,\,\dots\,\,,\)
is \textsf{defined} as
\begin{equation}
\label{UWP} %
W_{\mathscr{M}}^{\,p}(t)=
\sum\limits_{l=0}^{[\frac{n}{2}]}%
\frac{2^{-l}\,\Gamma(\frac{p}{2}+1)}{\Gamma(\frac{p}{2}+l+1)}%
w_{2l}(\mathscr{M}) \cdot{}t^{2l}\,,
\end{equation}
where the Weyl coefficients \(w_{2l}(\mathscr{M})\) are introduced
in \textup{Definition \ref{DNWP}}.
\end{defn}
Let us emphasize that in Definition \ref{DGWP} no assumption
concerning the smo\-othness of the surface  \(\mathscr{M}\) are
made. We already mentioned that the definitions of the Weyl
coefficients \(w_{2l}\) for smooth manifolds and for convex
surfaces agree. Therefore, if the convex surface \(\mathscr{M}\) is also a smooth manifold,
then the definitions \ref{DEfWP} and \ref{DGWP}
 of the Weyl polynomial agree as well.\\

We also define the infinite index Weyl polynomial
\(W_{\mathscr{M}}^{\,\infty}\).
\begin{defn}%
\label{deLWP}%
Let \(\mathscr{M},\,\dim{}\mathscr{M}=n\) be either a smooth
manifold, or a closed compact convex surface, and let
\(w_{2l}(\mathscr{M}),\,l=0,\,1,\,\ldots\,,\,[\frac{n}{2}]\),  be
the Weyl coefficients of \(\mathscr{M}\), defined by
\textup{Definition \ref{DeWC}} in the smooth case, and by
\textup{Definition \ref{DNWP}} in the convex case. The infinite index  Weyl
polynomial \(W_{\mathscr{M}}^{\,\infty}\) is
defined as
\begin{equation}%
\label{DeWPInfInd}%
W_{\,\mathscr{M}}^{\infty}(t)=\sum\limits_{l=0}^{[\frac{n}{2}]}%
w_{2l}(\mathscr{M})\cdot{}t^{2l}.
\end{equation}%
\end{defn}%
\begin{rem}
\label{InfLimC}%
In view of \eqref{DecFor},
\[W_{\mathscr{M}}^{\,p}(\sqrt{p}t)=w_{0}(\mathscr{M})+\sum\limits_{l=1}^{[\frac{n}{2}]}%
\frac{p^l}{(p+2)(p+4)\,\cdots\,\,(p+2l)}\,
w_{2l}(\mathscr{M})\cdot{}t^{2l}\,.\] Therefore, the polynomial
\(W_{\,\mathscr{M}}^{\infty}(t)\) can be considered as a limiting
object for the family
\(\big\{W_{\mathscr{M}}^{\,p}(\sqrt{p}t)\big\}_{p=1,\,2,\,3,\,\dots\,}\)
of the (renormalized) Weyl polynomials of the index \(p\):
\begin{equation}
\label{LiRe}
W_{\,\mathscr{M}}^{\infty}(t)=\lim_{p\to\infty}W_{\,\mathscr{M}}^{p}(\sqrt{p}t)\,.
\end{equation}
\end{rem}

\begin{center}
\begin{minipage}{0.93\linewidth} \textsl{Thus, the sequence
\begin{math}
\label{SeWePo}\big\{W_{\mathscr{M}}^{\,p}\big\}_{p=1,\,2,\,3,\,\dots\,}
\end{math}
of the Weyl polynomials,
\begin{math}%
%\label{DegDeg1}
\deg{}W_{\mathscr{M}}^{p}= 2{\textstyle\left[\frac{n}{2}\right]}
\end{math}, %
as well as the `limiting' polynomial
\(W_{\,\mathscr{M}}^{\infty}\), are associated  with any closed
\(n\)\,-\,dimensional convex surface \(\mathscr{M}\).}
\end{minipage}
\end{center}

\vspace{2.0ex}%
\noindent%
 \textsf{ Weyl polynomials (and Steiner polynomials
 in the convex case) somehow  describe intrinsic properties of the
 appropriate manifolds. On the other hand, there are remarkable geometrical
 objects such as regular
 polytopes, compact matrix groups, spaces of constant
 curvatures, etc. Our belief is that the Weyl polynomials associated
 with these geometric objects are of fundamental importance and
 possess interesting properties. These
 polynomials should be carefully studied. In particular, the
 following question is natural:\\
 \centerline{\textit{What can we say about the roots of such polynomials?}}}\\

\begin{rem}
\label{EhrPol}
In the theory of lattice polytopes, the \textit{Ehrhart polynomials} are
a  counterpart to the Steiner polynomials. For more on the Ehrhart
polynomials we refer to \cite{BeRo}. See also \cite{Gru}.
Location of the roots of Ehrhart polynomials was studied in
\cite{BLD}, \cite{BHW}.
\end{rem}
%%%%%%%%%%%%%%%%%%%%%%%%%%%%%%%%%%%%%
%%%%%%%%%%%%%%%%%%%%%%%%%%%%%%%%%%%%%%%%%%%%%%%%%%
\section{Formulation of main results.\label{FMR}}
In this section we formulate the main results
on the locations of roots belongings to Steiner and Weyl polynomials
related to convex sets and surfaces.\\

\paragraph{Dissipative and conservative polynomials.\label{DCPol}}
%%%%%%%%%%%%%%%%%%%%%%%%%%%%%%%%%%%%%%%%%%
 We introduce two classes of polynomials: dissipative polynomials
 and conservative polynomials.
In many cases the Steiner polynomials related to convex sets
 are dissipative and the Weyl polynomials are conservative.

\begin{defn}
\label{DeHuPo}
The polynomial \(M\) is said to be \textsf{dissipative} if all roots
of \(M\) are situated in the open left half plane \(\{z:\textup{Re}\,z<0\}.\)
The dissipative polynomials are also called the \textsf{Hurwitz polynomials}
or the \textsf{stable polynomials}.
\end{defn}
\begin{defn}
\label{DeCoPo} The polynomial \(W\) is said to be
\textsf{conservative} if all roots of \(W\) are purely imaginary and
simple, in other words,if all roots of \(W\) are contained in the
imaginary axis \(\{z:\textup{Re}\,z=0\}\) and each of them is of
multiplicity one.
\end{defn}

\begin{thm} %
\label{H10H0}%
Given a closed compact convex surface \(\mathscr{M}\), \(\dim
\mathscr{M}=n\), \(\mathscr{M}=\partial{}V\),
 let \(W_{\mathscr{M}}^{\,1}\) be the Weyl polynomial of index
 \(1\) associated with \(\mathscr{M}\), and let \(S_{\,\,V}^{\mathbb{R}^{n+1}}\) be the Steiner
 polynomial associated with the set \(V\).\\[0.8ex]
\hspace*{2.0ex}If the polynomial \(S_{\,\,V}^{\mathbb{R}^{n+1}}\)
 is dissipative, then the polynomial
\(W_{\mathscr{M}}^{1}\) is conservative.
\end{thm}
The proof of Theorem \ref{H10H0} is based on the relation
\eqref{WMP}. Theorem \ref{H10H0} is derived from \eqref{WMP} using
the Hermite-Biehler Theorem. We do this in Section \ref{HBieT}.

From \eqref{LiRe} it follows that if for every \(p\) the
polynomial \(W_{\mathscr{M}}^{\,p}\) has only purely imaginary
roots, then all the roots of the polynomial
\(W_{\mathscr{M}}^{\,\infty}\) are purely imaginary as well. In
particular, \textit{all the roots of the polynomial
\(W_{\mathscr{M}}^{\,\infty}\) are purely imaginary if for every
\(p\) the polynomial \(W_{\mathscr{M}}^{\,p}\) is conservative}.

However, what is important for us is that, the converse statement:
\begin{lem}
\label{LWPo} If the polynomial \(W_{\mathscr{M}}^{\,\infty}\) is
conservative, then all the polynomials \(W_{\mathscr{M}}^{\,p}\),
\(p=1,\,2,\,3,\,\dots\,\,,\) are conservative as well.
\end{lem}

Lemma \ref{LWPo} is the consequence of a result of Laguerre  about
the multiplier sequences. Proof of Lemma \ref{LWPo} appears at
the end of Section \ref{PEFG}.

Keeping  Lemma \ref{LWPo} in mind,  we will concentrate our efforts
on the study of the location of the roots of the Weyl polynomial
\(W_{\mathscr{M}}^{\,\infty}\) with infinite index.\\

\paragraph{The case of low dimension.}
In this section we discuss the Steiner polynomials of convex
sets \(V,\,V\subset{}\mathbb{R}^n,\) and the Weyl polynomials of
closed convex surfaces \(\mathscr{M}\), \(\dim \mathscr{M}=n,\) for `small'
\(n\): \(n=2,\,3,\,4,\,5\).

\begin{thm}
\label{LDC}%
 Let \(n\) be one of the numbers \(2,\,3,\,4\) or
\(5\), and let \(V,\,V\subset{}\mathbb{R}^n,\) be a solid compact
convex set. Then the Steiner polynomial \(S_{V}^{\mathbb{R}^n}\)
is dissipative.
\end{thm}
\begin{thm}
\label{LDCW}%
 Let \(n\) be one of the numbers \(2,\,3,\,4\) or \(5\), and let
\(\mathscr{M}\) be a closed proper\,%
\footnote{That is, the generating set \(V\) is solid.} %
 convex surface of dimension
\(n\).\\[1.0ex]
\hspace*{1.0ex}Then the following hold: \\[-3.0ex]
\begin{enumerate}%
\item%
The Weyl polynomial \(W_{\mathscr{M}}^{\infty}\) with  infinite
index is conservative.
\item
For every \(p=1,\,2,\,3,\,\dots\,\,,\,\)  the Weyl polynomial
\(W_{\mathscr{M}}^{\,p}\)\,\ with index \(p\) %
 is conservative.
\end{enumerate}
\end{thm}

\begin{rem}
\label{Prio}
After this work was completed, Martin Henk called our attention
to  Remark 4.4 of \cite{Tei}, which
appears on the last page of this paper. In this remark,
the statement is formulated which is very close to our Theorem \ref{LDC}.
Some negative results are stated there as well.
Detailed proofs are lacking.
\end{rem}

Theorem \ref{LDC} and \ref{LDCW} are proved in Section \ref{LoDi}.
In proving these theorems, we combine the Routh-Hurwitz Criterion,
which expresses the property of a polynomial to be dissipative in
terms of its coefficients, and the Alexandrov-Fenchel
inequalities, which express the logarithmic convexity property for
the sequence of the cross-sectional measures of a convex set.\\
%%%%%%%%%%%%%%%%%%
\paragraph{Selected 'regular' convex sets: balls, cubes, squeezed cylinders.}
For large \(n\), the statements analogous to Theorems \ref{LDC}
and \ref{LDCW} do not hold. If \(n\) is large enough, then there
exists solid compact convex sets%
\footnote{Very flat ellipsoids  can be taken as such V. See Theorem \ref{NMR}.} %
 \(V\), \(\dim V=n\), such that the
Min\-kow\-ski polynomials \(S_{\,V}^{\mathbb{R}^{n+1}}\) are not
dissipative and the Weyl polynomials \(W^{p}_{\partial{}V}\) are
not conservative. However, for some `regular' convex sets \(V\),
like balls and cubes, the Weyl polynomials \(W^{p}_{\partial{}V}\)
are conservative and the Steiner polynomials are dissipative in
any dimension.

Let us present the collection of `regular' convex sets and their
boundary surfaces, which we will be dealing with further on. Such sets and
surfaces will be considered for every \(n\), so that we, in fact,
consider families of sets (indexed by their dimensions) and not single sets.

\begin{itemize}%
\item[\(\Diamond\)]
The unit ball \(B^n\):
\begin{gather}
\label{Dub}%
 B^n=\{x=(x_1,\,\ldots\,,\,x_n)\in\mathbb{R}^{n}:
\sum\limits_{1\leq{}k\leq{}n}|x_k|^2\leq{}1\,\},\\
\textup{Vol}_n(B^n)=\frac{\pi^{n/2}}{\Gamma(\frac{n}{2}+1)}\,.
\end{gather}
\item[\(\Diamond\)]
The squeezed spherical cylinder \(B^{n}\times{}0\),
\(B^{n}\times{}0\subset\mathbb{R}^{n+1}\).
\item[\(\Diamond\)]
The unit sphere, \[
S^n=\{x=(x_1,\,\ldots\,,\,x_n,\,x_{n+1})\in\mathbb{R}^{n+1}:
\sum\limits_{1\leq{}k\leq{}n+1}|x_k|^2=1\,\},\]
 in other words, the boundary surface of the
unit ball: \(S^n=\partial{}B^{n+1}\,,\)
\begin{gather}
\label{Dus}%
\textup{Vol}_n(S^n)=(n+1)\,\textup{Vol}_{n+1}(B^{n+1})\,.
\end{gather}
\item[\(\Diamond\)]
 The boundary surface of the squeezed spherical cylinder
\(\partial{}(B^{n}\times{}0)\) :
\begin{equation}
\label{SCub}%
\textup{Vol}_n(\partial{}(B^{n}\times{}0))=2\,\textup{Vol}_n(B^n)\,.
\end{equation}
\item[\(\Diamond\)]
The unit cube \(Q^n\):
\begin{gather}
\label{Duc}%
 Q^n=\{x=(x_1,\,\ldots\,,\,x_n)\in\mathbb{R}^{n}:
\max\limits_{1\leq{}k\leq{}n}|x_k|\leq{}1\,\},\\
 \textup{Vol}_n(Q^n)=2^n\,.
\end{gather}
\item[\(\Diamond\)]
The squeezed cubic cylinder \(Q^{n}\times{}0\),
\(Q^{n}\times{}0\subset\mathbb{R}^{n+1}\).
\item[\(\Diamond\)]
The boundary surface \(\partial{}Q^{n+1}\) of the unit cube:
\begin{equation}
\label{Bsuc}%
\textup{Vol}_n(\partial{}Q^{n+1})=(n+1)\,\textup{Vol}_{n+1}(Q^{n+1}).
\end{equation}
\item[\(\Diamond\)]
The boundary surface of the squeezed cubic cylinder
\(\partial{}(Q^{n}\times{}0)\):
\begin{equation}
\label{SqCub}%
\textup{Vol}_n(\partial{}(Q^{n}\times{}0))=2\,\textup{Vol}_{n}(Q^{n})\,.
\end{equation}
\end{itemize}

\paragraph{Locating roots of the Steiner and Weyl
polynomials\\
related to `regular' convex sets.\\[1.0ex]} %
Let us state the main results about locating roots of the
Steiner polynomials and the Weyl polynomials related to the
above mentioned `regular' convex sets and their surfaces.
\begin{thm}%
\label{LoMP}%
For every \(n=1,\,2,\,3,\,\,\ldots\)\,\textup{:}
\begin{enumerate}%
\item%
The Steiner polynomial \(S_{B^n}^{\mathbb{R}^n}\) associated with the
ball \(B^n\) is dissipative, moreover all its roots are
negative\,%
\footnote{This part of the Theorem is trivial: \(S_{B^n}^{\mathbb{R}^n}(t)=(1+t)^n\)}. %
\item%
The Steiner polynomial
\,\(S_{B^{n}\times{}0}^{\mathbb{R}^{n+1}}\) \,associated with the
squeezed spherical cylinder \(B^{n}\times{}0\) is
of the form%
\footnote{\,\label{FaAp}The factors \(t\) appears
because the set \(B^{n}\times{}0\)
is not solid in \(\mathbb{R}^{n+1}\).} %
 \(S_{B^{n}\times{}0}^{\mathbb{R}^{n+1}}(t)=
t\cdot{}D_{B^{n}\times{}0}^{\mathbb{R}^{n+1}}(t)\), where the
polynomial \(D_{B^{n}\times{}0}^{\mathbb{R}^{n+1}}\) is
dissipative. If \(n\) is large enough, then the polynomial
\(S_{B^{n}\times{}0}^{\mathbb{R}^{n+1}}\) has non-real roots.
\item%
The Steiner polynomial \(S_{Q^n}^{\mathbb{R}^n}\) associated with
the cube \(Q^n\) is dissipative, moreover all its roots are
negative.
\item%
The Steiner polynomial \(S_{Q^{n}\times{}0}^{\mathbb{R}^{n+1}}\)
associated with the squeezed cubical cylinder \(Q^{n}\times{}0\) is of
the form\({}^{\ref{FaAp}}\)
\(S_{Q^{n}\times{}0}^{\mathbb{R}^{n+1}}(t)=
t\cdot{}D_{Q^{n}\times{}0}^{\mathbb{R}^{n+1}}(t)\), where the
polynomial \(D_{Q^{n}\times{}0}^{\mathbb{R}^{n+1}}\) is
dissipative and all roots of the polynomial
\(D_{Q^{n}\times{}0}^{\mathbb{R}^{n+1}}\) are negative.
\end{enumerate}%
\end{thm}
\begin{thm}%
\label{LoWP}%
For every \(n=1,\,2,\,3,\,\,\ldots\)\,\textup{:}
\begin{enumerate}%
\item%
The Weyl polynomials \(W_{\partial{}B^{n+1}}^{\,\infty}(t)\) of
infinite index, as well as the Weyl polynomials
\(W_{\partial{}B^{n+1}}^{\,p}(t)\) of arbitrary finite index
\(p,\,p=1,\,2,\,\ldots\,\),
 associated with
the boundary surface of the ball \(B^{n+1}\) are conservative.
\item%
The Weyl polynomials \(W_{\partial{}(B^{n}\times{}0)}^{\,p}\) of
order\,%
\footnote{\,The case \(p=3\) remains open.} %
 \(p=1,\,p=2\) and \(p=4\) associated with the boundary surface
of the squeezed spherical cylinder \(B^{n}\times{}0\) are
conservative.
\item%
The Weyl polynomials \(W_{\partial{}Q^{n+1}}^{\,\infty}(t)\) of
infinite index, as well as the Weyl polynomials
\(W_{\partial{}Q^{n+1}}^{\,p}(t)\) of arbitrary finite index
\(p,\,p=1,\,2,\,\ldots\,\),
 associated with
the boundary surface of the cube \(Q^{n+1}\) are conservative.
\item%
The Weyl polynomials
\(W_{\partial{}(Q^{n}\times{}0)}^{\,\infty}(t)\) of infinite
index, as well as the Weyl polynomials
\(W_{\partial{}(Q^{n}\times{}0)}^{\,p}(t)\) of arbitrary finite
index \(p,\,p=1,\,2,\,\ldots\,\),
 associated with
the boundary surface of the squeezed cubic cylinder
\(Q^{n}\times{}0\) are conservative.
\end{enumerate}%
\end{thm}
\begin{rem}
\label{EERW} The roots of the Weyl polynomial
\(W^1_{\partial{}B^{n+1}}\) can be found explicitly. Indeed
\[W^1_{\partial{}B^{n+1}}(it)=\textup{Vol}_{n+1}(B^{n+1})\frac{1}{2it}\big((1+it)^{n+1}-(1-it)^{n+1}\big)\,.\]
Changing variable \[t\to\varphi:\,1+it=|1+it|e^{i\varphi},
t=\tg{}\varphi\,,\ \ -\frac{\pi}{2}<\varphi<\frac{\pi}{2}\,,\] we
reduce the equation \(W^1_{\partial{}B^{n+1}}(it)=0\) to the
equation
\[\frac{\sin{}(n+1)\varphi}{\sin{}\varphi}=0,\ \ -\frac{\pi}{2}<\varphi<\frac{\pi}{2}\,.\]
The roots of the latter equation are:
\[\varphi_k=\frac{k\pi}{n+1},
\quad{}-\left[\frac{n}{2}\right]\leq{}k\leq{}\left[\frac{n}{2}\right],\
\ k\not=0\,.\]
 So, the roots \(t_k\) of the equation
\(W^{1}_{\partial{}B^{n+1}}(it)=0\) are
\[t_k=\tg{}\textstyle{\frac{k\pi}{n+1},\quad{}
-\left[\frac{n}{2}\right]\leq{}k\leq{}\left[\frac{n}{2}\right],\ \
k\not=0\,.}\] In particular, the polynomial \(W^{1}_{S^n}\) is
conservative.
\vspace{1.0ex}
\end{rem}
\noindent
\textsf{Negative results}:
\begin{thm}
\label{NRWPSq}%
 Let \(p\in\mathbb{Z}\) with \(p\geq{}5\). If \(n\) is
large enough: \(n\geq{}N(p)\), then the Weyl polynomial
\(W^p_{\partial(B^n\times{}0)}\) is not conservative: some of its
roots do not belong to the imaginary axis.
\end{thm}
 For an integer \(q:q\geq{}1\),  let
\(E_{n,\,q,\,\varepsilon}\) be the \(n+q\)-\,dimensional
ellipsoid:
\begin{subequations}
\label{Ell}
\begin{equation}
\label{Ell1}
E_{n,\,q,\,\varepsilon}=\{(x_1,\,x_2,\,\ldots\,,\,x_n,\,\ldots\,,\,x_{n+q})\in\mathbb{R}^{n+q}:
\sum\limits_{0\leq{}j\leq{}n+q}(x_j/a_j)^2\leq{}1\},%
\end{equation}%
where
\begin{equation}
\label{Ell2} a_j=1\ \ \textup{for}\ \ 1\leq{}j\leq{}n,\ \
a_j=\varepsilon\ \
\textup{for}\ \ n+1\leq{}j\leq{}n+q\,.%
\end{equation}%
\end{subequations}

\begin{thm}\ \
\label{NMR}%
\begin{enumerate}
\item Let \(q\in\mathbb{Z}\) with \(5\leq{}q<\infty\). If \(n\) is large
enough:\,\(n\geq{}N(q)\), and \(\varepsilon\) is small enough:
\(0<\varepsilon\leq\varepsilon(n,\,q)\), then the Steiner
polynomial \(S_{E_{n,\,q,\,\varepsilon}}^{\mathbb{R}^{n+q}}\) is
not dissipative: some of its roots are situated in the open
right-half plane.
\item
Let \(p, q\in\mathbb{Z}\) such that \(q\) is odd, \,
\(p\geq{}1,\,q\geq{}1,\,p+q\geq{}6\,\).  %
If \,\,\(n\) is large enough:\,\(n\geq{}N(p,q)\) and
\(\varepsilon\) is small enough:
\(0<\varepsilon\leq\varepsilon(n,\,p,\,q)\), then the Weyl
polynomial \, \(W^p_{E_{n,\,q,\,\varepsilon}}\) \,is not
conservative: some of its roots do not belong to the imaginary
axis.
\end{enumerate}
\end{thm}%
Proof of Theorem \ref{NMR} is presented in
Section \ref{ExInSp}.

\begin{rem}
\label{ONR}
In the recent paper \cite{HeHe} other examples of solid convex sets were constructed,
having Steiner polynomials with roots in the right half plane.
\end{rem}
%%%%%%%%%%%%%%%%%%%%%%%%%%%%%%%%%%%%%%%%%%%%%%%%%%
\section{The explicit expressions for the Steiner and Weyl
polynomials associated with the `regular' convex sets.\label{EEWMP}}
Hereafter, we use the following identity for the
\(\Gamma\)-\,function:
\begin{equation}
\label{IfGa} \Gamma(\zeta+1/2)\,
\Gamma(\zeta+1)={\pi}^{1/2}2^{-2\zeta}\Gamma(2\zeta+1)\,,\ \forall
\zeta\in\mathbb{C}: 2\zeta\not=-1,\,-2,\,-3,\,\ldots\,\,.
\end{equation}
Let as present explicit expressions for the Steiner polynomials
associated with the `regular' convex sets: balls, cubes, squeezed
cylinders, as well as the expression for the Weyl polynomials
associated with the boundary surfaces of these sets. The items related
to balls are marked by the symbol \(\odot\), the items related
to cubes are marked by the symbol \(\boxdot\) .
\\[0.5ex]
\hspace*{0.0ex}\(\odot\)\hspace*{0.5ex}\textsf{The unit ball
\(B^n\)}.\\
Since \(B^n+tB^n=(1+t)B^n\) for \(t>0\), then, according to
\eqref{DeMiPo},
\begin{equation}
\label{EMPb}
S_{B^n}^{\mathbb{R}^n}(t)=%
\textup{Vol}_n(B^n)\cdot(1+t)^n\,,
\end{equation}
or
\begin{equation}
\label{EMPb1}
S_{B^n}^{\mathbb{R}^n}(t)=%
\textup{Vol}_n(B^n)\sum\limits_{0\leq{}k\leq{}n}\frac{n!}{(n-k)!}\cdot{}\frac{t^k}{k!}\,.
\end{equation}
\hspace*{2.0ex}Thus, the coefficients of the Steiner polynomial
\(S_{B^n}^{\mathbb{R}^n}\) for the ball \(B^n\) are:
\begin{equation}
\label{MiCB}
s_{\,k}^{\mathbb{R}^n}(B^n)=\textup{Vol}_n(B^n)\cdot{}
\frac{n!}{(n-k)!}\cdot{}\frac{1}{k!}\,,\quad{}0\leq{}k\leq{}n\,.
\end{equation}
%%%%%%%%%%%%
\hspace*{0.0ex}\(\odot\)\hspace*{0.5ex}\textsf{The squeezed
spherical
cylinder \(B^{n}\times{}0\)}.\\
The Steiner polynomial for the squeezed spherical cylinder
\(B^{n}\times{}0\) is:
\begin{equation}
\label{MiSqB} S_{B^{n}\times{}0}^{\mathbb{R}^{n+1}}(t)=
\textup{Vol}_n(B^n)\cdot\!t\!\!\sum\limits_{0\leq{}k\leq{}n}\frac{n!}{(n-k)!}%
\frac{{\pi}^{1/2}\Gamma(\frac{k}{2}+1)}{\Gamma(\frac{k+1}{2}+1)}\frac{1}{k!}\,
t^k\,.
\end{equation}
The expression \eqref{MiSqB} is derived from \eqref{EMPb1} and
\eqref{IRP1}-\eqref{MPn1}. (See Lemma \ref{IPR}.)\\
\hspace*{1.5ex}Thus, the coefficients of the Steiner polynomial
\(S_{B^n\times{}0}^{\mathbb{R}^{n+1}}\) for the squeezed spherical
cylinder \(B^{n}\times{}0\) are:
\begin{multline}
\label{CoMiSqB}%
 s_{\,0}^{\mathbb{R}^{n+1}}(B^n\times{}0)=0,\quad
s_{\,k+1}^{\mathbb{R}^{n+1}}(B^n\times{}0)=\\
=\textup{Vol}_n(B^n)\cdot{}\frac{n!}{(n-k)!}\cdot{}
\frac{{\pi}^{1/2}\Gamma(\frac{k}{2}+1)}%
{\Gamma(\frac{k+1}{2}+1)}\frac{1}{k!}\,,\quad{}0\leq{}k\leq{}n\,.
\end{multline}

\noindent
\hspace*{0.0ex}\(\odot\)\hspace*{0.5ex}\textsf{The unit sphere
\(S^n=\partial{}B^{n+1}\)}.\\
According to \eqref{NWP} and \eqref{MiCB}, the Weyl coefficients
of the \(n\)-\,dimensional sphere \(S^n=\partial{}B^{n+1}\) are:
\begin{equation}
\label{WCSp} w_{2l}(\partial{}B^{n+1})=\textup{Vol}_{n}
(\partial{}B^{n+1})\cdot\frac{n!}{(n-2l)!}\cdot\frac{1}{l!}\frac{1}{2^l}\,,
\quad{}0\leq{}l\leq[{\textstyle\frac{n}{2}}]\,.
\end{equation}
Thus, the Weyl polynomials associated with the \(n\)-\,dimensional
sphere are:
\begin{multline}
\label{WPpS} %
W_{\partial{}B^{n+1}}^{\,p}(t)=\textup{Vol}_{n}
(\partial{}B^{n+1})\cdot{}\\
\cdot{}\sum\limits_{l=0}^{[\frac{n}{2}]}%
\frac{n!}
{(n-2l)!}\cdot{}\frac{2^{-l}\,\Gamma(\frac{p}{2}+1)}{\Gamma(\frac{p}{2}+l+1)}
\cdot{}\frac{1}{l!}\cdot{}\Big(\frac{t^2}{2}\Big)^l,\ \
p=1,\,2,\,\ldots\,.
\end{multline}
\begin{equation}
 \label{WPiS}
 W_{\partial{}B^{n+1}}^{\,\infty}(t)=\textup{Vol}_{n}
(\partial{}B^{n+1})
\cdot{}\sum\limits_{l=0}^{[\frac{n}{2}]}%
\frac{n!}
{(n-2l)!}\cdot{}\frac{1}{l!}\cdot{}\Big(\frac{t^2}{2}\Big)^l\,\cdot
\end{equation}
\hspace*{0.0ex}\(\odot\)\hspace*{0.5ex}\textsf{The boundary
surface \(\partial(B^{n}\times{}0)\) of the
squeezed spherical cylinder \(B^{n}\times{}0\)}.\\
According to \eqref{WCSqC} and \eqref{MiCB}, the Weyl coefficients
of the \(n\)-\,dimensional improper surface
\(\partial(B^{n}\times{}0)\) are:
\begin{equation}
\label{WCSqB} w_{2l}(\partial(B^{n}\times{}0))=\textup{Vol}_{n}
(\partial(B^{n}\times{}0))\cdot\frac{n!}{(n-2l)!}\cdot\frac{\Gamma(1/2)}{\Gamma(l+1/2)}\,\frac{1}{2^l},
\quad{}0\leq{}l\leq[{\textstyle\frac{n}{2}}]\,.
\end{equation}
Thus, the Weyl polynomials associated with the (improper) surface
\(\partial(B^{n}\times{}0)\) are:
\begin{multline}
\label{WPpSSqC} %
W_{\partial{}(B^{n}\times{}0)}^{\,p}(t)=\textup{Vol}_{n}
(\partial{}(B^{n}\times{}0))\cdot{}\\
\cdot{}\sum\limits_{l=0}^{[\frac{n}{2}]}%
\frac{n!}
{(n-2l)!}\cdot{}\frac{2^{-l}\,\Gamma(\frac{p}{2}+1)}{\Gamma(\frac{p}{2}+l+1)}
\cdot{}\frac{\Gamma(1/2)}{\Gamma(l+1/2)}\cdot{}\Big(\frac{t^2}{2}\Big)^l,\
\ p=1,\,2,\,\ldots\,.
\end{multline}
\begin{equation}
 \label{WPiSSqC}
 W_{\partial{}(B^{n+1}\times{}0)}^{\,\infty}(t)=\textup{Vol}_{n}
(\partial{}(B^{n+1}\times{}0))
\cdot{}\sum\limits_{l=0}^{[\frac{n}{2}]}%
\frac{n!}
{(n-2l)!}\cdot{}\frac{\Gamma(1/2)}{\Gamma(l+1/2)}
\cdot{}\Big(\frac{t^2}{2}\Big)^l\,\cdot
\end{equation}
\hspace*{0.0ex}\(\boxdot\)\hspace*{0.5ex}\textsf{The unit cube
\(Q^n\)}.\\
The Steiner polynomial \(S_{Q^n}^{\mathbb{R}^n}\) is:
\begin{equation}
\label{MPUQ}%
 S_{Q^n}^{\mathbb{R}^n}(t)= \textup{Vol}_{n}%
(Q^{n})\cdot{}\!\!\!\!\sum\limits_{0\leq{}k\leq{}n}\frac{n!}{(n-k)!}
\frac{1}{\Gamma(\frac{k}{2}+1)k!}\Big(\frac{\sqrt{\pi}}{2}\Big)^k\,t^k\,.
\end{equation}
Expression \eqref{MPUQ} is obtained in the following way. The
\(n\)-\,dimensional cube \(Q^{n}\) is considered as the Cartesian
product of the one-dimensional cubes:
\begin{equation*}
\label{CsCP} Q^{n}=Q^1\times\,\cdots\,\times{}Q^1\,.
\end{equation*}
For \(n=1\), the Steiner polynomial  is:
\(S_{Q^1}^{\mathbb{R}^1}(t)=2(1+t)\)\,.
We here use the fact that the Steiner polynomial of a Cartesian product
can be expressed in terms of the Steiner polynomials belonging to
the Cartesian factors. (See
details in Section \ref{MPCaPr}.)\\ %
\hspace*{2.0ex} We find, the coefficients of the Steiner polynomial
for the cube \(Q^n\) to be given by:
\begin{equation}
\label{MiCQ}
s_{\,k}^{\mathbb{R}^n}(Q^n)=\textup{Vol}_n(Q^n)\cdot{}
\frac{n!}{(n-k)!}\cdot{}\frac{1}{\Gamma(\frac{k}{2}+1)k!}
\Big(\frac{\sqrt{\pi}}{2}\Big)^k\,,\quad{}0\leq{}k\leq{}n\,.
\end{equation}
%%%%%%%%%%%%%%%%%%%%%%%%%%%%%%%%%%%%%%%%%%%%%%%%%
\hspace*{0.0ex}\(\boxdot\)\hspace*{0.5ex}\textsf{The squeezed
cubic cylinder \(Q^n\times{}0\)}.\\
The Steiner polynomial \(S_{Q^{n}\times{}0}^{\mathbb{R}^{n+1}}\)
is:
\begin{equation}
\label{MPUQC}%
 S_{Q^{n}\times{}0}^{\mathbb{R}^{n+1}}(t)= \textup{Vol}_{n}%
(Q^{n})\cdot{}t\!\!\!\!\sum\limits_{0\leq{}k\leq{}n}\frac{n!}{(n-k)!}
\frac{\Gamma(\frac{1}{2})}{\Gamma(\frac{k+1}{2}+1)k!}\Big(\frac{\sqrt{\pi}}{2}\Big)^k\,t^k\,.
\end{equation}
The expression \eqref{MPUQC} is derived from \eqref{MPUQ} and
\eqref{IRP1}-\eqref{MPn1}. (See Lemma \ref{IPR}.)\\ %
\hspace*{2.0ex}Thus, the coefficients of the Steiner polynomial
\(S_{Q^n\times{}0}^{\mathbb{R}^{n+1}}\) for the squeezed cubic
cylinder are:
\begin{multline}
\label{CoMiSqQ} s_{\,0}^{\mathbb{R}^{n+1}}(Q^n\times{}0)=0,\quad
s_{\,k+1}^{\mathbb{R}^{n+1}}(Q^n\times{}0)=\\
=\textup{Vol}_n(Q^n)\cdot{}\frac{n!}{(n-k)!}\cdot{}
\frac{\Gamma(\frac{1}{2})}%
{\Gamma(\frac{k+1}{2}+1)}\frac{1}{k!}\Big(\frac{\sqrt{\pi}}{2}
\Big)^k\,,\quad{}0\leq{}k\leq{}n\,.
\end{multline}
%%%%%%%%%%%%%%%%%%%%%%%%%%%%%%%%%%%%%%%
\hspace*{0.0ex}\(\boxdot\)\hspace*{0.5ex}\textsf{The boundary
surface \(\partial{}Q^{n+1}\) of the unit cube \(Q^{n+1}\)}.\\
According to \eqref{NWP} and \eqref{MiCQ}, the Weyl coefficients
of the \(n\)-\,dimensional surface \(\partial{}Q^{n+1}\) are:
\begin{multline}
\label{WCQp} w_{2l}(\partial{}Q^{n+1})=\textup{Vol}_{n}
(\partial{}Q^{n+1})\cdot\frac{n!}{(n-2l)!}\cdot\\
\cdot\frac{1}{\Gamma(l+\frac{1}{2}+1)}\frac{1}{2^l\,l!}\,\Big(\frac{\sqrt{\pi}}{2}\Big)^{2l+1},
\quad{}0\leq{}l\leq[{\textstyle\frac{n}{2}}]\,.
\end{multline}
Taking into account the identity
\(\Gamma(l+1+{\textstyle\frac{1}{2}})\cdot\Gamma(l+1)={\pi}^{1/2}2^{-(2l+1)}\Gamma(2l+2)\),
which is obtained by setting \(\zeta=l+1/2\) in \eqref{IfGa}, we can
rewrite the equality \eqref{WCQp}:
\begin{equation}
\label{aWCQp} w_{2l}(\partial{}Q^{n+1})=\textup{Vol}_{n}
(\partial{}Q^{n+1})\cdot\frac{n!}{(n-2l)!}
\cdot\frac{1}{(2l+1)!}\,\Big(\frac{\pi}{2}\Big)^{l},
\quad{}0\leq{}l\leq[{\textstyle\frac{n}{2}}]\,.
\end{equation}
Thus, the Weyl polynomials associated with the \(n\)-\,dimensional
surface \(\partial{}Q^{n+1}\) are:
\begin{multline}
\label{WPpQ} %
W_{\partial{}Q^{n+1}}^{\,p}(t)=\textup{Vol}_{n}
(\partial{}Q^{n+1})\cdot{}\\
\cdot{}\sum\limits_{l=0}^{[\frac{n}{2}]}%
\frac{n!} {(n-2l)!}\cdot{}\frac{
2^{-l}\,\Gamma(\frac{p}{2}+1)}{\Gamma(\frac{p}{2}+l+1)}
\cdot{}\frac{1}{(2l+1)!}\cdot{}\Big({\frac{{\pi}t^2}{2}}\,\Big)^l,\
\ p=1,\,2,\,\ldots\,.
\end{multline}
\begin{equation}
 \label{WPiQ}
 W_{\partial{}Q^{n+1}}^{\,\infty}(t)=\textup{Vol}_{n}
(\partial{}Q^{n+1})
\cdot{}\sum\limits_{l=0}^{[\frac{n}{2}]}%
\frac{n!}
{(n-2l)!}\cdot{}\frac{1}{(2l+1)!}\cdot{}\Big(\frac{{\pi}t^2}{2}\Big)^l\,\cdot
\end{equation}
%%%%%%%%%%%%%%%%%%%%%%%%%%%%%%%%%%%
\hspace*{0.0ex}\(\boxdot\) \hspace*{0.5ex}\textsf{The boundary
surface \(\partial(Q^{n}\times{}0)\) of the
squeezed cubic cylinder \(Q^{n}\times{}0\)}.\\
According to \eqref{WCSqC} and \eqref{MiCQ}, the Weyl coefficients
of the surface (improper) \(\partial{}(Q^{n}\times{}0)\) are:
\begin{multline}
\label{WCSQp} w_{2l}(\partial{}Q^{n}\times{}0)=\textup{Vol}_{n}
(\partial{}Q^{n}\times{}0)\cdot\frac{n!}{(n-2l)!}\cdot\\
\cdot\frac{\sqrt{\pi}}{\Gamma(l+\frac{1}{2})}\frac{1}{l!\,2^{l}\,}\,\Big(\frac{\pi}{2}\Big)^{l},
\quad{}0\leq{}l\leq[{\textstyle\frac{n}{2}}]\,.
\end{multline}
Using the identity
\(\Gamma(l+1/2)\Gamma(l+1)=\sqrt{\pi}2^{-2l}\Gamma(2l+1)\), which
is obtained by setting  \(\zeta=l\) in \eqref{IfGa},
the equality \eqref{WCSQp} can be rewritten as follows:
\begin{equation}
\label{aWCSQp} w_{2l}(\partial{}Q^{n}\times{}0)=\textup{Vol}_{n}
(\partial{}Q^{n}\times{}0)\cdot\frac{n!}{(n-2l)!}\cdot
\frac{1}{(2l)!}\,\Big(\frac{\pi}{2}\Big)^{2l},
\quad{}0\leq{}l\leq[{\textstyle\frac{n}{2}}]\,.
\end{equation}
Thus, the Weyl polynomials associated with the improper
\(n\)-\,dimensional surface \(\partial{}(Q^{n}\times{}0)\) are:
\begin{multline}
\label{WPpQSqC} %
W_{\partial{}(Q^{n}\times{}0)}^{\,p}(t)=\textup{Vol}_{n}
(\partial{}(Q^{n}\times{}0))\cdot{}\\
\cdot{}\sum\limits_{l=0}^{[\frac{n}{2}]}%
\frac{n!} {(n-2l)!}\cdot{}
\frac{2^{-l}\,\Gamma(\frac{p}{2}+1)}{\Gamma(\frac{p}{2}+l+1)}
\cdot{}\frac{1}{(2l)!}\cdot{}\Big({\frac{{\pi}t^2}{2}}\,\Big)^l,\
\ p=1,\,2,\,\ldots\,.
\end{multline}
\begin{equation}
 \label{WPiQSqC}
 W_{\partial{}(Q^{n}\times{}0)}^{\,\infty}(t)=\textup{Vol}_{n}
(\partial{}(Q^{n}\times{}0))
\cdot{}\sum\limits_{l=0}^{[\frac{n}{2}]}%
\frac{n!}
{(n-2l)!}\cdot{}\frac{1}{(2l)!}\cdot{}\Big(\frac{{\pi}t^2}{2}\Big)^l\,\cdot
\end{equation}

\section{Weyl and Steiner polynomials of `regular' convex
sets
as renormalized Jensen polynomials.\label{WMPJen}}
It would be difficult to directly investigate the location of the roots
for the Steiner
polynomials \(S_{B^n}^{\mathbb{R}^n}\),
\(S_{B^{n}\times{}0}^{\mathbb{R}^{n+1}}\),\,
\(S_{Q^n}^{\mathbb{R}^n}\),\,
\(S_{Q^{n}\times{}0}^{\mathbb{R}^{n+1}}\)\, \,and Weyl polynomials
\(W_{\partial{}B^{n+1}}^{\,p}\),\,
\(W_{\partial{}(B^{n}\times{}0)}^{\,p}\),\,\,
\(W_{\partial{}Q^{n+1}}^{\,\infty}\),\,\,
\(W_{\partial{}(Q^{n}\times{}0)}^{\,p}\) for a \textit{finite}
\(n\). It turns out to be much easier first  to
investigate the roots of the entire functions,
which are the limits of the (renormalized) Steiner and Weyl
polynomials as \(n\to\infty\). The properties of the roots
belonging to the original Steiner and Weyl polynomials can then
be deduced from the properties of these limiting entire functions.\\

\paragraph{Jensen polynomials.}
From the explicit expressions \eqref{EMPb1}, \eqref{MiSqB},
\eqref{MPUQ}, \eqref{MPUQC} for the Steiner polynomials and
\eqref{WPpS}, \eqref{WPiS},  \eqref{WPpSSqC}, \eqref{WPiSSqC},
\eqref{WPpQ}, \eqref{WPiQ}, \eqref{WPpQSqC}, \eqref{WPiQSqC} for
the Weyl polynomials we notice that each of these expressions
contains the factor \(\dfrac{n!}{(n-k)!}\), which is `a part' of
the binomial coefficient \(\binom{n}{k}\). This factor
can be expressed as%
\begin{equation}%
\label{BinRat}%
 \frac{n!}{(n-k)!}=1\cdot{}
\Big(1-\frac{1}{n}\Big)\cdot{}\Big(1-\frac{2}{n}\Big)\cdot\,\,\cdots\,\cdot\,{}
\Big(1-\frac{k-1}{n}\Big)\cdot{}n^k\,,\quad{}\,1\leq{}k\leq{}n\,.
 \end{equation}
\begin{defn} \ \ %
\label{DefJP}%
\begin{enumerate}
\item
Given a formal power series \(f\):
\begin{equation}%
\label{FPS}%
 f(t)=\sum\limits_{0\leq{}l<\infty}a_{l}t^{l}\,,
\end{equation}
we define the following sequence of polynomials
\(\mathscr{J}_n(f;t),\,n=1,\,2,\,3,\,\ldots:\)
\begin{equation}
\label{DJP}%
\mathscr{J}_n(f;t)=\sum\limits_{0\leq{}l\leq{}n}\frac{n!}{(n-l)!}\frac{1}{n^l}\cdot{}a_{l}t^{l},
\end{equation}
or, rewriting the factor \(\frac{n!}{(n-l)!}\frac{1}{n^l}\),
\begin{equation}
\label{DJPd}%
\mathscr{J}_n(f;t)=a_0+\sum\limits_{1\leq{}l\leq{}n}1\big(1-{\textstyle{}\frac{1}{n}}\big)
\big(1-{\textstyle{}\frac{2}{n}}\big)\,\cdots\,\big(1-{\textstyle{}\frac{l-1}{n}}\big)\,
\cdot{}a_{l}t^{l}.
\end{equation}
The polynomials \(\mathscr{J}_n(f;t)\) are said to be \textsf{the
Jensen polynomials associated with the power series \(f\).}\\[0.0ex]
\item
Given a holomorphic \emph{function} \(f\)  in the disc
\(\{t:\,|t|<R\}\), where \(R\leq{}\infty\), we associate the
sequence of the Jensen polynomials with the Taylor series
\eqref{FPS} of the function \(f\) according to \eqref{DJP}.
We denote these
 polynomials by \(\mathscr{J}_n(f;t)\) as well and call them \textsf{the Jensen
 polynomials associated with the function \(f\).}\\[0.0ex]
 \item
 The factors
 \begin{multline}
 \label{JeFa}
  j_{n,0}=1,\quad{}j_{n,k}=1\big(1-{\textstyle{}\frac{1}{n}}\big),
\big(1-{\textstyle{}\frac{2}{n}}\big)\,%
\cdots\,\big(1-{\textstyle{}\frac{k-1}{n}}\big),\,\,1\leq
  k\leq n,\\
j_{n,k}=0,\,\,k>n\,,
 \end{multline}
 are said to be \textsf{the Jensen multipliers}.
 \vspace{1.0ex} %

Thus, the Jensen polynomials associated with an \(f\) of the form
\eqref{FPS}
 can be written as:
 \begin{equation}
\label{DJP1}%
\mathscr{J}_n(f;t)=\sum\limits_{0\leq{}<\infty}j_{n,l}\cdot{}a_{l}t^{l}\,.
\end{equation}
\end{enumerate}
\end{defn}
Since
\begin{math}%
\label{ReLRel}%
 j_{n,k}\to{}1 {}\ \textup{as \(k\) is fixed}, %
 \ n\to\infty\,,
\end{math}
 the following result is evident:
\begin{lem}[The approximation property of Jensen polynomials]{}\ %
\label{CJPL} %
Given the power series \eqref{FPS}, the following statements hold:
\begin{enumerate}
\item
The sequence
 of the Jensen polynomials \(\mathscr{J}_n(f;t)\) converges to the series \(f\)
 coefficients-wise\,;
\item
If, moreover, the radius of convergence of the power series
\eqref{FPS} is positive, say equal to \(R,\,0<R\leq{}\infty\),
then the sequence
 of the Jensen polynomials \(\mathscr{J}_n(f;t)\) converges to the function
 which is the sum of this power series
 locally uniformly in the disc \(\{t:\,|t|<R\}\).
\end{enumerate}
\end{lem}%

The approximation property in not specific to the polynomials
constructed from the \textit{Jensen multipliers} \(j_{n,k}\). This
property holds for \textit{any} multipliers \(j_{n,k}\), which
satisfy the conditions
\( j_{n,k}\to{}1 \,\textup{as \(k\) is fixed},%
\,n\to\infty\,,\) and are uniformly bounded:
\(\sup\limits_{k,n}|j_{n,k}|<\infty\,.\) What is much more
specific, that for some \(f\), the polynomial
\(\mathscr{J}_n(f;t)\) constructed from the \textit{Jensen
multipliers} \(j_{n,k}\) preserve the property of \(f\) to possess
only real roots.
\begin{nthm}
[Jensen] Let \(f\) be a \emph{polynomial} such that all its roots
are real. Then for each \(n\), all roots of the Jensen polynomial
\(\mathscr{J}_n(f,\,t)\) are real as well.
\end{nthm}
This result is a special case of the Schur Composition Theorem
\cite{Schu1}. Actually, Jensen, \cite{Jen}, obtained a more
general result in which formulation \(f\) can be, not only a
polynomial with real roots, but also an entire function belonging to the
 \emph{Laguerre-P\'olya class}. We return to this generalization later, in Section
\ref{LPEF}. Now we focus our attention on the representation of the
Steiner and Weyl polynomials as Jensen polynomials of certain
entire functions.

The relation \eqref{ReLRel} as well as the expressions
\eqref{EMPb1},\,\eqref{MiSqB},\,\eqref{MPUQ},\,\eqref{MPUQC} for
the Steiner polynomials suggest us how the Steiner polynomials
should be renormalized so that the renormalized polynomials tend
to a non-trivial limit as \(n\to\infty\).
%%%%%%%%%%%%%%%%%%%%%%%%%%%%%%%%%%%%%%%%%%%%%%%%%%%%%%%%%%%%%%%%%%%%%%%
\paragraph{Entire functions which generate the Steiner polynomials
for balls, cubes, spherical and cubic cylinders.} Let us introduce
the infinite power series:
\begin{subequations}
\label{LiEFMPS}
\begin{align}
 \mathcal{M}_{B^{\infty}}(t)&=\sum\limits_{0\leq{}k<\infty}
\frac{1}{k!}\,t^k\,;
\label{LiEFMPS1}\\%
\mathcal{M}_{B^{\infty}\times{}0}(t)&=\sum\limits_{0\leq{}k<\infty}
\frac{\Gamma(\frac{1}{2}+1)\Gamma(\frac{k}{2}+1)}{\Gamma(\frac{k+1}{2}+1)}\,
\frac{1}{k!}\,t^k\,;
\label{LiEFMPS2}\\%
\mathcal{M}_{Q^{\infty}}(t)&=\sum\limits_{0\leq{}k<\infty}
\frac{1}{\Gamma(\frac{k}{2}+1)k!}\Big(\frac{\sqrt{\pi}}{2}\Big)^k\,\,t^k\,;
\label{LiEFMPS3}\\%
\mathcal{M}_{Q^{\infty}\times{}0}(t)&=\sum\limits_{0\leq{}k<\infty}
\frac{\Gamma(\frac{1}{2}+1)}{\Gamma(\frac{k+1}{2}+1)k!}\Big(\frac{\sqrt{\pi}}{2}\Big)^k\,t^k\,.%
\label{LiEFMPS4}
\end{align}
\end{subequations}
The series \eqref{LiEFMPS} represent entire functions which grow
 not faster than  exponentially. More precisely, the
functions \(\mathcal{M}_{B^{\infty}}\) and
\(\mathcal{M}_{B^{\infty}\times{}0}\) grow exponentially: they are
of  order  \(1\) and of normal type, the functions
\(\mathcal{M}_{Q^{\infty}}\) and
\(\mathcal{M}_{BQ^{\infty}\times{}0}\) grow subexponentially: they
are of order \(2/3\) and of normal type.

For each of the entire functions \eqref{LiEFMPS} we consider the associated
sequence of Jensen polynomials:
\begin{subequations}
\label{JLiEFMPS}
\begin{align}
 \mathcal{M}_{B^{n}}(t)&{=}
 \mathscr{J}_n(\mathcal{M}_{B^{\infty}};t)\,,
\label{JLiEFMPS1}\\%
\mathcal{M}_{B^{n}\times{}0}(t)&{=}\mathscr{J}_n(\mathcal{M}_{B^{\infty}\times{}0};t)\,,
\label{JLiEFMPS2}\\%
\mathcal{M}_{Q^{n}}(t)&{=}\mathscr{J}_n(\mathcal{M}_{Q^{\infty}};t)\,
\label{JLiEFMPS3}\\%
\mathcal{M}_{Q^{n}\times{}0}(t)&{=}
\mathscr{J}_n(\mathcal{M}_{Q^{\infty}\times{}0};t)\,,%
\label{JLiEFMPS4}
\end{align}
\end{subequations}
From the expressions \eqref{EMPb1}, \eqref{MiSqB}, \eqref{MPUQ},
\eqref{MPUQC} for the Steiner polynomials, it follows that they
are related to the above introduced polynomials \eqref{JLiEFMPS}
as follows:
\begin{subequations}
\label{RemOr}
\begin{alignat}{2}
S_{B^n}^{\mathbb{R}^n}(t)&=
\textup{Vol}_n(B^n)&\,&\mathcal{M}_{B^n}(nt)\,;
\label{RemOr1}\\%
S_{{B^n}\times{}0}^{\mathbb{R}^{n+1}}(t)&=
\textup{Vol}_n(B^{n})\,\omega_{1}t&\,&\mathcal{M}_{B^n\times{}0}(nt)\,;
\label{RemOr2}\\%
S_{Q^n}^{\mathbb{R}^n}(t)&=
\textup{Vol}_n(Q^n)&\,&\mathcal{M}_{Q^n}(nt)\,;
\label{RemOr3}\\%
S_{{Q^n}\times{}0}^{\mathbb{R}^{n+1}}(t)&=
\textup{Vol}_n(Q^n)\,\omega_{1}t&\,&\mathcal{M}_{Q^n\times{}0}(nt)\,;
\label{RemOr4}
\end{alignat}
\end{subequations}
The polynomials \(\mathcal{M}_{B^n}\),
\(\mathcal{M}_{B^n\times{}0}\), \(\mathcal{M}_{Q^n}\),
\(\mathcal{M}_{Q^n\times{}0}\) can be interpreted as
\textit{renormalized Steiner polynomials}. We take
the equalities \eqref{RemOr} as the \textsf{definition} of the
renormalized Steiner polynomials \(\mathcal{M}_{B^n}\),
\(\mathcal{M}_{B^n\times{}0}\), \(\mathcal{M}_{Q^n}\),
\(\mathcal{M}_{Q^n\times{}0}\) in terms of the `original'
Steiner polynomials \(S_{B^n}^{\mathbb{R}^n}\),
\(S_{{B^n}\times{}0}^{\mathbb{R}^{n+1}}\),
\(S_{Q^n}^{\mathbb{R}^n}\), and
\(S_{{Q^n}\times{}0}^{\mathbb{R}^{n+1}}\)\,.

From the approximation property of Jensen polynomials and from
\eqref{JLiEFMPS} it follows that
\begin{multline}
\label{LiEFM} %
\mathcal{M}_{B^{n}}(t)\to\mathcal{M}_{B^{\infty}}(t),\ \ %
\mathcal{M}_{B^{n}\times{}0}(t)\to\mathcal{M}_{B^{\infty}\times{}0}(t),\ \ %
\mathcal{M}_{Q^{n}}(t)\to\mathcal{M}_{Q^{\infty}}(t),\\ %
\mathcal{M}_{Q^{n}\times{}0}(t)\to\mathcal{M}_{Q^{\infty}\times{}0}(t)\
\ \textup{as}\ \ n\to\infty\,.%
\end{multline}
This explains the notation \eqref{LiEFMPS}\,.\\

We summarize the above stated consideration as follows\\[-4.0ex]
\begin{thm}%
\label{MPaJP}%
Let \(\{V^n\}\) be one of the four families of convex sets:
\(\{B^n\}\), \(\{B^n\times{}0\}\),  \(\{Q^n\}\),
\(\{Q^n\times{}0\}\). For each of these four families, there
exists an entire function\,\footnote{The symbol
\(V^\infty\) denotes whichever of the sets \(\{B^\infty\}\), \(\{B^\infty\times{}0\}\),
\(\{Q^\infty\}\),
\(\{Q^\infty\times{}0\}\) the case demands.} %
\(\mathcal{M}_{V^{\infty}}\)  such that in \textsf{any}
dimension \(n\), the renormalized Steiner polynomials
\(\mathcal{M}_{V^{n}}\), defined  by  \eqref{RemOr}, are generated
by this entire function \(\mathcal{M}_{V^{\infty}}\) to be the Jensen
polynomials \(\mathscr{J}_n(\mathcal{M}_{V^{\infty}})\): for all \(n\) the
equalities \eqref{JLiEFMPS} hold.
\end{thm}%
%%%%%%%%%%%%%%%%%%%%%%%%%%%%%%%%%%%%%%%%%%%%%%%%%%%%%%%%%%%%%%%%%%%%%%%
\paragraph{Entire functions which generate the Weyl polynomials
for the surfaces of balls, cubes, spherical and cubic cylinders.}
Let us introduce the infinite power series:
\begin{subequations}
\label{ReWP}
\begin{align}%
\mathcal{W}_{\partial{}B^{\infty}}^{\,p}(t)&=
\sum\limits_{l=0}^{\infty}%
\frac{2^{-l}\,\Gamma(\frac{p}{2}+1)}{\Gamma(\frac{p}{2}+l+1)}
\cdot{}\frac{1}{l!}\cdot{}\Big(-\frac{t^2}{2}\Big)^l,\ \
p=1,\,2,\,\ldots\,;
\label{ReWP1}\\
\mathcal{W}_{\partial{}B^{\infty}}^{\,\infty}(t)&=
\sum\limits_{l=0}^{\infty}%
\frac{1}{l!}\cdot{}\Big(-\frac{t^2}{2}\Big)^l\,;
 \label{ReWP2}\\
\mathcal{W}_{\partial{}(B^{\infty}\times{}0)}^{\,p}(t)&=
\sum\limits_{l=0}^{\infty}%
\frac{2^{-l}\,\Gamma(\frac{p}{2}+1)}{\Gamma(\frac{p}{2}+l+1)}
\cdot{}\frac{\Gamma(1/2)}{\Gamma(l+1/2)}\cdot{}\Big(-\frac{t^2}{2}\Big)^l,\
\  p=1,\,2,\,\ldots\,;%
 \label{ReWP3}\\ %
\mathcal{W}_{\partial{}(B^{\infty}\times{}0)}^{\,\infty}(t)&=
\sum\limits_{l=0}^{\infty}%
\frac{\Gamma(1/2)}{\Gamma(l+1/2)}\cdot{}\Big(-\frac{t^2}{2}\Big)^l\,; %
\label{ReWP4} \\ %
%\end{align}
 %\begin{align}
\mathcal{W}_{\partial{}Q^{\infty}}^{\,p}(t)&=
\sum\limits_{l=0}^{\infty}%
 \frac{2^{-l}\,\Gamma(\frac{p}{2}+1)}{\Gamma(\frac{p}{2}+l+1)}
\cdot{}\frac{1}{(2l+1)!}\cdot{}\Big(-{\frac{{\pi}t^2}{2}}\,\Big)^l,\
\  p=1,\,2,\,\ldots\,;%
\label{ReWP5}\\ %
\mathcal{W}_{\partial{}Q^{\infty}}^{\,\infty}(t)&=
\sum\limits_{l=0}^{\infty}%
\frac{1}{(2l+1)!}\cdot{}\Big(-{\frac{{\pi}t^2}{2}}\,\Big)^l\, ;
\label{ReWP6}\\ %
\mathcal{W}_{\partial{}(Q^{\infty}\times{}0)}^{\,p}(t)&=
\sum\limits_{l=0}^{\infty}%
\frac{2^{-l}\,\Gamma(\frac{p}{2}+1)}{\Gamma(\frac{p}{2}+l+1)}
\cdot{}\frac{1}{(2l)!}\cdot{}\Big(-{\frac{{\pi}t^2}{2}}\,\Big)^l,\
\ p=1,\,2,\,\ldots\,;%
\label{ReWP7}\\ %
\mathcal{W}_{\partial{}(Q^{\infty}\times{}0)}^{\,\infty}(t)&=
\sum\limits_{l=0}^{\infty}%
\frac{1}{(2l)!}\cdot{}\Big(-{\frac{{\pi}t^2}{2}}\,\Big)^l\,.
\label{ReWP8}
\end{align}%
\end{subequations}
The series \eqref{ReWP} represent entire functions. The functions
\eqref{ReWP2} and \eqref{ReWP4} are of order \(2\) and normal
type, the functions \eqref{ReWP1}, \eqref{ReWP3}, \eqref{ReWP6}
and \eqref{ReWP8} are of order \(1\) and normal type, the
functions \eqref{ReWP5} and \eqref{ReWP7} are of order \(2/3\) and
normal type.

For each of the entire functions \eqref{ReWP} we consider the assotiated
sequence of Jensen polynomials:
\begin{subequations}
\label{JReWP}
\begin{alignat}{2}
 \mathcal{W}^p_{\partial{}B^{n+1}}(t)&{=}
 \mathscr{J}_{2[n/2]}(\mathcal{W}^p_{\partial{}B^{\infty}};t)\,,&\qquad&1\leq{}p\leq\infty\,;
\label{JReWP1}\\%
\mathcal{W}^p_{\partial{}(B^{n+1}\times{}0)}(t)&{=}
\mathscr{J}_{2[n/2]}(\mathcal{W}^p_{\partial{}(B^{\infty}\times{}0)};t)\,,&\qquad&1\leq{}p\leq\infty\,;
\label{JReWP2}\\%
\mathcal{W}^p_{\partial{}Q^{n}}(t)&{=}
\mathscr{J}_{2[n/2]}(\mathcal{W}^p_{\partial{}Q^{\infty}};t)\,,&\qquad&1\leq{}p\leq\infty\,;
\label{JReWP3}\\%
\mathcal{W}^p_{\partial{}(Q^{n}\times{}0)}(t)&{=}
\mathscr{J}_{2[n/2]}(\mathcal{W}^p_{\partial{}(Q^{\infty}\times{}0)};t)\,,&\qquad&1\leq{}p\leq\infty\,.
\label{JReWP4}
\end{alignat}
\end{subequations}
From the expressions \eqref{WPpS}, \eqref{WPiS}, \eqref{WPpSSqC},
\eqref{WPiSSqC},  \eqref{WPpQ}, \eqref{WPiQ}, \eqref{WPpQSqC},
\eqref{WPiQSqC},    for the Weyl polynomials it follows that they
are related to the above introduced polynomials \eqref{JReWP} in the following way:
\begin{subequations}
\label{WRemOr}
\begin{alignat}{2}
W^p_{\partial{}B^{n+1}}(t)&=\phantom{(\times{})}
\textup{Vol}_n(\partial{}B^{n+1})&\,\cdot\,&\mathcal{W}^p_{\partial{}B^{n+1}}(int)\,;
\label{WRemOr1}\\%
W^p_{\partial{}(B^{n}\times{}0)}(t)&=
\textup{Vol}_n(\partial{}(B^{n}\times{}0))
\,&\,\cdot\,&\mathcal{W}^p_{\partial{}(B^{n}\times{}0)}(int)\,;
\label{WRemOr2}\\%
W^p_{\partial{}Q^{n+1}}(t)&=\phantom{(\times{})}
\textup{Vol}_n(\partial{}Q^{n+1})&\,\cdot\,&\mathcal{W}^p_{\partial{}Q^{n+1}}(int)\,;
\label{WRemOr3}\\%
W^p_{\partial{}(Q^{n}\times{}0)}(t)&=
\textup{Vol}_n(\partial{}(Q^{n}\times{}0))
\,&\,\cdot\,&\mathcal{W}^p_{\partial{}(Q^{n}\times{}0)}(int)\,;
\label{WRemOr4}
\end{alignat}
\end{subequations}
The equalities  \eqref{WRemOr} hold for all
\(n:\,1\leq{}n<\infty,\,\,p:\,1\leq{}p\leq{}\infty\).

The polynomials \(\mathcal{W}^p_{\partial{}B^{n+1}}\),
\(\mathcal{W}^p_{\partial(B^n\times{}0)}\),
\(\mathcal{W}^p_{\partial{}Q^{n+1}}\),
\(\mathcal{W}^p_{\partial{}(Q^n\times{}0)}\) can be interpreted as
\textit{renormalized Weyl polynomials}. We take the equalities
\eqref{WRemOr} as the \textsf{definition} of the renormalized Weyl
polynomials \(\mathcal{W}^p_{\partial{}B^{n+1}}\),
\(\mathcal{W}^p_{\partial(B^n\times{}0)}\),
\(\mathcal{W}^p_{\partial{}Q^{n+1}}\),
\(\mathcal{W}^p_{\partial{}(Q^n\times{}0)}\) in terms of the
`original' Steiner polynomials \(W^p_{\partial{}B^{n+1}}\),
\(W^p_{\partial(B^n\times{}0)}\), \(W^p_{\partial{}Q^{n+1}}\),
\(W^p_{\partial{}(Q^n\times{}0)}\).

From the approximation property of Jensen polynomials and from
\eqref{JReWP} it follows that for every fixed
\(p,\,\,1\leq{}p\leq{}\infty\),
\begin{multline}
\label{LiEFMW} %
\mathcal{W}^p_{\partial{}B^{n+1}}(t)\to\mathcal{W}^p_{\partial{}B^{\infty}}(t),\ \ %
\mathcal{W}^p_{\partial{}(B^{n}\times{}0)}(t)
\to\mathcal{W}^p_{\partial{}(B^{\infty}\times{}0)}(t),\ \ \\ %
\mathcal{W}^p_{\partial{}Q^{n+1}}(t)\to\mathcal{W}^p_{\partial{}Q^{\infty}}(t), %
\mathcal{W}^p_{\partial{}(Q^{n}\times{}0)}(t)
\to\mathcal{W}^p_{\partial{}(Q^{\infty}\times{}0)}(t)\
\ \textup{as}\ \ n\to\infty\,.%
\end{multline}
This explains the notation \eqref{ReWP}\,.

We summarize the above as follows.
\begin{thm}%
\label{WPaJP}%
Let \(\{\mathscr{M}^n\}\) be one of the four families of
\(n\)-dimensional convex surfaces: \(\{\partial{}B^{n+1}\}\),
\(\{\partial(B^n\times{}0)\}\), \(\{\partial{}Q^{n+1}\}\),
\(\{\partial(Q^n\times{}0)\}\). For each of these four families,
and for each \(p,\,1\leq{}p\leq{}\infty\), there exists an
 entire function\,\footnote{The symbol
\(\mathscr{M}^\infty\) denotes whichever of the sets \(\{\partial{}B^\infty\}\),
\(\{\partial{}(B^\infty\times{}0)\}\), \(\{\partial{}Q^\infty\}\),
\(\{\partial{}(Q^\infty\times{}0)\}\) the case demands.} %
\(\mathcal{W}^p_{\,\mathscr{M}^{\infty}}\)  such that in
\textsf{any} dimension \(n\), the renormalized Weyl
polynomials \(\mathcal{W}^p_{\mathscr{M}^{n}}\), defined  by
\eqref{WRemOr}, are generated by this entire function
\(\mathcal{W}^p_{\mathscr{M}^{\infty}}\) to be the Jensen polynomials
\(\mathscr{J}_2[n/2](\mathcal{W}^p_{\mathscr{M}^{\infty}})\).
\end{thm}%
\section{Entire functions of the Hurwitz
 and of the Laguerre-P\'olya class.
Multipliers preserving location of roots.
\label{LPEF}}%
\paragraph{Hurwitz class of entire functions.}
\begin{defn}
\label{DeHuC}
An  entire function \(H\) is said to
be in the \textsf{Hurwitz class}, written
\(H\in\mathscr{H}\), if
\begin{enumerate}
\item
\(H\not\equiv{}0\), and the real part of any root
 of \(H\) is strictly negative, i.e. if \(H(\zeta)=0\),
then \(\textup{Re}\,\zeta<0\).
\item
The function \(H\) is of exponential type:
 \(\varlimsup\limits_{|z|\to\infty}\frac{\ln{}|H(z)|}{|z|}<\infty\),
 and its defect \(d_H\) is non-negative: \(d_H\geq{}0\), where
 \begin{equation}%
 \label{DefH}%
 2d_H=\textstyle{\varlimsup\limits_{r\to+\infty}\frac{\ln{}|H(r)|}{r}\,-
 \varlimsup\limits_{r\to+\infty}\frac{\ln{}|H(-r)|}{r}}\,.
 \end{equation}%
\end{enumerate}
\end{defn}

The following functions are examples of entire functions
in class \(\mathscr{H}\):\\
a). A dissipative polynomial \(P(t)\).\\ %
b). An exponential \(\exp\{\alpha{}t\}\), where
\(\textup{Re}\,\alpha\geq{}0\,\). \\ %
c). The product \(P(t)\cdot\exp\{\alpha{}t\}\): \(P(t)\) is a
dissipative polynomial, \(\textup{Re}\,\alpha\geq{}0\,\).\\

\noindent
\textsf{ The significance of the Hurwitz class of entire functions stems from
the fact that functions in this class\,%
\footnote{The full description of the class of entire functions which are the
 limits of dissipative polynomials can be found in \cite[Chapter VIII, Theorem 4]{Lev1}.
 This class (up to the change of variables \(z\to{}iz\))
 is denoted by \(P^{\ast}\) there.} %
are locally uniform limits in \(\mathbb{C}\) of
dissipative polynomials.}
\vspace{2.5ex}
%%%%%%%%%%%%%%%%%%%%%%%%%%%%%%%%%%%%%%%%%%%%%%%%%%%%%%%%%%%%%%%%%%%%%%%%%%%%%%%%
\paragraph{Laguerre-P\'olya class of entire functions.}
\begin{defn}
\label{DFLPC} An  entire function \(E\) is said to
be in the \textsf{Laguerre-P{\'o}lya class}, written
\(E\in\mathscr{L}\text{-}\mathscr{P}\), if \(E\) is real
and
 can be expressed in the form
\begin{equation}
\label{DLP2}%
E(t)=ct^ne^{-{\beta}t^2+{\alpha}t}\prod\limits_{k=1}^{\infty}
\left(1+t\alpha_k\right)e^{-t\alpha_k},
\end{equation}
where
\(c\in\mathbb{R}\setminus{}0,\,\beta\geq{}0,\,\alpha\in\mathbb{R},\,
\alpha_k\in\mathbb{R}\), \(n\) is non-negative integer, and
\(\sum_{k=1}^{\infty}\alpha_k^{2}<\infty\).

A function \(E\in\mathscr{L}\text{-}\mathscr{P}\) is  said to be
of type I, denoted \(E\in\mathscr{L}\text{-}\mathscr{P}\text{-}\textup{I}\)\,,
if it is expressible in the form:
\begin{equation}
\label{DLP23}%
E(t)=ct^ne^{{\alpha}t}\prod\limits_{k=1}^{\infty}
\left(1+t\alpha_k\right),
\end{equation}
where \(c\in\mathbb{R}\setminus{}0,\,\,\alpha\geq{}0,\,
\alpha_k\geq{}0\), \(n\) is non-negative integer, and
\(\sum_{k=1}^{\infty}\alpha_k<\infty\).\\
\end{defn}

\noindent
\textsf{ The significance of the Laguerre-P\'olya class stems from
the fact that functions in this class, \textit{and only these}, are
the locally uniform limits in \(\mathbb{C}\) of polynomials with
only real roots.} (See \cite[Chapter 8]{Lev1},  \cite[Chapter
II, Theorems 9.1,\,9.2,\,9.3]{Obr}.)

\begin{lem}
\label{HLPI} An entire function \(E\), which is in
 the Laguerre-P\'olya class of type \textup{I}
 also belongs to the Hurwitz class:
\[\mathscr{L}\text{-}\mathscr{P}\text{-}\textup{I}\,\subset\,\mathscr{H}\,.\]
\end{lem}
\begin{proof} The roots of the entire function \(E\) which admit the
representation \eqref{DLP23} are located at the points
\(-(\alpha_k)^{-1}\), thus they are strictly negative. From the
properties of the infinite product
\(\prod\limits_{k=1}^{\infty}\left(1+t\alpha_k\right)\)  %
with
 \(\sum_{k=1}^{\infty}|\alpha_k|<\infty\), it follows that
a function \(E\) which admits the representation \eqref{DLP23} is
of exponential type \(\alpha\), and
\(\varlimsup\limits_{r\to+\infty}\frac{\ln{}|H(\pm{}r)|}{r}=\pm{}\alpha\).
Thus, the defect \(d_H=\alpha\geq{}0\) since \(\alpha\geq{}0\).
\end{proof}

\paragraph{Multipliers preserving the reality of roots.}
\begin{defn}%
\label{PReZ}%
A sequence \(\{\gamma_k\}_{0\leq{}k<\infty}\) of real numbers is a
\textsf{P-S multiplier sequence}%
\footnote{P-S stands for P\'olya-Schur} %
 if for every polynomial \(f\):
\[f(t)=\sum\limits_{0\leq{}k\leq{}n}a_kt^k\]
with only real roots, the polynomial
\[h(t)=\sum\limits_{0\leq{}k\leq{}n}\gamma_ka_kt^k\]
has only real roots as well. (The degree \(n\) of the polynomial \(f\)
can be arbitrary.)
\end{defn}%
\begin{nthm} [P\'olya, Schur]
A sequence  \(\{\gamma_k\}_{0\leq{}k<\infty}\) of real numbers, \(\gamma_k\not\equiv{}0\),
 is a\\ P-S multiplier sequence if and only
if the power series
\[\Psi(t)=\sum\limits_{0\leq{}k\leq{}\infty}\frac{\gamma_k}{k!}\,t^k\]
represents an entire function, and either the function \(\Psi(t)\)
or the function \(\Psi(-t)\) is in the Lagierre-P\'olya class of
type \textup{I}.
\end{nthm}
\noindent
This result was obtained in \cite{PoSch}.
\begin{thm}{\textup{\textsf{[Jensen-Craven-Csordas-Williamson.]}}}
\label{JSW}
Let \(E(t)\) be an entire function belonging to the
Laguerre-P\'olya class \(\mathscr{L}\text{-}\mathscr{P}\), and
\(\{\mathscr{J}_n(E,\,t)\}_{n=1,\,2,\,3,\,\ldots}\) be
the sequence of the Jensen polynomials associated with the
function \(E\). \textup{(Definition \ref{DefJP}.)}
\begin{enumerate}
\item
Then, for each \(n\), all roots of the polynomial
\(\mathscr{J}_n(E,\,t)\) are  real.\\
\item
If \(E(t)\) belongs to the subclass \(\mathscr{L}\text{-}\mathscr{P}\textup{-I}\)
of the Laguerre-P\'olya class \(\mathscr{L}\text{-}\mathscr{P}\), then
for each \(n\), all roots of the polynomial
\(\mathscr{J}_n(E,\,t)\) are  negative.\\
\item
If, moreover, \(E(t)\) is not of the form
\(E(t)=p(t)\,e^{\beta{}t}\), where \(p(t)\) is a polynomial, then
for each \(n\), all roots of the polynomial
\(\mathscr{J}_n(E,\,t)\) are simple\,.
\end{enumerate}
\end{thm}
Statement 1 of the theorem was proved by Jensen\,\footnote{Though
Jensen himself did not introduce explicitly the polynomials which
are called `the Jensen polynomials' now.}, \cite{Jen}. It is a
special case of Theorem by G.P\'olya and I.\,Schur obtaining by
choosing \(\Psi(t)=\Big(1+\dfrac{t}{n}\Big)^n\). Statement 3 is a
refinement of Statement 1 due to G.Csordas and J.\,Williamson in
\cite{CsWi}, who also offer an alternative proof of Statement 1.
In \cite{CsWi}, the main Theorem  is formulated on p.\,263. It
appears as the Statement 3 of Theorem \ref{JSW} of the present
paper.  In \cite{CsWi}, this theorem was inaccurately formulated.
A corrected version can be found in \cite[Section 4.1]{CrCs3}.
\begin{thm}
\label{JDM} Let \(H\) be an entire function belonging to the
Hurwitz class \(\mathscr{H}\), and
\(\{\mathscr{J}_n(H,\,t)\}_{n=1,\,2,\,3,\,\ldots}\) be the
sequence of the Jensen polynomials associated with the function
\(H\). \textup{(Definition \ref{DefJP}.)} Then, for each \(n\),
the polynomial \(\mathscr{J}_n(H,\,t)\) is dissipative.
\end{thm}
Theorem \ref{JDM} can be obtained as a consequence of Theorem
\ref{JSW} and the Hermite-Biehler Theorem. A proof of Theorem \ref{JDM}
will be presented in Section \ref{LPEF}.

\vspace{1.0ex}
\paragraph{Laguerre multipliers.}
\begin{nthm}[Laguerre]
Let an entire function \(E(t)\),
\begin{equation}
\label{EE}%
E(t)=\sum\limits_{0\leq{}l<\omega}\varepsilon_lt^l,\quad
\omega\leq{}\infty,
\end{equation}
be in the Laguerre-P\'olya class:
\(E\in\mathscr{L}\text{-}\mathscr{P}\). Furthermore, let \(\psi\) be an entire
function  be in the Laguerre-P\'olya class
\(\mathscr{L}\text{-}\mathscr{P}\), such that all of its roots are negative.
\begin{enumerate}
\item
Then the power series
\begin{equation}
\label{PSPs}%
E_{\psi}=\sum\limits_{0\leq{}l<\omega}\varepsilon_l\psi(l)t^l
\end{equation}
converges for every \(t\) and its sum is an entire function of
the Laguerre-P\'olya class:
\(E_{\psi}\in\mathscr{L}\text{-}\mathscr{P}\).
\item
If moreover \(E(t)\) is of type \textup{I}:
\(E\in\mathscr{L}\text{-}\mathscr{P}\textup{-I}\), then the the
sum of power series \eqref{PSPs} is also an entire function of
type \textup{I}:
\(E_{\psi}\in\mathscr{L}\text{-}\mathscr{P}\textup{-I}\).
\end{enumerate}
\end{nthm}
This theorem appeared in \cite[section 18,~p.117]{Lag1},
 or \cite[p.~202]{Lag}. Laguerre himself formulated
this theorem for a polynomial with
real roots. The extended formulation, where \(\E\) is a general
entire function from the class \(\mathscr{L}\text{-}\mathscr{P}\),
can be found in the paper \cite[p.\,112]{PoSch}, or in its reprint
in \cite[p.123]{Po}. In \cite{PoSch} the extended formulation is
attributed to Jensen, see \cite{Jen}.

The above mentioned results of P\'olya, Schur,
Laguerre, and Jensen, as well as of many related results, can be found
in \cite[Chapter II]{Obr}, \cite[Chapter VIII]{Lev1},
\cite[Chapter 5, especially Sections 5.5, 5.6,  5.7]{RaSc}, and in
numerous papers of Th.\,Craven and G.\,Csordas (See for example
\cite{CrCs1}). See also \cite[Part five]{PoSz}. The book by L.\,de
Branges \cite{deBr} is also closely related to this group of problems.

\section{Properties of entire functions
generating Steiner and Weyl polynomials
of `regular' convex sets and their surfaces. \label{PEFG}}%
\paragraph{Entire functions generating the Steiner polynomials.}
\begin{thm}\ \ %
\label{EFGWP}%
The entire functions \eqref{LiEFMPS} generating the renormalised
Steiner polynomials of balls, cubes, squeezed spherical and
cubic cylinders, possesses the following properties:
\begin{enumerate}
\item
The function \(\mathcal{M}_{B^{\infty}}\) is of type \textup{I} of
the Laguerre-P\'olya class.
\item
The function \(\mathcal{M}_{B^{\infty}\times{}0}\) belongs to the
Hurwitz class. It has infinitely roots, and all but finitely many of its
roots are non-real;
\item
The function \(\mathcal{M}_{Q^{\infty}}\) is of type \textup{I} of
the Laguerre-P\'olya class;
\item
The function \(\mathcal{M}_{Q^{\infty}\times{}0}\) is of type
\textup{I} of the Laguerre-P\'olya class\,.
\end{enumerate}
\end{thm}%
\begin{lem}
\label{PrGa} The function \(\dfrac{1}{\Gamma(t+1)}\), where
\(\Gamma\) is the Euler Gamma function, is in the Laguerre-P\'olya
class and all its roots are negative.
\end{lem}
Indeed,
\[\frac{1}{\Gamma(t+1)}=
e^{Ct}\prod\limits_{1\leq{}k<\infty}\left(1+\frac{t}{k}\right)e^{-\frac{t}{k}}\,,\]
(\(C\) is the Euler constant,
\(C\approx{}0.5772156\ldots\,\,\,.\))%

\begin{proof}[Proof of Theorem \ref{EFGWP}] Statement 1 is evident:
\(\mathcal{M}_{B^{\infty}}(t)=e^t\). \\%
\hspace*{3.0ex}To obtain Statement 3, we remark that the function
\(\mathcal{M}_{Q^{\infty}}\) is of the form \(E_{\psi}\),
\eqref{PSPs}, where \(E(t)=\exp\{\frac{\sqrt{\pi}}{2}t\}\), and
\(\psi(t)=\dfrac{1}{\Gamma(\frac{t}{2}+1)}\). Then we apply the
Laguerre theorem on multipliers to these \(E\) and \(\psi\). The
needed property of \(\psi\) is formulated as Lemma \ref{PrGa}.\\ %
\hspace*{3.0ex}Statement 4 can be obtained in the same way
as the statement 3. One need only take
\(E(t)=\exp\{\frac{\sqrt{\pi}}{2}t\}\), and
\(\psi(t)=\dfrac{\Gamma(\frac{1}{2})}{\Gamma(\frac{t+1}{2}+1)}\).\\ %
\hspace*{3.0ex}Proof of Statement 2 is more complicated. From
\eqref{LiEFMPS2} it follows that
\[\mathcal{M}_{B^{\infty}\times{}0}(t)=
\sum\limits_{0\leq{}k<\infty}B({\textstyle{\frac{k}{2}+1,\frac{1}{2}}})\dfrac{1}{k!}\,t^k
=\sum\limits_{0\leq{}k<\infty}
\int\limits_0^1{\xi}^{\frac{k}{2}}(1-\xi)^{-\frac{1}{2}}\,d\xi\,%
\dfrac{1}{k!}\,t^k\,.\] Changing the order of summation and
integration and summing the exponential series, we obtain the
integral representation:
\begin{equation}
\label{IntRep}%
\mathcal{M}_{B^{\infty}\times{}0}(t)=
 2\int\limits_0^1{}(1-\xi^2)^{-\frac{1}{2}}\xi{}e^{\xi{}t}d\xi\,.
\end{equation}
The fact that the function \(\mathcal{M}_{B^{\infty}\times{}0}\)
belongs to the Hurwitz class will be derived from the integral
representation \eqref{IntRep}. This will be done in %
Section~\ref{LocRoot}.%
\end{proof}

\paragraph{Entire functions generating the Weyl polynomials.}
\begin{lem}%
\label{PLaM}%
 Let
\begin{equation}
\label{EEF}%
E(t)=\sum\limits_{0\leq{}l<\infty}a_l{}t^{2l}
\end{equation}
 be an even entire function of the class
\(\mathscr{L}\text{-}\mathscr{P}\), and let \(p>0\) be a number.
Then the function \(E_p(t)\) defined by the power series
\begin{equation}
\label{EEFs}%
E_p(t)\stackrel{\textup{\tiny
def}}{=}\sum\limits_{1\leq{}l<\infty}\frac{2^{-l}\Gamma\big(\frac{p}{2}+1\big)}
{\Gamma\big(l+\frac{p}{2}+1\big)}\cdot{}a_lt^{2l},
\end{equation}
belongs to the class \(\mathscr{L}\text{-}\mathscr{P}\) as well.
\end{lem}
\begin{proof} %
Lemma \ref{PLaM} is the consequence of the Laguerre theorem on
multipliers. The function
\begin{equation}%
\label{LaMup}
\psi_p(t)=\frac{2^{-\frac{t}{2}}\Gamma\big(\frac{p}{2}+1\big)}
{\Gamma\big(\frac{t}{2}+\frac{p}{2}+1\big)}
\end{equation}%
is in the Laguerre-P\'olya class (see Lemma \ref{PrGa}), and its
roots are negative.
\end{proof}

\vspace{2.0ex}%
We point out, see \eqref{ReWP}, that the entire functions
\(\mathcal{W}_{\partial{}B^{\infty}}^{\,p}\),
\(\mathcal{W}_{\partial{}(B^{\infty}\times{}0)}^{\,p}\),
\(\mathcal{W}_{\partial{}Q^{\infty}}^{\,p}\),
\(\mathcal{W}_{\partial{}(Q^{\infty}\times{}0)}^{\,p}\), which
generate the Weyl polynomials with finite index \(p\) for the
appropriate families of convex surfaces, can be obtained from the
entire functions
\(\mathcal{W}_{\partial{}B^{\infty}}^{\,\infty}\),
\(\mathcal{W}_{\partial{}(B^{\infty}\times{}0)}^{\,\infty}\),
\(\mathcal{W}_{\partial{}Q^{\infty}}^{\,\infty}\),
\(\mathcal{W}_{\partial{}(Q^{\infty}\times{}0)}^{\,\infty}\),
which generate the Weyl polynomials with infinite index, by
means of a transformation of the form
\[\sum\limits_{0\leq{}k<\infty}a_kt^k\to\,
\sum\limits_{0\leq{}k<\infty}\psi_p(k)\,a_kt^k\,.\]
\begin{thm}
\label{BLPC}%
\ \ %
\begin{enumerate}
\item %
 The functions \(\mathcal{W}_{\partial{}B^{\infty}}^{\,\infty}\),
\(\mathcal{W}_{\partial{}Q^{\infty}}^{\,\infty}\),
\(\mathcal{W}_{\partial{}(Q^{\infty}\times{}0)}^{\,\infty}\)
belong to the Laguerre-P\'olya class
\(\mathscr{L}\text{-}\mathscr{P}\).
\item
The function
\(\mathcal{W}_{\partial{}(B^{\infty}\times{}0)}^{\,\infty}\) does
not belong to the Laguerre-P\'olya class
\(\mathscr{L}\text{-}\mathscr{P}\): this function has infinitely
many non-real roots.
\end{enumerate}
\end{thm}
\noindent%
\begin{proof} Statement 1 is evident in view of the explicit
expressions:
\begin{align}%
\label{EE1}%
\mathcal{W}_{\partial{}B^{\infty}}^{\,\infty}(t)&=\exp\{-t^2/2\}\,,\\
\label{EE2}%
 \mathcal{W}_{\partial{}Q^{\infty}}^{\,\infty}\,(t)&=
\frac{\sin\{(\pi/2)^{\frac{1}{2}}{}t\}}{(\pi/2)^{\frac{1}{2}}{}t}\,,\\
\label{EE3}%
 \mathcal{W}_{\partial{}(Q^{\infty}
\times{}0)}^{\,\infty}&=\cos\{(\pi/2)^{\frac{1}{2}}{}\,t\}\,.
\end{align}%
The function \(\mathcal{W}_{\partial{}(B^{\infty}
\times{}0)}^{\,\infty}\), which appears in \textup{Statement 2},
can not be expressed in terms of `elementary' functions, but it
can be expressed in terms of the Mittag-Leffler function
\(\mathscr{E}_{1,\,\frac{1}{2}}\):
\begin{equation}
\label{EE4}%
 \mathcal{W}_{\partial{}(B^{\infty}
\times{}0)}^{\,\infty}(t)=\sqrt{\pi}\mathscr{E}_{1,\,\frac{1}{2}}\left(-\frac{t^2}{2}\right),
\end{equation}
where
\begin{equation}%
\label{GML}%
\mathscr{E}_{\alpha,\beta}(z)=\sum\limits_{0\leq{}k<\infty}\frac{z^k}
{\Gamma(\alpha{}k+\beta)}\,.
\end{equation}
The integral representation
\begin{equation}%
\label{GMLI}%
\sqrt{\pi}\mathscr{E}_{1,\,\frac{1}{2}}(t)=1+
t\int\limits_0^1(1-\xi)^{-\frac{1}{2}}e^{t\xi}\,d\xi\,.
\end{equation}
can be derived from \eqref{GML}.
The integral representation \eqref{GMLI} can be derived from the
Taylor series \eqref{GML} in the same way as the integral
representation \eqref{IntRep} was derived from the Taylor series
\eqref{LiEFMPS2}. From \eqref{GMLI} the following asymptotic relations can
be obtained:
\begin{equation}%
\label{AsML}%
\sqrt{\pi}\mathscr{E}_{1,\,\frac{1}{2}}(t)=
\begin{cases}%
\hspace*{4.0ex}\frac{1}{2t}\,(1+o(1)),& t\to{}-\infty,\\[1.0ex]
\sqrt{\pi{}t}\, e^t\,(1+o(1)), &t\to{}+\infty.\,\\[1.0ex]
\hspace*{6.0ex}O(|t|),&t\to{}\pm{}i\infty.
\end{cases}%
\end{equation}%
From \eqref{AsML} it follows that the indicator diagram of the
entire function \(\mathscr{E}_{1,\,\frac{1}{2}}(t)\) of the
exponential type is the interval \([0,\,1]\,.\) Moreover, the
function \(\mathscr{E}_{1,\,\frac{1}{2}}(it)\) belongs to the
class \(C\),  as this class was defined in \cite[Lecture 17]{Lev2}.
It follows from the Cartwright-Levinson Theorem,
which appears as Theorem 1 of the Lecture 17 in \cite{Lev2}),
that the function \(\mathscr{E}_{1,\,\frac{1}{2}}(t)\) has
infinitely many roots.
These roots have a positive density, and are located `near' the
rays \(\arg t=\frac{\pi}{2}\) and \(\arg t=-\frac{\pi}{2}\). From
this and from \eqref{EE4} it follows that the roots of the
function \( \mathcal{W}_{\partial{}(B^{\infty}
\times{}0)}^{\,\infty}(t)\) are located near four rays
\(\arg{}t=\frac{\pi}{4},\,\arg{}t=
\frac{3\pi}{4},\,\arg{}t=\frac{5\pi}{4},\,\arg{}t=\frac{7\pi}{4}\,.\)
In particular, infinitely many of the roots of the function \(
\mathcal{W}_{\partial{}(B^{\infty} \times{}0)}^{\,\infty}(t)\) are
non-real.
\end{proof}
 \begin{rem}
Much more precise results about the Mittag-Leffler function
\(\mathscr{E}_{\alpha,\beta}\) and the distribution of its roots are
available. See, for example, \cite[sec\-tion 18.1]{EMOT},
or \cite{Djr}.
 \end{rem}
\begin{thm} \ \
\label{ImCo}%
\begin{enumerate}
\item
For every \(p=1,\,2,\,\ldots\), the functions
\(\mathcal{W}_{\partial{}B^{\infty}}^{\,p}\),
\(\mathcal{W}_{\partial{}Q^{\infty}}^{\,p}\),
\(\mathcal{W}_{\partial{}(Q^{\infty}\times{}0)}^{\,p}\) belong to
the Laguerre-P\'olya class \(\mathscr{L}\text{-}\mathscr{P}\).
\item
If \(p\) is large enough, then the function
\(\mathcal{W}_{\partial{}(B^{\infty}\times{}0)}^{\,p}\) does not
belong to the Laguerre-P\'olya class
\(\mathscr{L}\text{-}\mathscr{P}\): it has non-real roots.
\end{enumerate}
\end{thm}
\begin{proof} Statement 1 of this theorem is a consequence of the
statement 1 of Theorem \ref{BLPC} and Lemma \ref{PLaM}. The
statement 2 of this theorem is a consequence of statement 2 from
Theorem \ref{BLPC} and the approximation property \eqref{LiRe}.
\end{proof}
\begin{rem}%
\label{RefToBes}%
The fact that the function
\(\mathcal{W}_{\partial{}B^{\infty}}^{\,p}\) belongs to the
Laguerre-P\'olya class \(\mathscr{L}\text{-}\mathscr{P}\), i.e., the fact that
all its roots are real, can be established without reference to
\textup{Lemma \ref{PLaM}}. The function
\(\mathcal{W}_{\partial{}B^{\infty}}^{\,p}\) can be expressed in
terms of Bessel functions \(J_{\nu}\). Recall that for arbitrary
\(\nu\),
\begin{equation}%
\label{Bess}%
J_{\nu}(t)=\left(\frac{t}{2}\right)^{\nu}\sum\limits_{0\leq{}l<\infty}%
\frac{(-1)^l(t^2/4)^l}{l!\,\Gamma(\nu+l+1)}\,.
\end{equation}%
Comparing \eqref{Bess} with \eqref{ReWP1}, we see that
\begin{equation}%
\label{CompBess}%
\mathcal{W}_{\partial{}B^{\infty}}^{\,p}(t)=\Gamma\left(\frac{p}{2}+1\right)
\left(\frac{t}{2}\right)^{-\frac{p}{2}}J_{\frac{p}{2}}(t)\,.
\end{equation}%
In particular, \\ %
 for\,%
\footnote{ Deriving \eqref{ExC1} from \eqref{CompBess}, we used
the formula
\(J_{\frac{1}{2}}(t)=\big(\frac{2}{\pi{}t}\big)^{\frac{1}{2}}\sin{}t\,.\)
\textup{(}Concerning this formula, see, for example,
\textup{\cite[section \textsf{17.24}]{WhWa}.}\textup{)} However,
\eqref{ExC1} may be obtained directly from \eqref{ReWP1}.}
\(p=1\),
\begin{equation}%
\label{ExC1}%
\mathcal{W}^{\,1}_{\partial{}B^{\infty}}(t)=\frac{\sin{}t}{t},
\end{equation}
for
 \(p=2\),
\begin{equation}%
\label{ExC2}%
\mathcal{W}^{\,2}_{\partial{}B^{\infty}}(t)=2\,\frac{J_1(t)}{t}.
\end{equation}
It is known that for every \(\nu>-1\), all roots of the Bessel
function \(J_{\nu}(t)\) are real (This result is due to A.Hurwitz.
See, for example, \textup{\cite[Chapter XV, Section 15.27]{Wat}}.)
\end{rem}%
The statement 2 of Theorem \ref{ImCo} can be further refined.
\begin{thm}\ \
\label{str}
\begin{enumerate}
\item
For \(p=1,\,2,\,4\), the function
\(\mathcal{W}_{\partial{}(B^{\infty}\times{}0)}^{\,p}\)
 belongs to the Laguerre-P\'olya class
\(\mathscr{L}\text{-}\mathscr{P}\)\,;
\item
For \(p:\,5\leq{}p\leq{}\infty\), the function
\(\mathcal{W}_{\partial{}(B^{\infty}\times{}0)}^{\,p}\) does not
 belong to the Laguerre-P\'olya class
\(\mathscr{L}\text{-}\mathscr{P}\)\,: it has infinitely many
non-real roots\,.
\end{enumerate}
\end{thm}
\begin{proof} For every \(p\geq{}1\), the function
\(\mathcal{W}_{\partial{}(B^{\infty}\times{}0)}^{\,p}\) admits
 the integral representation
\begin{equation}
\label{IRWGF}
\mathcal{W}_{\partial{}(B^{\infty}\times{}0)}^{\,p}(t)= %
p\int\limits_{0}^{1}(1-\xi^2)^{\frac{p}{2}-1}\xi\cos{}t\xi\,d\xi\,.
\end{equation}
This integral representation can be obtained from \eqref{ReWP3} in
the same way that the integral representation \eqref{IntRep} was
obtained from \eqref{LiEFMPS2}. Using the identity
\[\Gamma(l+1/2)\,\Gamma(l+1)=\Gamma(1/2)\,2^{-2l}\,\Gamma(2l+1),\]
in equation  \eqref{ReWP3}, we obtain:
\[\mathcal{W}_{\partial{}(B^{\infty}\times{}0)}^{\,p}(t)=
\frac{p}{2}\sum\limits_{0\leq{}l<\infty}\text{B}(l+1,p/2)(-1)^l\frac{t^{2l}}{(2l)!}\,.\]
We use then the integral representation for the beta-function,
change the order of summation and integration and sum the
series using the Taylor expansion for \(\cos z\).
 For every \(p: \,1\leq{}p<\infty\), the function
\(\mathcal{W}_{\partial{}(B^{\infty}\times{}0)}^{\,p}\) can be
calculated asymptotically. This calculation can be done using the
integral representation \eqref{IRWGF}. The
asymptotic expression for the function
\(\mathcal{W}_{\partial{}(B^{\infty}\times{}0)}^{\,p}\) is
presented in Section \ref{LocRoot}, see \eqref{AsyW},
\eqref{EstRema}. From this expression
it follows that:\\
1. For \(p>4\), infinitely many (more specifically, all but finitely many)
roots of the
\(\mathcal{W}_{\partial{}(B^{\infty}\times{}0)}^{\,p}\) are
non-real. This is sufficient for the negative result of the
statement 2 of Theorem \ref{str} to be obtained.
\\
2. For \(p\leq{}4\), all but finitely many roots of the function
\(\mathcal{W}_{\partial{}(B^{\infty}\times{}0)}^{\,p}\) are real
and simple. This alone is not sufficient to show
Statement 1.  For \(p=2\) and \(p=4\), the function
\(\mathcal{W}_{\partial{}(B^{\infty}\times{}0)}^{\,p}\) can be
calculated explicitly. The case \(p=3\) remains open. A proof of the
fact that for \(p=1,\,2,\,4\) all roots of the function
\(\mathcal{W}_{\partial{}(B^{\infty}\times{}0)}^{\,p}\) are real
will be presented in Section \ref{LocRoot}. See Lemma \ref{WPBLP}.
\end{proof}
\noindent%
\begin{proof}[Proof of Theorem \ref{LoWP}] According to Theorems
\ref{BLPC}, \ref{ImCo} and \ref{str} (we here refer to the first statement
of each of these theorems), each of the functions
\(\mathcal{W}_{\partial{}B^{\infty}}^{\,p}\),
\(\mathcal{W}_{\partial{}Q^{\infty}}^{\,p}\),
\(\mathcal{W}_{\partial{}(Q^{\infty}\times{}0)}^{\,p}\) with
\(p:\,1\leq{}p\leq{}\infty\), and
\(\mathcal{W}_{\partial{}(B^{\infty}\times{}0)}^{\,p}\) with
\(p=1,\,2,\,4\), belongs to the  Laguerre-P\'olya class
\(\mathscr{L}\text{-}\mathscr{P}\). By  Theorem \ref{JSW},
the Jensen polynomials associated with
each of these entire functions, have only simple real roots.
According to Theorem \ref{WPaJP}, the renormalized Weyl
polynomials \(\mathcal{W}^p_{\partial{}B^{n+1}}\),
\(\mathcal{W}^p_{\partial{}Q^{n+1}}\),
\(\mathcal{W}^p_{\partial{}(Q^{n}\times{}0)}\) with
\(p:\,1\leq{}p\leq{}\infty\), and
\(\mathcal{W}^p_{\partial{}(B^{n}\times{}0)}\) with
\(p=1,\,2,\,4\) have only simple real roots. Because of the
renormalizing relations \eqref{WRemOr}, the Weyl polynomials
\(W^p_{\partial{}B^{n+1}}\), \(W^p_{\partial{}Q^{n+1}}\),
\(W^p_{\partial{}(Q^{n}\times{}0)}\) are conservative.
 \end{proof}

\begin{proof}[Proof of Theorem \ref{NRWPSq}] According to Theorem
\ref{str}, Statement 2, for \(p:\,5\leq{}p\leq{}\infty\), each of
the entire functions
\(\mathcal{W}_{\partial{}(B^{\infty}\times{}0)}^{\,p}\) has
infinitely many non-real roots. Since for fixed \(p\),
\(\mathscr{J}_n(\mathcal{W}_{\partial{}(B^{\infty}\times{}0)}^{\,p};t)
\to\mathcal{W}_{\partial{}(B^{\infty}\times{}0)}^{\,p}(t)\)
locally uniformly in \(\mathbb{C}\) as \(n\to\infty\), the Hurwitz
Theorem yields that every polynomial
\(\mathscr{J}_n(\mathcal{W}_{\partial{}(B^{\infty}\times{}0)}^{\,p})\)
with \(p,\,n:\,\,p\geq{}5,\,\,n\geq{}N(p)\) has non-real roots.  By
Theorem \ref{WPaJP}, \(\mathscr{J}_n(\mathcal{W}_{\partial{}(B^{\infty}\times{}0)}^{\,p})=
\mathcal{W}^p_{\partial{}(B^{n}\times{}0)}\).
Thus, the polynomial
\(\mathcal{W}^p_{\partial{}(B^{n}\times{}0)}\) has non-real
roots. Because of \eqref{JReWP}, the polynomial
\(W^p_{\partial{}(B^{n}\times{}0)}\) has roots
which do not belong to the imaginary axis. \end{proof}
\begin{proof}[Proof of Lemma \ref{LWPo}] This lemma is a consequence of
lemma \ref{PLaM}. If the polynomial
\(W_{\mathscr{M}}^{\,\infty}(t)\) is conservative, then the
polynomial \(E(t)=W_{\mathscr{M}}^{\,\infty}(it)\) is a real
polynomial with only real simple roots. The function
\(E_p(t)=W_{\mathscr{M}}^{\,p}(it)\) is related to \(E(t)=W_{\mathscr{M}}^{\,\infty}(it)\)
as well as the function \(E_p(t)\) from \eqref{EEFs} is
related to \(E(t)\) from \eqref{EEF}. By Lemma \ref{PLaM}, all
roots of \(E_p\) are real. Let us show that the roots are simple.
Consider the function \(E(t)+\varepsilon\), were \(\varepsilon\)
is a small real number, positive or negative. Since all roots of
the polynomial \(E(t)\) are real and simple, all roots of the
polynomial \(E(t)+\varepsilon\) are real if \(|\varepsilon|\) is
small enough. By Lemma \ref{PLaM}, all roots of the polynomial
\(E_p(t)+\varepsilon\) are real. However if the polynomial \(E_p(t)\)
has a multiple root, by applying the perturbation
\(E_p(t)\to{}E_p(t)+\varepsilon\) with an appropriate choice of sign
for \(\varepsilon\), this root splits into  simple
roots and some of these roots
 will be non-real.\end{proof}
\begin{rem}%
\label{SiSpC}%
We apply \textup{Lemma \ref{LWPo}} in the special cases
\(n=2,\,3,\,4,\,5\) only. In these cases the Lemma is quite
elementary. Only the cases \(n=4\) and \(n=5\) deserve some
attention. The cases \(n=2\) and \(n=3\) are trivial.
The cases \(n=4\) and \(n=5\) are reduced to the following
 elementary statement:\\[1.0ex]
\hspace*{3.0ex}\begin{minipage}[t]{0.9\linewidth} \textsl{Let
\(w_0,\,w_2,\,w_4\) be positive numbers. Assume that the roots of
the polynomial \(Q(t)=w_0+w_2t+w_4t^2\) are negative and
different. Then for every \(p>0\), the roots of the polynomial
\[Q^p(t)=w_0+\frac{w_2}{(p+2)}t+\frac{w_4}{(p+2)(p+4)}t^2\]
are also negative and different.}
\end{minipage}
\end{rem}%

Indeed, the conditions posed on polynomials \(Q\) and \(Q^p\) are
equivalent to the inequalities
\[w_2^2>w_0w_4\quad\textup{and}\quad{}\bigg(\frac{w_2}{p+2}\bigg)^2>w_0\frac{w_2}{(p+2)(p+4)}\,. \]
It is evident that the first of these inequalities implies the
second.

\section{The Hermite-Biehler Theorem and its application.
\label{HBieT}}%
In its traditional form, Hermite-Biehler Theorem gives conditions under
which all roots of a polynomial belong to the upper half plane
\(\{z:\,\textup{Im}\,z>0\}\). We need the version of this theorem
adopted to the left half plane, and for the case of polynomials
with non-negative coefficients only. Before we present out
reformulation of the Hermite-Biehler Theorem, we give several
definitions:
\begin{defn} %
\label{DefInlac}
Let \(S_1\) and \(S_2\) be two sets which are situated on the same straight
line\footnote{In our considerations the straight line \(L\)
 will be either the real axis
or the imaginary axis.}
 \(L\) of the complex plane:
\(S_1\subset{}L,\ S_2\subset{}L\,\), such that each of the sets
\(S_1,S_2\) consists only of isolated points. The sets \(S_1\) and
\(S_2\) \textsf{interlace} if between every two points of \(S_1\)
there is a point of \(S_2\), and between every two points of
\(S_2\) there is a point of \(S_1\).
\end{defn}
\begin{defn} %
\label{DefRIPa} %
Let \(P\) be a power series:
\begin{equation}
\label{CoPo} P(t)=\sum\limits_{0\leq{}k}p_kt^k,
\end{equation}
where \(t\) is a complex variable and the coefficients \(p_k\)
are  complex numbers.

 The \textsf{real part} \ \({}^{\mathscr{R}\!}P\)
and the
\textsf{imaginary part} \ \({}^{\mathscr{I}\!}P\)  of \(P\) are defined as %
\begin{equation}%
\label{RIPa}
{}^{\mathscr{R}\!}P(t)=\frac{P(t)+\overline{P(\overline{t})}}{2},
\quad{}{}^{\mathscr{I}\!}P(t)=\frac{P(t)-\overline{P(\overline{t})}}{2i}\,,
\end{equation}%
where the overline bar is used as a notation for complex
conjugation.

The  \textsf{even part} \({}^{\mathscr{E}\!}P\) and the
\textsf{odd part} \({}^{\mathscr{O}\!}P\)  of \(P\) are defined as %
\begin{equation}%
\label{EOIPa} {}^{\mathscr{E}\!}P(t)=\frac{P(t)+P(-t)}{2},
\quad{}{}^{\mathscr{O}\!}P(t)=\frac{P(t)-P(-t)}{2}\,,
\end{equation}%
In term of coefficients,
 \begin{subequations}
 \label{TeCo}
\begin{gather}%
\label{TeCo1}%
{}^{\mathscr{R}\!}P(t)=\sum\limits_{0\leq{}k}a_kt^k,\quad
{}^{\mathscr{I}\!}P(t)=\sum\limits_{0\leq{}k}b_kt^k, \\
\intertext{where}%
 a_k=\frac{p_k+\overline{p_k}}{2},
\quad{}b_k=\frac{p_k-\overline{p_k}}{2i}\,,
\label{TeCo2}%
\end{gather}%
\end{subequations}
and
\begin{equation}%
{}^{\mathscr{E}\!}P(t)=\sum\limits_{0\leq{}l}p_{2l}t^{2l},\quad
{}^{\mathscr{O}\!}P(t)=\sum\limits_{0\leq{}l}p_{2l+1}t^{2l+1}.
\label{EOCo}%
\end{equation}%
\end{defn}
\begin{nthm}[Hermite-Biehler] %
Let \(P\) be a polynomial, \(A={}^{\mathscr{R}\!}P\) and
\(B={}^{\mathscr{I}\!}P\) be the real and imaginary parts of
\(P\), i.e.
\[P(t)=A(t)+iB(t),\]
where \(A\) and \(B\) be a polynomials with real coefficients. In
order for all roots of \(P\) to be contained within the open upper
half plane \(\{z:\,\textup{Im}\,z>0\}\), it is necessary and
sufficient that the following three conditions be satisfied:
\begin{enumerate}
\item
The roots of each of the polynomials \(A\) and \(B\) are all real
and simple.
\item
The sets \(\mathscr{Z}_A\) and \(\mathscr{Z}_B\) of the roots of
the polynomials \(A\) and \(B\) interlace.
\item
The inequality
\begin{equation}%
\label{ALN} %
B^{\prime}(0)A(0)-A^{\prime}(0)B(0)>0
\end{equation}%
 holds.
\end{enumerate}
\end{nthm}

Let us formulate a version of Hermite-Biehler Theorem for the left half plane.
\begin{lem} %
\label{VHBT}%
Let \(M\) be a polynomial with positive coefficients,
\[M(t)=\sum\limits_{0\leq{}k\leq{}n}s_{k}t^{k},
\quad{}s_k>0,\ 0\leq{}k\leq{}n\,,\] and let
\({}^{\mathscr{E}\!}M\) and \({}^{\mathscr{O}\!}M\) be the even
and the odd parts of \(M\). In order for the polynomial \(M\) to be
dissipative it is necessary and sufficient that the following two
conditions be satisfied:
\begin{enumerate}%
\item
 The polynomials \({}^{\mathscr{E}\!}M\) and
\({}^{\mathscr{O}\!}M\) are conservative.
\item
The sets of roots for the polynomials  \({}^{\mathscr{E}\!}M\)
and \({}^{\mathscr{O}\!}M\) interlace.
\end{enumerate}%
\end{lem}%
\begin{lem}
\label{CPtBC}%
 Let \(W\),
\begin{equation}
W(t)=w_{0}+w_{2}t^{2}+w_{4}t^{4}\,\cdots\,+w_{2m-2}t^{2m-2}+w_{2m}t^{2m}
\end{equation}
  be an even polynomial with positive coefficients \(w_{2l}\):
  \[w_{0}>0,\,w_{2}>0,\,\dots\,,\,w_{2m}>0\,.\]
  In order for the polynomial \(W\) to be conservative it is
  necessary and sufficient that the polynomial \(M=W+W^{\prime}\)
   be dissipative, where \(W^{\prime}\) is the derivative of
  \(W\):
\begin{equation}
W^{\prime}(t)=2\cdot{}w_{2}t^{}+4\cdot{}w_{4}t^{3}\,\cdots\,
+(2m-2)\cdot{}w_{2m-2}t^{2m-3}+2m\cdot{}w_{2m}t^{2m-1}\,.
\end{equation}
\end{lem}
\begin{proof}[Proof of Lemma \ref{VHBT}] Let
\begin{equation}%
\label{LtoU}%
P(t)=M(it),\quad{}A(t)=({}^{\mathscr{E}\!}M)(it),
\quad{}B(t)=i\sp{-1}\cdot({}^{\mathscr{O}\!}M)(it),
\end{equation}%
so that
\[P(t)=A(t)+iB(t)\,.\]
\(A\) and \(B\) are polynomials with real coefficients:
\[A(t)=\sum\limits_{0\leq{}l\leq\left[\frac{n}{2}\right]}(-1)^ls_{2l}t^{2l},\quad{}
B(t)=t\!\!\!\!\sum\limits_{0\leq{}l\leq\left[\frac{n-1}{2}\right]}\!\!\!\!(-1)^ls_{2l+1}t^{2l}\,.\]
Moreover,
\begin{equation}%
\label{ALN1}%
 B^{\prime}(0)A(0)-A^{\prime}(0)B(0)=s_0\,s_1\,.
\end{equation}%
From \eqref{LtoU} it is evident that
\begin{align*}
(\textup{\footnotesize{}All roots of }A\textup{ \footnotesize
{}are real and simple})&\Leftrightarrow(\textup{\footnotesize{}The
polynomial}{}^{\ \mathscr{E}\!}M\textup{ \footnotesize {}is conservative})\\
(\textup{\footnotesize{}All roots of }B\textup{ \footnotesize
{}are real and simple})&\Leftrightarrow(\textup{\footnotesize{}The
polynomial}{}^{\ \mathscr{O}\!}M\textup{ \footnotesize {}is conservative})\\
(\textup{\footnotesize{}All roots of }P\textup{ \footnotesize{}lie
in \(\{z\!:\textup{Im}\,z>0\}\)})&
\Leftrightarrow(\textup{\footnotesize{}The polynomial }M\textup{
\footnotesize{}is dissipative})
\end{align*}
and under the condition that all roots of \(A\) and \(B\) are real,
\begin{equation*}
(\textup{\footnotesize{}The roots of \(A\) and \(B\)
interlace})\Leftrightarrow{}(\textup{\footnotesize{}The roots of \ %
\({}^{\mathscr{E}\!}M\) and \({}^{\mathscr{O}\!}M\) interlace})
\end{equation*}
Thus, Lemma \ref{VHBT} is an immediate consequence of the
Hermite-Biehler Theorem. \textit{The
inequality \eqref{ALN} is ensured automatically by \eqref{ALN1},
since the
coefficients \(s_k\) are assumed to be positive.}\\ %
\end{proof}%
 \begin{proof}[Proof of Lemma \ref{CPtBC}]
It is clear that the polynomials \(W\) and \(W^{\prime}\) are, respectively,
the even and the odd parts of \(M=W+W^{\prime}\):
\[W={}^{\mathscr{E}}\!M,\quad{}W^{\prime}={}^{\mathscr{O}}\!M.\]
Let \(M\) be dissipative. Then, according to Lemma \ref{VHBT},
\(W\) is conservative. Conversely, let \(W\) be conservative.
According to Rolle's Theorem, the polynomial \(W^{\prime}\) is also
conservative, and the sets of roots belonging to \(W\) and
\(W^{\prime}\) interlace. By Lemma \ref{VHBT}, the polynomial
\(M\) is dissipative. \end{proof}
\begin{proof}[Proof of Theorem \ref{H10H0}] The relation \eqref{MPOP}
means that the polynomial \(tW_{\partial V}^{1}(t)\) is the odd
part of the Steiner polynomial \(S_V\).
 Thus, we are in the situation of Lemma \ref{VHBT}.
 Since the polynomial
\(S_V\) is dissipative, the point \(z=0\) is not a root of \(M\),
that is, \(s_0(V)\not=0\). According to \eqref{EvEq}, this means
that \(\textup{Vol}_n(V)\not=0\). Thus, the set \(V\) is solid. By
Proposition \ref{PCMP}, all coefficients \(s_k(V)\) of the
polynomial \(S_V\) are strictly positive. According to Lemma
\ref{VHBT}, the polynomial \({}^{\mathscr{O}\!}(S_V)\) is
conservative. Since \({}^{\mathscr{O}\!}(S_V)(0)=0\), the polynomial
\(t^{-1}\cdot{}{}^{\mathscr{O}\!}(S_V)(t)=W_{\partial V}^{1}(t)\) is conservative as well.%
\end{proof}

\begin{proof}[Proof of Theorem \ref{JDM}.] In the course of the proof we
shall refer to some facts from the theory of entire functions
which are usually  formulated  in literature for functions whose
roots are in the upper rather than in the left half plane.
Therefore, it is convenient to pass from the variable \(t\) to the
variable \(it\). Given a function \(H(t)\) of the Hurwitz class
\(\mathscr{H}\), let \(f(t)=H(it)\).  Then \(f\) is an entire
function of exponential type. All roots of \(f\) are in the upper
half plane and the defect \(d_f\) of \(f\) is
non-negative, where
\[2\,d_f=\varlimsup\limits_{r\to+\infty}f(-ir)-\varlimsup\limits_{r\to+\infty}f(-ir)\,.\]
(It is clear that \(d_f=d_H\), where \(d_H\) is the same as in
\eqref{DefH}.) Thus the function \(f\) is in the class \(P\) as
this class was defined in \cite[Chapter VII, Section 4]{Lev1}. Let
\[f(t)=A(t)+iB(t)\,,\]
where \(A\) and \(B\) are real entire functions. Combining Lemma 1
from \cite[Chapter VII, Section 4]{Lev1} with Theorem 4 from
\cite[Chapter VII, Section 2]{Lev1},  we see that the functions
\(A\) and \(B\) possess the following properties:
\begin{enumerate}
\item
\(A\) and \(B\) are real entire functions of exponential type;
\item
\(A(0)B^{\prime}(0)-B(0)A^{\prime}(0)>0;\)
\item
For every \(\theta\in\mathbb{R}\), all roots of the linear
combination \(C_{\theta}\), where
\(C_{\theta}(t)=\cos\theta{}A(t)+\sin\theta{}B(t),\) are simple
and real.  (The entire functions \(A\) and \(B\) form \textit{a
real pair} in the terminology of N.G.Chebotarev, \cite{Cheb}.)
\end{enumerate}
According to Hadamard's Factorization Theorem, the entire function
\(C_{\theta}\) is in the Laguerre-P\'olya class. According to the
Jensen-Csordas-Williamson Theorem (Theorem \ref{JSW}), for each
\(n\), all roots of the Jensen polynomial
\(C_{\theta,n}(t)=\mathscr{J}_n(C_{\theta};t)\) are real and
simple. Thus, the real polynomials \(A_n(t)=\mathscr{J}_n(A;t)\)
and \(B_n(t)=\mathscr{J}_n(B;t)\) possess the following property:
\textit{For every \(\theta\in\mathbb{R}\), all roots of the linear
combination \(\cos\theta{}A_{n}(t)+\sin\theta{}B_{n}(t),\) are
real and simple}. (The polynomials \(A_n\) and \(B_n\) are a real
pair as well.) From the real pair property of the polynomials \(A_n\) and
\(B_n\) and the property
\(A_n(0)B_n^{\prime}(0)-B_n(0)A_n^{\prime}(0)\) it follows that
all roots of the polynomial \(f_n(t)=A_n(t)+iB_n(t)\) are in the
upper half plane. Thus, all roots of the polynomial
\(H_n(t)=f_n(-it)\) are in the left half plane. In other words,
the polynomial \(H_n\) is a Hurwitz polynomial. On the other hand,
from the construction it follows that
\(H_n(t)=\mathscr{J}_n(H;t)\,.\) \end{proof}
%%%%%%%%%%%%%%%%%%%%%%%%%%%%%%%%%%%%%%%%%%%%%%%%%%%%%%%%%%%%%%%%%%
\section{Properties of Steiner polynomials.
 \label{ChPrMiPo}  }
%%%%%%%%%%%%%%%%%%%%%%%%%%%%%%%%%%%%%%%%%%%%%%%%%%%
 {\small\textsf{RIGID MOTION INVARIANCE}}: \textit{Let
\(V,\,V\subset \mathbb{R}^n\), be a compact convex set,
{\large\(\tau\)} be a motion\footnote{The rigid motion of the space
\(\mathbb{R}^n\) is an affine transformation of \(\mathbb{R}^n\)
which preserves the Euclidean distance in \(\mathbb{R}^n\).} of
the space \(\mathbb{R}^n\)}, and {\large\(\tau\)}(V) be the image
of the set \(V\) under he motion {\large\(\tau\)}. Then
\(S_{{\mbox{\small\(\tau\)}}(V)}(t)=S_V(t)\).\\[1.0ex]
\noindent {\small\textsf{CONTINUITY}}: \textit{The correspondence
\(V\to{}S_V\) between compact convex sets \(V\) in
\(\mathbb{R}^n\) and their Steiner polynomials \(S_V\) is
continuous \footnote{\label{topo}The set of compact convex sets in
\(\mathbb{R}^n\) is equipped with the Hausdorff metric, the set of
all polynomials is equipped with the topology of the locally uniform
convergence in \(\mathbb{C}\).}.}
\\[1.0ex]
A sketch of the proof of  the continuity property can be found in
\cite[section \textsf{29}]{BoFe}, \cite[section
\textsf{19.2}]{BuZa}, \cite[section \textsf{5.1}]{Schn}, \cite{Web}.\\

\noindent {\small\textsf{MONOTONICITY}}:
\textit{%
Let \(V_1\) and \(V_2\) be compact convex sets in
\(\mathbb{R}^n\), and \(S_{V_1}, S_{V_2}\) be the associated
Steiner polynomials.} \textit{If \(V_1\subset{}V_2\), then the
coefficients \(s_k(V_1)\), \,\,\(s_k(V_2)\) of the polynomials
\(S_{V_1}\), \(S_{V_2}\), defined as in \eqref{MiP}, satisfy the
inequalities}
\begin{equation}
\label{IMiP} s_k(V_1)\leq{}s_k(V_2),\quad{}0\leq k\leq n\,.
\end{equation}
\noindent%
\textsf{Explanation.} According to the definition of the
\textit{mixed volumes},
\begin{equation}
\label{MiX}%
 s_k(V)=\frac{n!}{(n-k)!\,k!}\,
\textup{Vol}(\underbrace{V,V,\,\dots\,,V}_{n-k};\,\underbrace{B^n,B^n,\,\dots\,,B^n}_{k}).
\end{equation}
Inequalities \eqref{IMiP} follow from the \textit{monotonicity of the
mixed volumes \eqref{MiX} with respect to \(V\)}. (Concerning the
monotonicity of the mixed volumes see, for example, \cite[section \textsf{29}]{BoFe},
 \cite[section \textsf{19.2}]{BuZa},
\cite[Theorem \textsf{6.4.11}]{Web},  \cite[ section \textsf{5.1},
formula (5.1.23)]{Schn}.)
\begin{lem}
\label{PCMP}
\begin{enumerate}
\begin{subequations}
\label{PosProp}
\item[\textup{a)}.]
 For any compact convex set
\(V,\,V\subset{}\mathbb{R}^n\), the coefficients \(s_k(V)\) of its
Steiner polynomial, defined as in \eqref{MiP}, are non-negative:
\begin{equation}
\label{PPa} 0\leq{}s_k(V),\ \ 0\leq k\leq n.
\end{equation}
\textup{(}According to \eqref{EvEq}, the coefficient \(s_n(V)\) is strictly
positive.\textup{)}
\item[\textup{b)}.]
 If, moreover, the set \(V\) is solid, then
all coefficients \(s_k(V)\) are strictly  positive:
\begin{equation}
\label{PPb}%
 0<s_k(V),\ \ 0\leq k\leq n.
\end{equation}
\end{subequations}
The Weyl coefficients \(w_{2l}(\partial{}V),\,0\leq
l\leq\left[\frac{n-1}{2}\right]\), \,defined by \textup{Definition
\ref{DNWP}},
 are strictly positive as well.
\end{enumerate}
\end{lem}

\begin{proof} Taking  \(V\) as \(V_2\) and any one-point subset of \(V\) as
\(V_1\) in \eqref{IMiP}, we obtain \eqref{PPa}. If the set \(V\)
is solid, then there exist a ball \(x_0+\rho{}B^n\) of some
positive radius \(\rho\): \(x_0+\rho{}B^n\subset{}V\). Taking the
ball \(x_0+\rho{}B^n\) as \(V_1\) and \(V\) as \(V_2\) in
\eqref{IMiP}, we obtain the inequalities
\(s_k(x_0+\rho{}B^n)\leq{}s_k(V),\,\,\,0\leq k\leq n.\) Moreover,
\(s_k(x_0+\rho{}B^n)=s_k(\rho{}B^n)=\rho^{n-k}s_k(B^n)=
\rho^{n-k}\frac{n!}{k!\,(n-k)!}\,\textup{Vol}_n(B^n)>0.\)
\end{proof}

\vspace{2.0ex}
\begin{rem}%
\label{NSND2}%
The notion of an interior point for a set \(V\) depend on the
space in which \(V\) is embedded. The set \(V\),
\(V\subset\mathbb{R}^{n},\) which is non-solid with respect to the
`original' space \(\mathbb{R}^{n}\), is solid if \(V\) is
considered as embedded in the space \(\mathbb{R}^{d}\) of appropriate
dimension \(d,\,d<n\). The dimension \(\dim{}V\) of the
set \(V\) should be chosen as \(d\).
\end{rem}
\begin{defn}
\label{DDCS} Let \(V,\,V\subseteq\mathbb{R}^n\), be a convex set.
The \textsf{dimension} \(\dim V\) \,of~\(V\) is the dimension of
the smallest affine subspace of \(\mathbb{R}^n\) which contains
\(V\).
\end{defn}
\begin{lem}%
\label{Degdeg}%
Let \(V,\,V\subset\mathbb{R}^n\), be a compact convex set of
dimension \(d\):
\begin{equation}
\label{DiCoS} \dim{}V=d,\quad{}0\leq{}d\leq{}n.
\end{equation}
Then
\begin{equation}
s_k^{\mathbb{R}^n}(V)=0\ \ \textup{for}\ \ 0\leq{}k<n-d;\ \
s_k^{\mathbb{R}^n}(V)>0\ \ \textup{for}\ \ n-d\leq{}k\leq{}n\,.
\end{equation}
\end{lem}%
This lemma is a consequence of Lemma \ref{PCMP} and of the
following
\begin{lem}
\label{CoRiDi}%
 Let \(V\), \(V\subset{}\mathbb{R}^n\), be a convex set of
 dimension \(d\), \(d\leq{}n\), and let
 \(S_V^{\mathbb{R}^n}(t)=\sum\limits_{0\leq{}k\leq{}n}s_k^{\mathbb{R}^n}(V)t^k\)
 and\,%
 \footnote{Defining the Steiner polynomial
 \(S_V^{\mathbb{R}^d}\), we can assume that
 the
smallest affine subspace of \(\mathbb{R}^n\) which contains \(V\) is the space \(\mathbb{R}^d\).} %
  \(S_V^{\mathbb{R}^d}(t)=\sum\limits_{0\leq{}{\,k}\leq{}d}s_{\,k}^{\mathbb{R}^d}(V)t^k\)
 be the Minkovski polynomials of the set \(V\) with respect to the
 ambient spaces \(\mathbb{R}^n\) and \(\mathbb{R}^d\)
 respectively. Then
 \begin{equation}
S_V^{\mathbb{R}^n}(t)=t^{n-d}\cdot\hspace*{-1.0ex}\sum\limits_{0\leq{}{\,k}\leq{}d}
\,\frac{{\pi}^{\frac{n-d}{2}}%
\Gamma(\frac{k}{2}+1)}{\Gamma(\frac{k+n-d}{2}+1)}s_{\,k}^{\mathbb{R}^d}(V)t^k\,.
 \end{equation}
\end{lem}
Lemma \ref{CoRiDi} appears in slightly different notation as
Theorem \ref{IRPp} in Section \ref{ExInSp}, where a proof is
presented.

\begin{defn}
\label{CrSM}%
The mixed volumes appearing in \eqref{MiX} are said to be
\textsf{cross-sectional measures}  of the set \(V\) and are denoted
 as \(s_{n-k}(V)\):
\begin{equation}
\label{CSM}
\textup{Vol}(\underbrace{V,V,\,\dots\,,V}_{n-k};\,\underbrace{B^n,B^n,\,\dots\,,B^n}_{k})=
s_{n-k}(V),\quad 0\leq k\leq n\,.
\end{equation}
\end{defn}
Thus, the coefficients of the Minkovski polynomials \(S_V\), which
appear in \eqref{MiP}, can be written as %
\begin{equation}
\label{CrSFC}%
 s_k(V)=\binom{n}ks_{n-k}(V),\quad
\binom{n}k=\frac{n!}{k!\,(n-k)!} \text{ \small are binomial
coefficients,}%
\end{equation}
and the Steiner polynomial itself can be written as
\begin{equation}
\label{CrSFP}%
S_V(t)=\sum\limits_{0\leq k\leq n}\binom{n}kv_{n-k}(V)t^k,
\end{equation}
The following fact  will be used in Section \ref{LoDi}: \\[2.5ex]
\noindent \textbf{Alexandrov\,--\,Fenchel Inequalities.}
\textit{Let \(V\), \(V\subset\mathbb{R}^n\), be a compact convex
set. Then its cross-sectional measures \(v_k(V)\) satisfy the
inequalities}
%\vspace*{-1.5ex}%
\begin{equation}
\label{AFIn}%
v_k^2(V)\geq{}v_{k-1}(V)\,v_{k+1}(V),\ \ 1\leq k\leq n-1\,.
\end{equation}

\vspace{2.0ex}
\noindent%
A.D.\,Alexandrov published two proofs of this inequality in
\cite{Al1} and \cite{Al2}. The first of them, a combinatorial proof,
is carried out for convex polyhedra. The second proof is more
analytical. It uses the theory of self-adjoint elliptic operators
depending on a parameter. This proof is carried out for smooth
convex sets. To the general case, both proofs are generalized by
limit arguments. The first proof is developed in detail in the
textbook \cite{Le}. The second proof is reproduced in Busemann
\cite{Bus}. It has become customary to refer to \eqref{AFIn} as the
 `Alexandrov-Fenchel
inequality', because Fenchel \cite{Fen} also stated the inequality
and sketched  the proof. Its detailed presentation was never
published. At the end of 1978  Tessier in Paris and
A.G.Khovanski{\u\i} in Moscow both independently obtained an
algebraic-geometrical proof of the Alexandrov-Fenchel inequality using
 the Hodge index
theorem. This proof is developed in \S 27 of the English
translation of \cite{BuZa} and was written by
A.G.\,Khovanski{\u\i}. (In the Russian original of \cite{BuZa} an
erroneous algebraic proof of the Alexandrov-Fenchel inequality was
included, which has been excluded
 in the English translation.) Regarding the Alexandrov-Fenchel
 inequality, see also \cite[\S\,20]{BuZa},  and
 \cite[Section \textbf{6.3}]{Schn}.
\begin{defn}%
\label{DefLogConc}%
A sequence \(\{p_k\}_{0\leq{}k\leq{}n}\) of non-negative numbers:
\begin{equation}%
\label{NoNeCo}%
p_k\geq{}0\,, \ \ \ 0\leq{}k\leq{}n,
\end{equation}%
 is said to be
\textsf{logarithmically concave}, if the following inequalities hold:
\begin{equation}
\label{LogCoIn}
p_k^2\geq{}p_{k-1}p_{k+1},\quad{}1\leq{}k\leq{}n-1\,.
\end{equation}
\end{defn}%
Thus, the Alexandrov-Fenchel inequalities can be formulated in the
form:\\[2.0ex]
\hspace*{3.0ex}
\begin{minipage}{0.9\linewidth}
 \textit{For any convex set \(V\), the sequence \(\{v_k(V)\}_{0\leq{}k\leq{}n}\) of
 its
cross sectional measures  is logarithmically concave.}
\end{minipage}

\vspace{2.0ex} Under the extra condition \eqref{NoNeCo}, the
logarithmic concavity inequalities \eqref{LogCoIn} for the
coefficients of the polynomial
\begin{equation}
\label{TuPo}%
 P(t)=\sum\limits_{0\leq{}k\leq{}n}\binom{n}k\,p_kt^k,
\end{equation}
or for the coefficients of the entire function
\begin{equation}
\label{TuEf}
P(t)=\sum\limits_{0\leq{}k<\infty}\frac{p_k}{k!}\,t^{k},
\end{equation}
have been considered in connection with the distribution of the roots
belonging to \(P\). In this setting, such (and analogous) inequalities are
commonly known as \textsf{Tur\'an Inequalities}
(\textsf{Tur\'an-like Inequalities).} Concerning Tur\'an
inequalities see, for example, \cite{KaSc} and
\cite{CrCs2}.\\
\begin{rem}
\label{CoTiI}%
The Tur\'an inequalities \eqref{LogCoIn} for the coefficients of
the polynomial \eqref{TuPo} or entire function \eqref{TuEf} impose
some restrictions on roots of \(P\). However,
\textsf{these inequalities alone do not ensure} that all roots of
\(P\) are located in the left half plane
\(\{z:\,\textup{Re}\,z<0\}.\)
\end{rem}
For example, given \(m\in\mathbb{N}\), let
\begin{equation}
\label{SE}%
p_k=1 \textup{ \ \small{}for \ }k=0, 1, \,\dots\,,\,m  \textup{ \
\small{}and \  }p_k=0\textup{ \ \small{}for \ }k>m.
\end{equation}
Such \(p_k\) satisfy the Tur\'an inequalities \eqref{LogCoIn}. The
function \eqref{TuEf} corresponding to these \(p_k\) is the
polynomial
\begin{equation}
\label{SES} %
P_m(t)=\sum\limits_{0\leq{}k\leq{}m}\frac{t^k}{k!}\,.
\end{equation}
This polynomial is a truncation of the exponential series. It is
known that already for \(m=5\) the polynomial \eqref{SES} has two
roots located in the half plane \(\{z:\,\textup{Re}\,z>0\}.\) G.\,Szeg\"o,
\cite{Sz}, studied the limiting distribution of roots for
sequence of polynomials \(P_m\), \eqref{SES}, as \(m\to\infty\).
From his results on the limiting distributions of the roots, it
follows that for large \(m\), the polynomial \(P_m\) not only has
roots in the half plane \(\{z:\,\textup{Re}\,z>0\}\), but that the total
number of its roots located there has a positive density as
\(m\to\infty\). For further information on the roots of truncated power series
we refer the reader to the book \textup{\cite{ESV}} and to the survey
\textup{\cite{Ost1}.}
 For \(m<n\), the polynomial
\eqref{TuPo} with \(p_k\) as in \eqref{SE} takes the form
\begin{equation}
\label{TBiP}%
P_{m,n}(t)=\sum\limits_{0\leq{}k\leq{}m}\binom{n}kt^k\,.
\end{equation}%
I.V.\,Ostrovskii, \textup{\cite{Ost2}}, studied the limiting
distribution of roots for sequence of the polynomials
\(P_{m,n}\) as \(m,n\to\infty\),
\(m/n\to\alpha,\,\alpha\in(0,1).\) From his results it follows
that for large \(m,n\): \(n/m=O(1),\,n/(n-m)=O(1)\) the polynomial
\(P_{m,n}\) not only has roots in the half plane
\(\textup{Re}\,z>0\), but that the total number of its roots
located there has a positive density as
\(m,n\to\infty,\,m/n\to\alpha\in(0,1)\).%

\vspace{3.0ex}
\section{The Routh-Hurwitz criterion.\label{DeCr}}
 If possible, a geometric approach to finding the location
of the roots belonging to Steiner and Weyl polynomials would be
useful. At the moment we are, however, not able to do this. The
only general tool from geometry which we can use are the
the Alexandrov-Fenchel inequalities \eqref{AFIn} for
cross-sectional measures \(v_k(V)\) of convex sets. Therefore, one
should express all polynomials which we investigate in terms of
these cross-sectional measures.

As it was explained in \eqref{CrSFP}, the expression of the
Steiner polynomial \(S_V^{\mathbb{R}^n}\) for the convex set
\(V,\,V\subset\mathbb{R}^n\), in terms of the cross-sectional
measures \(v_k(V)\) is
\begin{equation}
\label{CrSFP1}%
S_V^{\mathbb{R}^n}(t)=\sum\limits_{0\leq k\leq
n}\binom{n}kv_{n-k}(V)t^k\,.
\end{equation}
\begin{lem} %
\label{EWRC}%
 Let \(\mathscr{M}\) be a closed convex surface, \(\dim\mathscr{M}=n\), and let
 \(V\), \(V\subset\mathbb{R}^{n+1}\), be a generating convex set:
 \(\mathscr{M}=\partial{}V.\)

 Then the Weyl polynomial \(W_{\mathscr{M}}^{\infty}\) can be
 expressed as
\begin{equation}
\label{MaWePM}%
W_{\mathscr{M}}^{\infty}(t)=\sum\limits_{0\leq{}l\leq{}\left[\frac{n}{2}\right]}
\frac{(n+1)!}{2^ll!(n-2l)!}\,v_{n-2l}(V)\,t^{2l}\,,
\end{equation}
where \(v_k(V)\) are the cross-sectional measures of the
generating convex set \(V\).
\end{lem}
 \begin{proof} The expression \eqref{MaWePM} is a consequence of
 \eqref{UWP}, \eqref{NWP} and \eqref{CrSFC}.
 \end{proof}
%%%%%%%%%%%%%%%%%%%%%%%%%%%%%%%%%%%%%%%%

The inequalities \eqref{AFIn}  for the
coefficients of the polynomials \eqref{CrSFP1} and \eqref{MaWePM}
comprise one of two general tools, which will be used in the
study of the location of roots of these polynomials. The second
general tool at our disposal is the collection
of criteria describing the dissipativity
(or, alternatively conservativeness) property of polynomials
in term of their coefficients.
 \begin{nthm}[Routh-Hurwitz]
Let
 \begin{equation}
 \label{RaHuPo}
 P(t)=a_0t^n+a_1t^{n-1}+\,\dots{}\,+a_{n-1}t+a_n
 \end{equation}
  be a
polynomial with strictly positive coefficients:
\begin{equation}
 \label{PoCo}
a_0>0,\,a_1>0,\,\dots\,,\,a_{n-1}>0,\,a_n>0\,.
\end{equation}
For the polynomial \(P\) to be dissipative, it is necessary and
sufficient that all the determinants
\(\Delta_k,\,k=1,\,2,\,\dots\,,\,n-1,\,n,\) be strictly positive:
\begin{equation}
\label{PRHD}%
\Delta_1>0,\,\Delta_2>0,\,\dots\,,\,\Delta_{n-1}>0,\,\Delta_n>0,
\end{equation}
where {\footnotesize
\begin{multline}
\label{RHD} \Delta_1=a_1,\quad\Delta_2=
\begin{vmatrix}a_1&a_3\\
a_0&a_2
\end{vmatrix},\quad
\Delta_3=
\begin{vmatrix}a_1&a_3&a_5\\
a_0&a_2&a_4\\
0&a_1&a_3
\end{vmatrix},\quad \\
\Delta_4=
\begin{vmatrix}a_1&a_3&a_5&a_7\\
a_0&a_2&a_4&a_6\\
0&a_1&a_3&a_5\\
0&a_0&a_2&a_4
\end{vmatrix},\quad\dots
\quad ,\ \Delta_n=
\begin{vmatrix}a_1&a_3&a_5&\dots&0\\
a_0&a_2&a_4&\dots&0\\
0&a_1&a_3&\dots&0\\
\hdotsfor{5} \\
\hdotsfor{4}&a_n%
\end{vmatrix}\cdot
\end{multline}
}
\vspace{2.0ex}
\end{nthm}

\noindent
This result, as well as many related results, can be found in
\cite[Chapter XV]{Gant}. See also \cite{KrNa}.
\begin{rem}
\label{LieCh}%
 Actually, to prove that the polynomial \(P\), \eqref{RaHuPo}, of
 degree \(n\) with positive coefficients \(a_k,\,k=1,\,2,\,\dots\,,\,n,\) is dissipative,
  there is no need to inspect \textit{all} Hurwitz determinants
 \(\Delta_k,\,k=1,\,2,\,\dots\,,\,n,\) for positivity. It is
 enough to inspect either the determinants \(\Delta_k\) with even
 \(k\), or the determinants \(\Delta_k\) with odd
 \(k\). \textup{(See \cite[Chapter XV, \S 13]{Gant}.)}
\end{rem}
Applying the Routh-Hurwitz criterion to determine whether the Steiner
polynomial \(S_V^{\mathbb{R}^n}\) is dissipative, we should take, according to
\eqref{CrSFP1},
\begin{equation}%
 \label{RHuC}%
 a_k=\frac{n!}{k!(n-k)!}v_k(V)\,, \quad 0\leq{}k\leq{}n, \ \ a_k=0, \quad
 k>n\,.
 \end{equation}%
From the criterion for dissipativity, the criterion of
conservativeness can be derived easily.
\begin{nthm}[Criterion for conservativeness]
Let
\begin{equation}
\label{CoPoA}
P(t)=a_0t^{2m}+a_{2}t^{2m-2}+\,\dots{}\,+a_{2m-2}t^{2}+a_{2m}
\end{equation}
be a polynomial with strictly positive coefficients
\(a_{2l},\,0\leq{}l\leq{}m\):
\begin{equation}
\label{PCFA}%
 a_0>0,\,a_2>0,\,\dots\,,\,a_{2m-2}>0,\,a_{2m}>0\,.
\end{equation}
For the polynomial \(P\) to be conservative, it is necessary and
sufficient that all the determinants
\(D_k,\,k=1,\,2,\,\dots\,,2m-1,\,2m,\) be strictly positive:
\begin{equation}
\label{PRHDc}%
D_1>0,\,D_2>0,\,D_3>0,\,\dots\,,\,D_{2m-1}>0,\,D_{2m}>0,
\end{equation}
where \(D_k\) are constructed from the
coefficients  of the polynomial \(P\) according to the following
rule:\(D_k\) is the determinant \(\Delta_k\),
\eqref{RHD}, whose entries \(a_{2l},\,\,0\leq{}l\leq{}m,\)  are
the coefficients of the polynomial \(P\), and
\(a_{2l+1}=(m-l)\,a_{2l}\), \(0\leq{}l\leq{}m-1\): {\footnotesize
\begin{multline*}
  D_1=\textstyle{}m\,a_0,\quad{}D_2=
\begin{vmatrix}\textstyle
m\,a_0&(m-1)a_2\\
a_0&a_2
\end{vmatrix},\quad
D_3=
\begin{vmatrix}ma_0&(m-1)a_2&(m-2)a_4\\
a_0&a_2&a_4\\
0&m\,a_0&(m-1)a_2
\end{vmatrix},%\quad \\[3.0ex]
\end{multline*}
\begin{equation}
\label{RHCo} D_4=
\begin{vmatrix}m\,a_0&(m-1)\,a_2&(m-2)a_4&(m-4)a_6\\
a_0&a_2&a_4&a_6\\
0&m\,a_0&(m-1)\,a_2&(m-2)a_4\\
0&a_0&a_2&a_4
\end{vmatrix},\quad\dots
%\quad ,\\[3.0ex]
\end{equation}
\begin{equation*}
 D_{2m}=
\begin{vmatrix}m\,a_0&(m-1)a_2&(m-2)a_4&\dots&0\\
a_0&a_2&a_4&\dots&0\\
0&ma_0&(m-1)a_2&\dots&0\\
\hdotsfor{5} \\
\hdotsfor{4}&a_{2m}%
\end{vmatrix}\cdot
\end{equation*}
}
\end{nthm}
\begin{proof} This theorem is the immediate consequence of the Hermite-Biehler
Theorem and Lemma \ref{CPtBC}.{\ }\end{proof}

Applying criterion for conservativeness to determine whether the Weyl
polynomial \(W_{\mathscr{M}}^{\infty}\) is conservative, we should take, according to
 \eqref{MaWePM} and \eqref{CoPoA},
\begin{multline}%
 \label{RHuCl}%
 a_{2l}=\frac{(n+1)!}{2^{m-l}(m-l)!(2l+n-2m)!}v_{2l+n-2m}(V)\,,\\
 \quad 0\leq{}l\leq{}m, \ \ a_{2l}=0, \quad
 l>m\,,\quad \textup{where }m=\textstyle{\left[\frac{n}{2}\right]}
 \end{multline}%

\section{The case of low dimension:
proofs of Theorems  \(\boldsymbol{\ref{LDC}}\) and \(\boldsymbol{\ref{LDCW}}\).\label{LoDi}}

\begin{proof}[Proof of Theorem \ref{LDC}]
 We apply the Routh-Hurwitz
criterion for dissipativity to the Steiner polynomial
\(S_{\,\,V}^{{\mathbb{R}^n}}\). `Expanding' the Hurwitz determinants
\(\Delta_k\), \eqref{RHD} , and taking into account that \(a_k=0\)
for \(k>n\), we obtain that for \(n\leq{}5\),
\begin{subequations}
\label{ODet}
\begin{gather}
\Delta_1=a_1\tag{\theequation.\footnotesize{1}}\label{ODet1}\\
\Delta_2=a_1a_2-a_0a_3\,,\tag{\theequation.\footnotesize{2}}\label{ODet2}\\
\Delta_3=a_1a_2a_3+a_0a_1a_5-a_0a_3^2-a_1^2a_4\,,\tag{\theequation.%
\footnotesize{3}}\label{ODet3}\\
\Delta_4=a_1a_2a_3a_4+a_0a_2a_3a_5+2a_0a_1a_4a_5-a_1^2a_4^2-a_0^2a_5^2-a_0a_3^2a_4-a_1a_2^2a_5,
\tag{\theequation.%
\footnotesize{4}}\label{ODet4}\\
\Delta_5=a_5\Delta_4\,,\tag{\theequation.%
\label{ODet5}\footnotesize{5}}
\end{gather}
\end{subequations}
where we should take \(a_k\) as in \eqref{RHuC}.

 According to the Routh-Hurwitz criterion, we have to prove
 that \(\Delta_1>0,\ \Delta_2>0,\,\dots\,,\,\Delta_n>0\).
 The cases \(n=2,\,3,\,4,\,5\) will be considered separately.
 Since \(V\) is solid, we have that \(v_k(V)>0,\,0\leq{}k\leq{}n\).
  (Corollary \ref{PCMP}.b  and \eqref{CrSFC}\,.)\\
  Thus, the determinant \(\Delta_1=\binom{n}{1}v_1(V)\) is always
  positive.\\[1.0ex]
  The cases \(n=2,\,3,\,4,\,5\) will again be considered separately.
  For simplicity, we write \(v_k\) instead
  \(v_k(V)\).\\[1.0ex]
  \hspace*{3.0ex}\(\boldsymbol{n=2.}\) In this case,
  \[a_0=v_0,\,a_1=2v_1,\,a_2=v_2,\]
  \[\Delta_2=2v_2v_1\,.\]
  The inequality \(\Delta_2>0\) is evident. Thus,
  the Steiner polynomial \(S_{V}^{\mathbb{R}^2}\) is dissipative.
  \\[1.0ex]
  \hspace*{3.0ex}\(\boldsymbol{n=3.}\) In this case,
  \[a_0=v_0,\quad{}a_1=3v_1,\quad{}a_2=3v_2,\quad{}a_3=v_3\,,\quad{}a_k=0,\,k>3\,.\]
  Substituting these expressions for \(a_k\) into \eqref{ODet}, we
  obtain
  \[\Delta_2=9v_1v_2-v_0v_3,\quad{}\Delta_3=v_3\Delta_2\,.\]
  Thus, \(S_V^{\mathbb{R}^3}\) is dissipative
  if, and only if
  \begin{subequations}
  \label{MP3}
  \begin{align}
 9v_1v_2&>v_0v_3,\tag{\ref{MP3}}
\end{align}
  \end{subequations}
  \hspace*{3.0ex}\(\boldsymbol{n=4.}\) In this case,
  \[a_0=v_0,\quad{}a_1=4v_1,\quad{}a_2=6v_2,\quad{}a_3=4v_3\,\quad{}a_4=v_4,\quad{}a_k=0,\,k>4\,.\]
 Substituting these expressions for \(a_k\) into \eqref{ODet}, we
  obtain
  \begin{multline*}
  \Delta_2=24v_1v_2-4v_0v_3,\quad{}\Delta_3=96v_1v_2v_3-16v_0v_3^2-16v_1^2v_4\,,\quad
\Delta_4=v_4\Delta_3\,.
  \end{multline*}
Thus, \(S_V^{\mathbb{R}^4}\) is dissipative
  if, and only if the pair of inequalities
  \begin{subequations}
  \label{MP4}
  \begin{align}
  6v_1v_2&>v_0v_3,%\tag{\ref{MP4}.\footnotesize{2}}
  \label{MP42} \\
  6v_1v_2v_3&>v_0v_3^2+v_1^2v_4\,.
  %\tag{\ref{MP4}.\footnotesize{3}}
  \label{MP43}
  \end{align}
  \end{subequations}
  \hspace*{3.0ex}\(\boldsymbol{n=5.}\) In this case,
  \[a_0=v_0,\ a_1=5v_1,\ a_2=10v_2,\ a_3=10v_3,\,\ a_4=5v_4,\ %
  a_5=v_5,\,\,a_k=0,\,k>5\,.\]
 Substituting these expressions for \(a_k\) into \eqref{ODet}, we
  obtain
  \begin{multline*}
  \Delta_2=50v_1v_2-10v_0v_3,\quad{}\Delta_3=500v_1v_2v_3+5v_0v_1v_5-100v_0v_3^2-125v_1^2v_4\,,\\
\Delta_4=2500v_1v_2v_3v_4+100v_0v_2v_3v_5+50v_0v_1v_4v_5\\
-625v_1^2v_4^2-v_0^2v_5^2-500v_0v_3^2v_4 -500v_0v_1^2v_5\,,\ \
\Delta_5=v_5\Delta_4\,.
  \end{multline*}
Thus, \(S_V^{\mathbb{R}^5}\) is dissipative
  iff the following three inequalities hold:
  \begin{subequations}
  \label{MP5}
  \begin{align}
  5v_1v_2&>v_0v_3,
  %\tag{\ref{MP5}.\footnotesize{2}}
  \label{MP52} \\
  100v_1v_2v_3+v_0v_1v_5&>20v_0v_3^2+25v_1^2v_4\,,
  %\tag{\ref{MP5}.\footnotesize{3}}
  \label{MP53}
  \\
  2500v_1v_2v_3v_4+100v_0v_2v_3v_5+&50v_0v_2v_4v_5>{}
  %\tag{\ref{MP5}.\footnotesize{4}}
  \label{MP54}\\
 {} >625v_1^2v_4^2&+500v_0v_3^2v_4+500v_1v_2^2v_5+v_0^2v_5^2\,.\notag
  \end{align}
  \end{subequations}
  As it is claimed  below in Lemma \ref{CoAFI},
  the inequalities \eqref{MP3}, \eqref{MP4},
  \eqref{MP5}, where \(v_k=v_k(V)\) are the cross-sectional
  measures of the solid compact set \(V\) of  appropriate
  dimension, are consequences of the Alexandrov-Fenchel
  inequalities. This completes the proof. \end{proof}
  \begin{rem} The above
\textsf{\textsl{does not}} hold for all \(n\). If \(n\) is large
enough, then the conditions \( v_k^2\geq{}v_{k-1}v_{k+1},\quad
1\leq{}k\leq{}n-1\,,\) for positive numbers \(v_k\), do not
imply the inequalities \(\Delta_k\geq{}0\) for all
\(k=1,\,\dots,\,n\), where \(\Delta_k\) is constructed from
\(a_k\)'s given by
\(a_k=\binom{n}{k}v_k\).
 Already for \(n=30\),  \(\Delta_5<0\) for certain
\(v_k\) satisfying these conditions.
 Moreover, as we will see later,
for \(n\)  large enough, there exist examples of  compact convex
sets \(V\subset\mathbb{R}^n\) for which the Steiner polynomial
\(S_V\) is not dissipative and the Weyl polynomial
\(W_{\partial{}V}^1\) is not conservative. Although the sets \(V\)
in such examples  are solid, they are "almost degenerate".
\end{rem}
  \begin{lem}
  \label{CoAFI}
  Let \(v_k,\,0\leq{}k\leq{}n,\) be strictly positive numbers
  satisfying the inequalities
  \begin{equation}
  \label{CoFVk}
  v_k^2\geq{}v_{k-1}v_{k+1},\quad 1\leq{}k\leq{}n-1\,.
  \end{equation}
  Then:\\[1.0ex]
  \hspace*{2.0ex}\textup{1}. If \(n=3\),  the inequality \eqref{MP3} holds;\\[1.0ex]
  \hspace*{2.0ex}\textup{2}. If \(n=4\),  the inequalities \eqref{MP4} holds;\\[1.0ex]
  \hspace*{2.0ex}\textup{3}. If \(n=5\),  the inequalities \eqref{MP5} holds.
  \end{lem}
  \begin{proof}[Proof of Lemma \ref{CoAFI}]
  The \textit{logarithmic concavity} inequalities \eqref{CoFVk} imply the inequalities
  \begin{equation}
  \label{CoFVkI}
  v_pv_s\geq{}v_qv_r\,,
  \end{equation}
  where \(p,\,q,\,r,\,s\) are arbitrary indices satisfying the conditions
  \begin{equation}%
  \label{ConC}
  1\leq{}p\leq{}q\leq{}r\leq{}s\leq{}n\,,\quad{}p+s=q+r\,.
  \end{equation}

\vspace*{1.0ex}
\noindent
   \(\boldsymbol{n=3}.\)\hfill \(v_1v_2\geq{}v_0v_3\,\ \ \Rightarrow\ \ \eqref{MP3}\)\hfill\hfill

\vspace{1.0ex}
 \noindent
\(\boldsymbol{n=4}.\)\hfill \(v_1v_2\geq{}v_0v_3\,,\ \ v_2v_3\geq{}v_1v_4\ \
  \Rightarrow\ \ \eqref{MP4}\)
\hfill\hfill

\vspace{1.0ex}
 \noindent
\(\boldsymbol{n=5}.\)\\
\hspace*{2.0ex}\hfill \(v_1v_2\geq{}v_0v_3\,,\  v_0v_5\geq{}v_1v_4\,,\
v_2v_3\geq{}v_1v_4\,,\  v_3v_4\geq{}v_2v_5\,,\ \ v_2v_3\geq{}v_0v_5\,\
  \Rightarrow\  \eqref{MP5}\)\hfill{\ }
 \end{proof}
%%%%%%%%%%%%%%%%%%%%%%%%%%%%%%%%%%%%%%%%%%%%%%%%%%%%%%%%%%%%%%%%%%%%%%%

\begin{proof}[PROOF of THEOREM \ref{LDCW}] We apply the \textsf{Conservativeness
Criterion}, which was formulated in the previous section,
for Weyl polynomial \(W_{\mathscr{M}}^{\infty}\). Expanding the
determinants \(D_k\), \eqref{RHCo} , and taking into account
that \(a_{2l}=0\) for \(l>\left[\frac{n}{2}\right]\), we obtain
that for \(n\leq{}5\),
that is,\,%
\footnote{Recall that \(n=\dim \mathscr{M},\,n+1=\dim V:
\mathscr{M}=\partial V.\)}
 for \(m=\left[\frac{n}{2}\right]\leq{}2\),
\begin{subequations}
\label{DDet}
\begin{gather}
D_1=ma_0\tag{\theequation.\footnotesize{1}}\label{DDet1}\\
D_2=a_0a_2\,,\tag{\theequation.\footnotesize{2}}\label{DDet2}\\
D_3=a_0\big((m-1)a_2^2-2ma_0a_4\big)\,,\tag{\theequation.%
\footnotesize{3}}\label{DDet3}\\
D_4=a_0a_4(a_2^2-4a_0a_4)\,.
\tag{\theequation.%
\footnotesize{4}}\label{DDet4}
\end{gather}
\end{subequations}
where we take \(a_{2l}\) as in \eqref{RHuCl}.

According to the Conservativeness Criterion, we have to prove that
\(D_1>0\), \(D_2>0,\,\dots\,,\,D_{2m}>0\), where
 \(m=\left[\frac{n}{2}\right]\).\\
Since \(V\) is solid, \(v_k(V)>0,\,0\leq{}k\leq{}n+1\).
  (Corollary \ref{PCMP}.b  and \eqref{CrSFC}\,.)
Thus, the determinants \(D_1,\,D_2\) are always positive.\\[1.0ex]
Therefore, if \(n=2\) or if \(n=3\), that is, if \(m=1\), the Weyl
polynomial \(W_{\mathscr{M}}^{\infty}\) is conservative. Of
course, this fact is evident without referring to the
conservativeness criterion:\\[1.0ex]
In the case \(\boldsymbol{n=2}\), according to \eqref{RHuC} or \eqref{MaWePM},
\[W_{\mathscr{M}}^{\infty}(t)=3v_2+3v_0t^2\,.\]
In the case \(\boldsymbol{n=3}\), according to \eqref{RHuC} or \eqref{MaWePM},
\[W_{\mathscr{M}}^{\infty}(t)=v_3+3v_1t^2\,.\]
In both cases, \(n=2\) or \(n=3\), the polynomial \(W_{\mathscr{M}}^{\infty}\)
is  clearly conservative.\\[1.0ex]
In the cases \(n=4,\ n=5\), which corresponds to the case where \(m=2\),
\[D_3=a_0(a_2^2-4a_0a_4),\quad D_4=a_4D_3\,.\]
According to \eqref{RHuCl}, we must also take the following
two cases into account:\\
In the case that \(\boldsymbol{n=4}\)
\[a_0=15v_0,\quad{}a_2=30v_2,\quad{}a_4=5v_4\,.\]
Thus,
\[D_3=15v_0(900v_2^2-300v_0v_4)\,.\]
The conditions \(D_3>0,\,D_4>0\) take the form
\begin{equation}
\label{WePo4}
3v_2^2>v_0v_4\,.
\end{equation}
In the case that \(\boldsymbol{n=5}\)
\[a_0=90v_1,\quad{}a_2=60v_3,\quad{}a_4=6v_5\,.\]
Thus,
\[D_3=90v_1(900v_2^2-300v_0v_4)\,.\]
The conditions \(D_3>0,\,D_4>0\) take the form
\begin{equation}
\label{WePo5}
5v_3^2>3v_1v_5\,.
\end{equation}
The Weyl polynomial \(W_{\mathscr{M}}^{\infty}\) is, therefore,
conservative in the cases \(n=4\) and \(n=5\) if the inequalities
\eqref{WePo4} and \eqref{WePo5} hold, respectively.
\(v_k=v_k(V)\) are the cross-sectional
  measures of the solid compact set \(V\) generating the surface
  \(\mathscr{M}:\ \mathscr{M}=\partial{}V.\) The inequalities
\eqref{WePo4} and \eqref{WePo5} are, in turn, consequences of the inequalities
\(v_2^2\geq{}v_0v_4\) and \(v_3^2\geq{}v_1v_5\), respectively.
 The latter inequalities are special cases
of the inequalities \eqref{CoFVkI}. Thus, in
the cases \(n=4\) and \(n=5\) the Weyl
\(W_{\mathscr{M}}^{\infty}\) polynomial with infinite index is
conservative. By Lemma \ref{LWPo}, all Weyl polynomials
\(W_{\mathscr{M}}^{\,p},\,p=1,\,2,\,3,\,\dots\,\,,\) are
 conservative as well.\end{proof}

%%%%%%%%%%%%%%%%%%%%%%%%%%%%%%%%%%%%%%%%%%%%%%%%%%%%%%%%
\section{Extending the ambient
space.\label{ExInSp}}
%%%%%%%%%%%%%%%%%%%%%%%%%%%%%%%%%%%%
\paragraph{Adjoint convex sets.}
Let \(V\) be a compact convex set, \(V\subset\mathbb{R}^n\).
 We may consider the space \(\mathbb{R}^n\) as a subspace
of  \(\mathbb{R}^{n+q}\), where
 \(q=1,\,2,\,3,\,\dots\,\). The
embedding \(\mathbb{R}^n\) in   \(\mathbb{R}^{n+q}\) is canonical:
\begin{equation*}%
\label{Emb}%
\mathbb{R}^{n}\hookrightarrow \mathbb{R}^{n+q}:\ \ %
(\xi_1,\,\dots\,,\,\xi_n)\to(\xi_1,\,\dots\,,\,\xi_n;%
\underbrace{\,0,\,\dots\,,\,0}_q\,).
\end{equation*}%
 Thus, the set \(V\), which originally was considered as a
subset of \(\mathbb{R}^n\), may also be considered as a subset of
\(\mathbb{R}^{n+q}\). In other words, we identify the set
\(V\subset\mathbb{R}^n\) with the set \(V\times{}0^{q}\), which is
the Cartesian product of the set \(V\) and the origin
\(0^{q}\) of the space \(\mathbb{R}^q\):
\(V\times{}0^{q}\subset\mathbb{R}^{n+q}.\)
\begin{defn}%
\label{DeAdM}
 Given a compact convex set \(V\), \(V\subset{}\mathbb{R}^n\),
and a number \(q\),
 \(q=0,\,1,\,2,\,3,\,\dots\,,\,\,\) the \textsf{\(q\)-th adjoint to \(V\)
set \(V^{(q)}\)} is defined as:
\begin{equation}
\label{AdS}
V^{(q)}\stackrel{\textup{\tiny{}def}}{=}V\times{}0^{q},\quad
V^{(q)}\subset\mathbb{R}^{n+q}\,,
\end{equation}
where \(0^q\) is the zero point of the space \(\mathbb{R}^{q}\),
and the space \(\mathbb{R}^{n+q}\) is considered as the Cartesian product:
\(\mathbb{R}^{n+q}=\mathbb{R}^{n}\times\mathbb{R}^{q}\).\\
The Steiner polynomial
\(S_{\,V\times{}0^q}^{\,\mathbb{R}^{n+q}}\) of the \(q\)-th
adjoint set, denoted by \(V^{(q)}\),
\begin{equation}
\label{defAMP}
S_{\,V\times{}0^q}^{\mathbb{R}^{n+q}}=\textup{Vol}_{n+q}(V\times{}0^q+tB^{n+q}),
\end{equation}
 is said to be the \textsf{\(q\)-th adjoint
Steiner polynomial for the set~\(V\)}.\\
\end{defn}
For \(q=0\), the set \(V^{(0)}\) coincides with \(V\), and the polynomial
 \(S_{\,V\times{}0^0}^{\mathbb{R}^{n+0}}\) coincides with  \(S_{\,\,V}^{\mathbb{R}^{n}}\).
 For \(q=1\), the set \(V^{(1)}\) is called the
 squeezed cylinder with the base \(V\).
%%%%%%%%%%%%%%%%%%%%%%%%%%%%%%%%%%%%%%%%%%%%%%%%%%%%%%%%%%%%%%%%%%%%%%%%%%%%%%
%%%%%%%%%%%%%%%%%%%%%%%%%%%%%%%%%%%%%%%%%%%%%%%%%%%%%%%%%%%%%%%%%%%%%%%%%%%%%%
%%%%%%%%%%%%%%%%%%%%%%%%%%%%%%%%%%%%%%%%%%%%%%%%%%%%%%%%%%%%%%%%%%%%%%%%%%%%%%%%%%
%%%%%%%%%%%%%%%%%%%%%%%%%%%%%%%%%%%%%%%%%%%%%%%%%%%%%%%%%%%%%%%%%%%%%%%%%%%%%%%%%%%
\paragraph{Steiner polynomials for adjoint sets.}
Let us consider the following question:\\[1.0ex]
\centerline{\textit{What is the relationsheep between the
polynomials \(S_{\,V}^{\mathbb{R}^{n}}(t)\) and
\(S_{\,V\times{}0^q}^{\mathbb{R}^{n+q}}(t)\)}?}\\[2.0ex]
We have
\begin{lem}
\label{IPRom}%
Let \(V\) be a compact convex set in \(\mathbb{R}^n\), and
\begin{equation}%
\label{IRPom1}%
 S_{\,V}^{\mathbb{R}^n}(t)=\sum\limits_{0\leq{}k\leq{}n}s_{\,k}^{\mathbb{R}^n}(V)t^k
\end{equation}%
be the Steiner polynomial with respect to the ambient space
\(\mathbb{R}^n\). Then the Steiner polynomial
\(S_{\,V\times{}0}^{\mathbb{R}^{n+1}}(t)\) is equal to:
\begin{equation}%
\label{MPn1om}%
S_{\,V\times{}0^1}^{\mathbb{R}^{n+1}}(t)=%
t\!\!\sum\limits_{0\leq{}k\leq{}n}\frac{{\pi}^{1/2}
\Gamma(\frac{k}{2}+1)}{\Gamma(\frac{k+1}{2}+1)}\,s_{\,k}^{\mathbb{R}^n}(V)%
\, t^k\,.
\end{equation}%
\end{lem}
Proceeding the induction over \(q\) and using Lemma \ref{IPRom} we obtain
\begin{thm}
\label{IRPp}%
Let \(V\) be a compact convex set in \(\mathbb{R}^n\), and
\begin{equation}%
\label{MPn}%
S_{\,V}^{\mathbb{R}^n}(t)=
\sum\limits_{0\leq{}k\leq{}n}s_{\,k}^{\mathbb{R}^n}(V)t^k
\end{equation}%
 be the Steiner polynomial of the set
\(V\). Then the \(q\)-th adjoint Steiner polynomial
\(S_{\,V\times{}0^q}^{\,\mathbb{R}^{n+q}}(t)\) of the set \(V\) is
equal to:
\begin{gather}%
\label{MPn2}%
S_{\,V\times{}0^q}^{\,\mathbb{R}^{n+q}}(t)=%
\sum\limits_{0\leq{}k\leq{}n}s_{\,k}^{\mathbb{R}^n}(V)\,\gamma_{\,\,k}^{(q)}\,\,%
t^{k+q}\,,\\
\intertext{where}
\label{MuSe}
\gamma_{\,k}^{(q)}={\pi}^{q/2}\,\frac{\Gamma(\frac{k}{2}+1)}{\Gamma(\frac{k+q}{2}+1)}\,,
\quad k=0,\,1,\,2,\,\dots\,\,;\ \ q=0,\,1,\,2,\,\dots.
\end{gather}
\vspace{2.0ex}
\end{thm}
\noindent
A sketch of the proof for this theorem can be found in \cite[Chapter VI, Section 6.1.9]{Had}.
 A detailed proof is presented below.
\begin{rem}%
\label{FnTp}%
Theorem \ref{IRPp} means that the sequence of the coefficients\\
\(\{s_{\,k}^{\mathbb{R}^{n+q}}(V\times{}0^q)\}_{0\leq k\leq n+q}\)
of the polynomial \(S_{\,V\times{}0^q}^{\mathbb{R}^{n+q}}\):
\begin{equation}%
S_{\,V\times{}0^q}^{\,\mathbb{R}^{n+q}}(t)=
\sum\limits_{0\leq{}k\leq{}n+q}s_{\,k}^{\mathbb{R}^{n+q}}(V\times{}0^q)\,t^k
\end{equation}%
 are obtained from the sequence of the coefficients
\(\{s_{k}(V)\}_{0\leq k\leq n}\) of the polynomial \(S_{\,V}^{\mathbb{R}^{n}}\),
\textup{\eqref{MPn}}, by means of a shift and multiplication:
\begin{subequations}
\label{SiMu}
\begin{alignat}{2}%
s_{\,\,k}^{\mathbb{R}^{n+q}}(V\times{}0^q)&=0, & & 0\leq k<q;
\label{SiMua1} \\ %
s_{\,{k+q}}^{\mathbb{R}^{n+q}}(V\times{}0^q)&=
s_{k}^{\mathbb{R}^{n}}(V)\,\gamma_{\,k}^{(q)},& \ \ &0\leq k\leq
n\,.\label{SiMua2}
\end{alignat}%
\end{subequations}
\end{rem}
\begin{rem} %
\label{CompPol}%
According to \textup{Theorem \ref{IRPp}},
the transformation which maps the polynomial \(S_{\,V}^{{\mathbb{R}^{n}}}\)
into the polynomial \(S_{\,V\times{}0^q}^{{\mathbb{R}^{n+q}}}\) is of the form
\begin{equation}
\label{MTP} %
\sum\limits_{0\leq k\leq n}s_kt^k\to \sum\limits_{0\leq k\leq
n}\gamma_ks_kt^k,
\end{equation}
where \(\gamma_k\) is a certain sequence of \textsf{multipliers}.
\textup{(}The factor \(t^q\) in front of the sum in \eqref{MPn2}  can
be omitted\textup{)}. Transformations of the form
\textup{\eqref{MTP}} were already discussed is Section \ref{LPEF}.
Such transformations were considered in relation with
locations of roots for polynomials and entire functions.
\end{rem}%
\begin{lem}%
\label{NonMult}%
 For any \(q\), \(q=1,\,2,\,3,\,\ldots\), the
sequence \(\{\gamma_{\,k}^{(q)}\}_{k=0,\,1,\,2,\,\ldots}\) is not
a multiplier sequence in the sense of \textup{Definition
\ref{PReZ}}.
\end{lem}%
\begin{proof} In Section  \ref{LocRoot} we explain that the entire
function
\begin{equation}
\label{Multq}
\mu_q(t)=\displaystyle\sum\limits_{0\leq{}k<\infty}\frac{\gamma_{\,k}^{(q)}}{k!}t^k
\end{equation}
has infinitely many non-real roots.
 The entire function \(\mu_q(t)\),
\eqref{Multq}, appears as the function
\(\mathcal{M}_{B^{\infty}\times{}0^q}(t)\) in Section
\eqref{JPMEb}. (Up to a constant factor which is not essential for
study the roots.)
 According to the P\'olya-Schur
Theorem, which was formulated in \ref{LPEF}, the sequence
\(\{\gamma_{\,k}^{(q)}\}_{k=0,\,1,\,2,\,\ldots}\) is not a
multiplier sequence. \end{proof}

\begin{rem}
In Section \ref{LocRoot} we study the function \(\mu_q\)
in much more detail that is needed to prove Lemma \ref{NonMult}.
We investigate for which \(q\) function's roots
are located in the left half plane.
The question of whether or  not there are non-real roots
is much less involved. This question may be answered
from very general considerations.
 The function \(\mu_q\) admits the
integral representation:
\begin{equation}%
\label{IRMuF}%
\mu_q(t) =
q\omega_q\int\limits_{0}^{1}(1-\xi^2)^{\frac{q}{2}-1}\xi\,{}e^{\xi{}t}\,d\xi\,.
\end{equation}
(Expanding the exponential \(e^{\xi{}t}\) into the Taylor series,
we see that the Taylor coefficients of the function on the right
hand side of \eqref{IRMuF} are the numbers
\(\frac{\gamma_{\,k}^{(q)}}{k!}\).) From \eqref{IRMuF} it follows
that the function \(\mu_q(t)\) is an entire function of
exponential type and that its indicator diagram is the interval
\([0,\,1]\). Moreover,
\(\sup_{-\infty<t<\infty}|\mu_q(it)|<\infty\). In particular, the
function \(\mu_q(it)\) belongs to the class of entire functions,
denoted by \(C\) in \cite[ Lecture 17]{Lev2}. It follows from the
\textup{Cartwright-Levinson Theorem (Theorem 1 of the Lecture
17 from \cite{Lev2})}  that the function \(\mu_q(t)\)
has infinitely many roots. These roots have  positive density, and
the `majority' of these roots are located `near' the rays \(\arg
t=\frac{\pi}{2}\) and \(\arg t=-\frac{\pi}{2}\). In particular,
the function \(\mu_q(t)\) has infinitely many non-real roots. (We
already used this reasoning to prove \textup{Statement 2 of Theorem
\ref{BLPC}}.)
\end{rem}

%%%%%%%%%%%%%%%%%%%%%%%%%%%%%%%%%%%%%%%%%%%%%%%%%%%%%%%%%%%%%%%%%%%%%%%%%%%%%%%%%%%%%
\noindent%
\begin{proof}[Proof of Lemma \ref{IPRom}] Let
\((x,s)\in{\mathbb{R}^{n+1}}\), where \(x\in\mathbb{R}^n\), and
\(s\in\mathbb{R}\). Then, by the Pythagorean Theorem,
\[\textup{dist}_{\mathbb{R}^{n+1}}^2((x,s),V\times{0})=
\textup{dist}_{\mathbb{R}^n}^2(x,V)+s^2.\] Therefore, the
equivalence
\begin{equation}
\label{eqd}%
\Big( \textup{dist}_{\mathbb{R}^{n+1}}((x,s),V\times{}0)\leq
t\Big) \Longleftrightarrow
\Big(\textup{dist}_{\mathbb{R}^{n}}(x,V)\leq\sqrt{t^2-s^2}\Big)
\end{equation}
holds. Let
\begin{equation}
\label{EV} \mathfrak{T}_{V\times{}0^1}^{\mathbb{R}^{n+1}}(t)=
\{(x,s)\in\mathbb{R}^{n+1}:\,\textup{dist}_{\mathbb{R}^{n+1}}((x,s),{V\times{}0^1})\leq
t\}.
\end{equation}
be the \(t\)-neighborhood of the set \(V\times{}0^1\) with respect
to the ambient space \(\mathbb{R}^{n+1}\). Thus,
\begin{equation}%
\label{Vnp1}%
\textup{Vol}_{n+1}(\mathfrak{T}_{V\times{}0^1}^{\mathbb{R}^{n+1}}(t))=
S_{V\times{}0^1}^{\mathbb{R}^{n+1}}(t).
\end{equation}%
For fixed \(s\in\mathbb{R}\), let \(\mathfrak{S}(s)\) be the
`horizontal section' of the set
\(\mathfrak{T}_{V\times{}0^1}^{\mathbb{R}^{n+1}}(t)\) on the
`vertical level' \(s\):
\[\mathfrak{S}(s)=
\{x\in\mathbb{R}^n:\,(x,s)\in\mathfrak{T}_{\,\,V}^{\mathbb{R}^{n+1}}(t)\}.\]
Clearly
\begin{equation}%
\label{Kav}
\textup{Vol}_{n+1}(\mathfrak{T}_{V\times{}0^1}^{\mathbb{R}^{n+1}}(t))=
\int\textup{Vol}_n(\mathfrak{S}(s))ds\,.
\end{equation}%
 The equivalence \eqref{eqd} means that
\begin{equation*}%
\mathfrak{S}(s)=\mathfrak{T}_{V}^{\mathbb{R}^{n}}(\sqrt{t^2-s^2})=
\{x\in\mathbb{R}^n:\,\textup{dist}_{\mathbb{R}^{n}}(x,V)\leq{}\sqrt{t^2-s^2}\}.
\end{equation*}
Thus,
\begin{equation}%
\label{VoSe}%
 \textup{Vol}_n(\mathfrak{S}(s))=S_V^{\mathbb{R}^n}(\sqrt{t^2-s^2}).
\end{equation}%
From \eqref{Kav} and \eqref{VoSe} it follows that
\begin{equation*}%
\label{fin}
S_{V\times{}0^1}^{\mathbb{R}^{n+1}}(t)=\int\limits_{-t}^tS_V^{\mathbb{R}^{n}}(\sqrt{t^2-s^2})ds\,.
\end{equation*}%
Making the substitutions \(s\to{}ts^{1/2}\), we obtain
\begin{equation*}
S_{\,\,V}^{\mathbb{R}^{n+1}}(t)=
t\,\int\limits_{0}^1S_V^{\mathbb{R}^{n}}(t(1-s)^{1/2})s^{-1/2}ds.
\end{equation*}
Substituting the expression \eqref{IRP1} for
\(S_{V}^{\mathbb{R}^{n}}\) into the latter formula, we obtain
\begin{equation*}
S_{V\times{}0^1}^{\mathbb{R}^{n+1}}(t)=t\sum\limits_{0\leq k\leq
n}s_k(V)\,t^k\int\limits_0^1(1-s)^{k/2}s^{-1/2}ds.
\end{equation*}
According to Euler,
\[\int\limits_0^1(1-s)^{k/2}s^{-1/2}ds=\textup{B}{\textstyle\big(\frac{k}{2}+1,\frac{1}{2}\big)}=
\frac{\Gamma\big(\frac{1}{2}\big)\,
\Gamma\big(\frac{k}{2}+1\big)}{\Gamma\big(\frac{k+1}{2}+1\big)}=
\pi^{1/2}\frac{
\Gamma\big(\frac{k}{2}+1\big)}{\Gamma\big(\frac{k+1}{2}+1\big)}\cdot\]
Thus, \eqref{MPn1om} holds.
 \end{proof}

\noindent%
 \begin{proof}[Proof of Theorem \ref{IRPp}]
We proceed by induction over \(q\). For \(q=0\), the assertion is
trivially clear. We now show that if the statement holds for \(q\),
then it also holds for \(q+1\). Since
 \(V\times{}0^{q+1}=(V\times{}0^{q})\times{}0^1\), and
  \(\mathbb{R}^{n+q+1}=\mathbb{R}^{n+q}\times{}\mathbb{R}^{1}\),
 we can apply Lemma \ref{IPRom} to the convex set \(V\times{}0^{q}\)
 whose Steiner polynomial is \eqref{MPn2} (by the inductive hypothesis).
 The inductive hypothesis can be formulated as
\begin{subequations}
\label{SiMup}
\begin{alignat}{2}%
s_{\,\,k}^{\mathbb{R}^{n+q}}(V\times{}0^q)&=0, & & 0\leq k<q;
\label{SiMup1} \\ %
s_{\,{k}}^{\mathbb{R}^{n+q}}(V\times{}0^q)&=
s_{k-q}^{\mathbb{R}^{n}}(V)\,\gamma_{\,k-q}^{(q)},& \ \ &q\leq
k\leq q+n\,.\label{SiMup2}
\end{alignat}%
\end{subequations}
 By Lemma \ref{IPR},
 \begin{subequations}
\label{SiMuInd}
 \begin{align}
 \label{SiMuaInd1}
 s_{\,{k}}^{\mathbb{R}^{(n+q)+1}}((V\times{}0^q)\times{}0^1)&=0,\ \ k=0\,;\\
 \label{SiMuaInd2}
s_{\,{k}}^{\mathbb{R}^{(n+q)+1}}((V\times{}0^q)\times{}0^1)&=
s_{\,{k-1}}^{\mathbb{R}^{(n+q)}}((V\times{}0^q))\cdot{}\gamma_{\,k-1}^{(1)}\,,\
\ 1\leq{}k\leq{}n+q+1\,.
 \end{align}
 \end{subequations}
 In view of the identity
 \[\gamma_{k}^{(q)}\cdot{}\gamma_{k+q}^{(1)}=\gamma_{k}^{(q+1)}\,,\]
 \eqref{SiMuInd} takes the form \eqref{SiMup} with \(q\) replaced
 by \(q+1\)\,.
 \end{proof}
\begin{rem}
From \eqref{Vpdb} and \eqref{MuSe} it follows that
\begin{equation}
\label{IdFG}%
 \gamma_{\,k}^{(q)}=\frac{\omega_{k+q}}{\omega_{k}}\,.
\end{equation}
Thus, the equalities \eqref{SiMua2} can be rewritten as
\begin{equation}
\label{IntrCo}%
\frac{\,s_{\,{k+q}}^{\mathbb{R}^{n+q}}(V\times{}0^q)}{\omega_{k+q}}=
\frac{\,s_{\,{k}}^{\mathbb{R}^{n}}(V)\,}{\omega_{k}},\quad
q=0,\,1,\,2,\,\ldots\,.
\end{equation}
The equality \eqref{IntrCo} holds for \(k=0,\,1,\,\ldots\,,\,n\).
The value \(s_k^{\mathbb{R}^n}(V)\) is not yet
defined for other values of \(k\). Let us agree that
\begin{equation}
\label{Agre}%
\omega_k=1 \textup{\ \ for\ \ }k<0,\ \ s_k^{\mathbb{R}^n}(V)=0
\textup{\ \ for\ \ }k<0\textup{\ \ and for\ \ }k>n\,.
\end{equation}
The equality \eqref{IntrCo} then holds for \emph{every}
\(k\in\mathbb{Z}\).  For \(k>n\) or for \(k<-q\)  \eqref{IntrCo}
is trivial. For \(-q\leq{}k\leq{}-1\) \eqref{IntrCo} coincides
with \eqref{SiMua1}. For \(0\leq{}k\leq{}n\) \eqref{IntrCo}
coincides with \eqref{SiMua2}. The ratio
\(\displaystyle\frac{\,s_{\,{n-k}}^{\mathbb{R}^{n}}(V)\,}{\omega_{n-k}}\)
is called the \textit{\(k\)-th intrinsic volume of \(V\)} in
\cite{McM}.
\end{rem}
\paragraph{The Steiner polynomials for the \(q\)-th adjoint to the ball \(B^n\).}
In particular, applying Theorem \ref{IRPp} to the case \(V=B^n\),
\(B^n\subset\mathbb{R}^n\), we obtain:
\begin{equation}
\label{ExpMinAdq}%
S_{B^n\times{}0^q}^{\mathbb{R}^{n+q}}(t)=
\omega_n\omega_q\,t^q\mathcal{M}_{B^n\times{}0^q}(nt)\,,
\end{equation}
where the normalized Steiner polynomial
\(\mathcal{M}_{B^n\times{}0^q}\) is defined as
\begin{equation}
\label{NMPq}%
\mathcal{M}_{B^n\times{}0^q}(t)=
\sum\limits_{0\leq{}k\leq{}n}\frac{n!}{(n-k)!n^k}%
\frac{\Gamma{(\frac{q}{2}+1)}\Gamma(\frac{k}{2}+1)}{\Gamma(\frac{k+q}{2}+1)}\frac{t^k}{k!}\,.
\end{equation}
The polynomial \(\mathcal{M}_{B^n\times{}0^q}\) is the Jensen
polynomial associated with the entire function
\(M^{B^n\times{}0^q}\):
\begin{equation}
\label{JPMPa}
\mathcal{M}_{B^n\times{}0^q}(t)=\mathscr{J}_n(\mathcal{M}_{B^{\infty}\times{}0^q};t),
\end{equation}
where
\begin{equation}
\label{JPMEb}
\mathcal{M}_{B^{\infty}\times{}0^q}(t)=\sum\limits_{0\leq{}k<\infty}
\frac{{\Gamma(\frac{q}{2}+1)}\Gamma(\frac{k}{2}+1)}{\Gamma(\frac{k+q}{2}+1)}\frac{t^k}{k!}\,.
\end{equation}
Comparing \eqref{JPMEb} with
\eqref{IRMuF}, we obtain
\begin{equation}
\label{NMPq1}%
\mathcal{M}_{B^{\infty}\times{}0^q}(t)=
q\int\limits_{0}^{1}(1-\xi^2)^{\frac{q}{2}-1}\xi\,{}e^{\xi{}t}\,d\xi\,.
\end{equation}
For every \(q=0,\,1,\,2,\,\ldots\,\), the function
\(\mathcal{M}_{B^{\infty}\times{}0^q}\) is an entire function of
the exponential type one.
\begin{lem}\ \
\label{HPEGM}
\begin{enumerate}
\item %
For \(q=0,\,1,\,2,\,4\), the entire function
\(\mathcal{M}_{B^{\infty}\times{}0^q}\) is in the Hurwitz class
\(\mathscr{H}\);
\item
For \(q\geq{}5\),  the entire function
\(\mathcal{M}_{B^{\infty}\times{}0^q}\) is not in the Hurwitz
class: it has infinitely many roots in the open right half plane
\(\{z:\,\,\textup{Im} z >0\}\).
\end{enumerate}
\end{lem}
Proof of this Lemma is presented in Section \ref{LocRoot}.
Statement 2 is a consequence of the asymptotic calculation of the
function \(\mathcal{M}_{B^{\infty}\times{}0^q}\). (See Lemma
\ref{NonHq}.)

 For \(q=0\), we have that \(\mathcal{M}_{B^{\infty}\times{}0^0}=e^t\).
\(\mathcal{M}_{B^{\infty}\times{}0^0}=e^t\) is therefore
of type \textup{I} in the Laguerre-P\'olya class:
\(\mathcal{M}_{B^{\infty}\times{}0^0}\in
\mathscr{L}\text{-}\mathscr{P}\text{-}\textup{I}\). For \(q=2\)
and \(q=4\) the functions \(\mathcal{M}_{B^{\infty}\times{}0^q}\)
can be calculated explicitly. The case \(q=1\) is more involved. The case \(q=3\)
remains open.\\
\noindent%
 \begin{proof}[Proof of Statement 1 of Theorem \ref{NMR}] Let \(q\geq{5}\) be given.
 According to Statement 2 of Lemma \ref{HPEGM}, the function
\(\mathcal{M}_{B^{\infty}\times{}0^q}\) has infinitely many roots
in the open right half plane. In view of the approximation
property of the Jensen polynomials (Lemma \ref{CJPL}),  some roots
of the Jensen polynomial
\(\mathscr{J}_n(\mathcal{M}_{B^{\infty}\times{}0^q};t)\) for
\(n\geq{}n(q)\) are located in the open right half plane. Because
of \eqref{ExpMinAdq}  and \eqref{JPMPa}, some roots of the
Minkovski polynomial \(S_{B^n\times{}0^q}\) of the (non-solid)
convex set \(B^n\times{}0^q\),
\(B^n\times{}0^q\subset{}\mathbb{R}^{n+q}\), are located in the
open right half plane. For fixed \(n:\,n\geq{}n(q)\), consider the
ellipsoids \(E_{n,\,q,\,\varepsilon}\) defined in \eqref{Ell},
\(E_{n,\,q,\,\varepsilon}\subset\mathbb{R}^{n+q}\). For
\(\varepsilon>0\), the ellipsoid \(E_{n,\,q,\,\varepsilon}\) is a
solid convex set with respect to the ambient space
\(\mathbb{R}^{n+q}\). The family of convex sets
\(\{E_{n,\,q,\,\varepsilon}\}_{\varepsilon>0}\) is monotonic,
 (See Remark \ref{AprC1} and footnote \ref{monot}), and
 \begin{equation}
 \label{LimEll}
 \lim_{\varepsilon\to+0}E_{n,\,q,\,\varepsilon}=B^n\times{}0^q\,.
 \end{equation}
It is known that the Steiner polynomials \(S_V(t)\) depends on
the set \(V\) continuously: see Section \ref{ChPrMiPo} and
footnote \ref{topo}. Therefore,
\begin{equation}%
\label{LRLMP}%
\lim_{\varepsilon\to{}0}S_{E_{n,\,q,\,\varepsilon}}^{\mathbb{R}^{n+q}}(t)=
S_{B^n\times{}0^q}^{\mathbb{R}^{n+q}}(t)
\end{equation}%
locally uniformly in \(\mathbb{C}\). Hence, there exists an
\(\varepsilon(q,n), \ \varepsilon(q,n)>0\) such that the Steiner
polynomial \(S_{E_{n,\,q,\,\varepsilon}}^{\mathbb{R}^{n+q}}\) has
roots located in the open right half plane. %
\end{proof}

\paragraph{The Weyl polynomials for the boundary surfaces of the adjoint convex sets.}
Keeping Definition \ref{DeWPDC} in mind, we now define the adjoint
Weyl polynomials\ \(W_{\,V\times{}0^q}^{\,\,p}\).
\begin{defn}
\label{DaWP}
Given a convex compact set \(V\), \(V\subset{}\mathbb{R}^n\), %
and a number \(q\),
 \(q= 0,\,1,\,2,\,3,\,\dots\,,\,\,\), the \(q\)-th adjoint
  Weyl polynomial \(W_{\,\partial(V\times{}0^q)}^{\,\,1}\) of the index \(1\) for the
convex surface \(\partial(V\times{}0^q)\) is defined using
the odd part of the \(q\)-th adjoint Steiner polynomial
\(S_{\,V\times{}0^q}^{\mathbb{R}^{n+q}}\):
\begin{equation}%
\label{fdAwp}%
2tW_{\partial(V\times{}0^q)}^{\,\,1}(t)
\stackrel{\textup{\tiny{}def}}{=}S_{\,V\times{}0^q}^{\mathbb{R}^{n+q}}(t)-
S_{\,V\times{}0^q}^{\mathbb{R}^{n+q}}(-t)\,,
\end{equation}%
where \(S_{\,V^{(q)}}^{\mathbb{R}^{n+q}}\) is the \(q\)-th adjoint
Steiner polynomial which was introduced in \textup{Definition
\ref{DeAdM}}. More specifically\,%
\footnote{According to the convention \eqref{Agre},
\(s_{2l+1}^{\mathbb{R}^{n+q}}(V\times{}0^q)=0\) for \(2l+1<0\) or
\(2l+1>n+q\,\).}\,,%
\begin{equation}%
\label{fdAwp1}%
W_{\partial(V\times{}0^q)}^{\,\,1}(t)=
\sum\limits_{l\in\mathbb{Z}}s_{2l+1}^{\mathbb{R}^{n+q}}(V\times{}0^q)t^{2l}\,.
\end{equation}%
\end{defn}
Using \eqref{fdAwp1}, we can define the Weyl coefficients
\(w_{2l}(\partial{}(V\times{}0^q))\) according to Definition
\ref{DNWP}:
\begin{equation}%
\label{WCASu}%
w_{2l}(\partial(V\times{}0^q))=
s_{2l+1}^{\mathbb{R}^{n+q}}(V\times{}0^q)\frac{(2\pi)^l\omega_1}{\omega_{1+2l}}\,.
\end{equation}%
The Weyl polynomials \(W_{\partial{}(V\times{}0^q)}^p\) with higher \(p\)
are then defined
according to\,%
\footnote{We remark that
\(\displaystyle{}\frac{2^{-l}\,\Gamma(\frac{p}{2}+1)}{\Gamma(\frac{p}{2}+l+1)}=
(2\pi)^{-l}\frac{\omega_{2l+p}}{\omega_{p}}\,.\)}
 Definition \ref{DGWP}:
 \begin{defn}
\begin{equation}%
\label{WPACSpq}%
W_{\partial{}(V\times{}0^q)}^p(t)\stackrel{ \textup{\tiny def}}{=}
\sum\limits_{l\in\mathbb{Z}}w_{2l}(\partial{}(V\times{}0^q))
(2\pi)^{-l}\frac{\omega_{2l+p}}{\omega_{p}}t^{2l}\,.
\end{equation}%
\end{defn}%
\vspace{2.0ex}
\noindent
Thus,
\begin{equation}%
\label{WPACS}%
\omega_{p}t^pW_{\partial{}(V\times{}0^q)}^p(t)=
\sum\limits_{l\in\mathbb{Z}}\frac{s_{2l+1}^{\mathbb{R}^{n+q}}(V\times{}0^q)}{\omega_{2l+1}}\,
\omega_{1}\omega_{2l+p}\,t^{2l+p}\,.
\vspace{2.0ex}
\end{equation}%

\noindent
Let us clarify how the Weyl polynomials for the convex surfaces
\(\partial{}V\) and \(\partial{}(V\times{}0^q)\) are related. \  Here,
we also have to distinguish between even and odd cases for \(q\).
\begin{lem}%
\label{TINVQ}%
Let \(V,\,\,V\subset\mathbb{R}^n,\) be a solid compact convex set,
and let \(p>0,\,q>0\) be integers. Then
 \begin{enumerate}
 \item%
 For even \(q\)
 \begin{equation}
\label{CeQ}%
\omega_pt^p\cdot{}W_{\partial(V\times{}0^q)}^{\,p}(t)=
\omega_{p+q}t^{p+q}\cdot{}W_{\partial{}V}^{p+q}(t)\,;
\end{equation}
 \item%
 For odd \(q\)
\begin{equation}
\label{CoQ}
 \omega_pt^p\cdot{}W_{\partial(V\times{}0^q)}^{\,p}(t)=
\omega_{p+q-1}t^{p+q-1}W_{\hspace*{0.5ex}\partial{}(V\times{}0^1)}^{p+q-1}(t)\,.
\end{equation}
\end{enumerate}
\end{lem}%
\begin{proof}[Proof of Lemma \ref{TINVQ}]\ \\[1.0ex]
\hspace*{3.0ex}1. Let \(q\) be even. The equality \eqref{IntrCo} with
\(k=2l+1-q\) takes the form
\[\frac{\,s_{\,{2l+1}}^{\mathbb{R}^{n+q}}(V\times{}0^q)}{\omega_{2l+1}}=
\frac{\,s_{\,{2l+1-q}}^{\mathbb{R}^{n}}(V)\,}{\omega_{2l+1-q}}.\]
From this and \eqref{WPACS} it follows
\begin{equation*}%
%\label{WPACSn1}%
\omega_{p}t^pW_{\partial{}(V\times{}0^q)}^p(t)=
\sum\limits_{l\in\mathbb{Z}}\frac{s_{2l+1-q}^{\mathbb{R}^{n}}(V)}{\omega_{2l+1-q}}\,
\omega_{1}\omega_{2l+p}\,t^{2l+p}\,.
\end{equation*}%
Changing the summation variable: \(l\to{}l+\frac{q}{2}\), we
obtain
\begin{equation*}%
%\label{WPACSn12}%
\omega_{p}t^pW_{\partial{}(V\times{}0^q)}^p(t)=
\sum\limits_{l\in\mathbb{Z}}\frac{s_{2l+1}^{\mathbb{R}^{n}}(V)}{\omega_{2l+1}}\,
\omega_{1}\omega_{2l+p+q}\,t^{2l+p+q}\,.
\end{equation*}%
The expression on the right hand side of the latter equality has the
same form as the expression on the right hand side of
\eqref{WPACS}. One need only replace \(V\times{}0\) with \(V\),
\(p\) with \(p+q\), and \(n+q\) with \(n\) in \eqref{WPACS} to see this.
It follows that \eqref{CeQ} holds.\\
\hspace*{3.0ex}2. Let \(q\) be odd. The equality
implies the equality
\[\frac{\,s_{\,{2l+1}}^{\mathbb{R}^{n+q}}(V\times{}0^q)}{\omega_{2l+1}}=
\frac{\,s_{\,{2l+1-(q-1)}}^{\mathbb{R}^{n+1}}(V\times{}0^1)\,}{\omega_{2l+1-(q-1)}}.\]
From this and \eqref{WPACS} it follows that
\begin{equation*}%
%\label{WPACSn1}%
\omega_{p}t^pW_{\partial{}(V\times{}0^q)}^p(t)=
\sum\limits_{l\in\mathbb{Z}}\frac{s_{2l+1-
(q-1)}^{\mathbb{R}^{n+1}}(V\times{}0^1)}{\omega_{2l+1-(q-1)}}\,
\omega_{1}\omega_{2l+p}\,t^{2l+p}\,.
\end{equation*}%
Changing the summation variable: \(l\to{}l+\frac{q-1}{2}\), we
obtain
\begin{equation*}%
%\label{WPACSn12}%
\omega_{p}t^pW_{\partial{}(V\times{}0^q)}^p(t)=
\sum\limits_{l\in\mathbb{Z}}\frac{s_{2l+1}^{\mathbb{R}^{n+1}}(V\times{}0^1)}{\omega_{2l+1}}\,
\omega_{1}\omega_{2l+p+q-1}\,t^{2l+p+q-1}\,.
\end{equation*}%
The expression on the right hand side of the latter equality has the
same structure that the expression on the right hand side of
\eqref{WPACS}. Again, one need only replace \(V\times{}0^q\)
with \(V\times{}0^1\),
\(p\) with \(p+q-1\), and \(n+q\) with \(n+1\). It follows that
\eqref{CoQ} holds. \end{proof}

Lemma \ref{TINVQ} tell us that when considering the
boundary surfaces
\(\partial{}(V\times{}0^q)\) of the \(q\)-th adjoint convex sets
\(V\times{}0^q\), it is enough to
restrict our considerations to the cases \(q=0\) and \(q=1\).
%%%%%%%%%%%%%%%%%%%%%%%%%%%%%%%%%%%%%%%%%%%%%%%%%
\noindent%
 \begin{proof}[Proof of Statement 2 of Theorem \ref{NMR}]
 By Statement 2 of Theorem \ref{str}, the entire function
 \(\mathcal{W}_{\partial{}(B^{\infty}\times{}0)}^{\,p+q-1}\) has
 infinitely many non-real roots. (We assume that \(p+q-1\geq{}5.\)) If \(n\) is large enough,
the Jensen polynomial
\(\mathcal{W}^{p+q-1}_{\partial{}(B^{n+1}\times{}0)}(t){=}
\mathscr{J}_{2[n/2]}(\mathcal{W}^{p+q-1}_{\partial{}(B^{\infty}\times{}0)};t)\)
also has non-real roots. According to \eqref{JReWP2} and
\eqref{WRemOr2}, the Weyl polynomial
\(W^{p+q-1}_{\partial{}(B^{n}\times{}0)}(t)\) has roots which do
not belong to the imaginary axis. By Statement 2 of Lemma
\ref{TINVQ},
\[W_{\partial{}(B^{n}\times{}0^q)}^{\,p}
={\textstyle\frac{\omega_{p+q-1}}{\omega_p}t^{q-1}}W_{\partial{}(B^{n}\times{}0)}^{\,p+q-1}\,.\]
Thus, the Weyl polynomial
\(W_{\partial{}(B^{n}\times{}0^q)}^{\,p}\) has roots which do not
belong to the imaginary axis. For fixed \(\,q,\,n\) and a positive
\(\varepsilon\), consider the ellipsoid
\(E_{n,\,q,\,\varepsilon}\) defined by \eqref{Ell}. Since
\(E_{n,\,q,\,\varepsilon}\to{}B^{n}\times{}0^q\) as
\(\varepsilon\to+0\), also \(W_{E_{n,\,q,\,\varepsilon}}^p
\to{}W_{\partial{}(B^{n}\times{}0^q)}^{\,p}\) as
\(\varepsilon\to+0\). Hence, if \(\varepsilon\) is small enough:
\(0<\varepsilon\leq{}\varepsilon(n,p,q)\), the polynomial
\(W_{E_{n,\,q,\,\varepsilon}}\) has roots which do not belong to
the imaginary axis.
 \end{proof}

\section{The Steiner polynomial associated with the Cartesian
 product\newline of convex sets.\label{MPCaPr}}
%%%%%%%%%%%%%%%%%%%%%%%%%%%%%%%%%%%%%%%%%%%%%%%%%%%%%%%%%%%%%%%%%%%%%%%%
Let \(V_1\) and \(V_2\) be compact convex sets,
\[V_1\subset{}\mathbb{R}^{n_1},\ \  V_2\subset{}\mathbb{R}^{n_2}.\]
 Then the Cartesian product \(V_1\times{}V_2\)
is a compact convex set embedded in
the Cartesian product \(\mathbb{R}^{n_1}\times{}\mathbb{R}^{n_2}\).
Since \(\mathbb{R}^{n_1}\times{}\mathbb{R}^{n_2}\)
can be naturally identified with \(\mathbb{R}^{n_1+n_2}\),
we can consider \(V_1\times{}V_2\) as being embedded in \(\mathbb{R}^{n_1+n_2}\):
 \[V_1\times{}V_2\subset{}\mathbb{R}^{n_1+n_2}.\]
The natural question arises:\\ %
 \textit{How the Steiner polynomial
\(S_{\,\,V_1\times{}V_2}^{\mathbb{R}^{n_1+n_2}}\) for the
Cartesian product \(V_1\times{}V_2\) can be expressed in terms of the Steiner
polynomials\footnote{The Steiner polynomials
\(S_{V_1},\,S_{V_2},\,S_{V_1\times{} V_2}\) are considered with
respect to the ambient spaces
\(\mathbb{R}^{n_1},\,\mathbb{R}^{n_2},\,\mathbb{R}^{n_1+n_2}\)
respectively.}
\(S_{\,\,V_1}^{\mathbb{R}^{n_1}}\) and \(S_{\,\,V_2}^{\mathbb{R}^{n_2}}\) for the Cartesian factors \(V_1\)  and \(V_2\) ?} \\
To answer this question, we introduce a special multiplication operation in the set
of polynomials, the so-called \textit{M-product}. %
\begin{defn} %
\label{DMMul}%
The \textsf{\textsf{M-product} \(t^k\Mm{}t^l\)}
of two monomials \(t^k\) and \(t^l\)
is defined as
\begin{equation}%
\label{Mmm}%
t^k\Mm{}t^l\stackrel{\textup{\tiny{}def}}{=}
{\frac{\Gamma\big(\frac{k}{2}+1\big)\Gamma\big(\frac{l}{2}+1\big)}%
{\Gamma\big(\frac{k+l}{2}+1\big)}}\,t^{k+l}\,,\ \  k\geq{}0,\,l\geq{}0.
\end{equation}%
\textup{\small It is clear that}
\begin{equation}%
\label{PrMM}
\textup{a). }t^0\Mm{}t^k=t^k,\ \ \textup{b). } t^k\Mm{}t^l=t^l\Mm{}t^k,\ \
\textup{c). }(t^k\Mm{}t^l)\Mm{}t^m=t^k\Mm{}(t^l\Mm{t^m}).
\end{equation}%
{\small{}The M-multiplication \eqref{Mmm} of monomials can
 be extended to the multiplication
of polynomials by linearity:}
\begin{subequations}
\label{MMult}
\begin{gather}%
\label{MMult1}
\textup{For\ \ } A(t)=\sum\limits_{0\leq k\leq n_{{}_{1}}}a_kt^k,\ \ \
B(t)=\sum\limits_{0\leq l\leq n_{{}_{2}}}b_lt^l,\\
(A \,\Mm{}B)(t)=\sum\limits_{\substack{0\leq k\leq n_1,\\
0\leq l\leq n_2}}\hspace*{-1.0ex}a_kb_l(t^k\Mm{}t^l)=
{\sum\limits_{\substack{0\leq k\leq n_1,
0\leq l\leq n_2}}a_kb_l{\frac{\Gamma\big(\frac{k}{2}+1\big)\Gamma\big(\frac{l}{2}+1\big)}%
{\Gamma\big(\frac{k+l}{2}+1\big)}}\,t^{k+l}}.
\notag
\end{gather}
Finally, the \textsf{M-product \(A\Mm{}B\)
of the polynomials} \(A\) and \(B\)
 is defined as
\begin{equation}%
\label{MMult2}
(A \,\Mm{}B)(t)=\sum\limits_{0\leq r\leq{}n_1+n_2}
\bigg(\sum\limits_{\substack{k\geq{}0,\,l\geq{}0,
k+l=r}}\hspace*{-1.5ex}
a_k\,b_l\,{\frac{\Gamma\big(\frac{k}{2}+1\big)\Gamma\big(\frac{l}{2}+1\big)}%
{\Gamma\big(\frac{k+l}{2}+1\big)}}\bigg)\,t^r\,.
\end{equation}%
\end{subequations}
\end{defn} %
\vspace{4.0ex}
\noindent
From (\ref{PrMM}.b) and (\ref{PrMM}.c) it follows that
\[A\Mm{}B=B\Mm{}A,\quad (A\Mm{}B)\Mm{}C=A\Mm(B\Mm{}C)
\]
for any polynomials \(A,\,B,\,C\). In particular, the  product
\(A\Mm{}B\Mm{}C\) of \(A,\,B,\,C\) is well defined. This product can
be explicitly expressed in terms of the coefficients of its factors: if
\begin{equation*}
A(t)=\sum\limits_{0\leq k\leq n_1}a_kt^k,\ \
B(t)=\sum\limits_{0\leq l\leq n_2}b_lt^l,\ \
C(t)=\sum\limits_{0\leq m\leq n_3}c_mt^m,\ \
\end{equation*}
then
\begin{equation*}
(A\Mm{}B\Mm{}C)(t)=\hspace*{-3.5ex}
\sum\limits_{0\leq r\leq{}n_1+n_2+n_3}\hspace*{-1.0ex}
\bigg(\sum\limits_{\substack{k\geq{}0,\,l\geq{}0,\,m\geq{}0\\
k+l+m=r}}\hspace*{-2.5ex}
a_k\,b_l\,c_m{\frac{\Gamma\big(\frac{k}{2}+1\big)\Gamma\big(\frac{l}{2}+1\big)%
\Gamma\big(\frac{m}{2}+1\big)}%
{\Gamma\big(\frac{k+l+m}{2}+1\big)}}\bigg)\,t^r\,.
\end{equation*}
It is clear that for every number \(\lambda\)
and for any polynomials \(A\) and \(B\),
\[(\lambda{}A)\Mm{}B=\lambda(A\Mm{}B).\]
Moreover, if
\[\mathbb{I}(t)\equiv 1,\ \ \mathbb{T}(t)\equiv{}t,\] then
\[\mathbb{I}\Mm{}A=A.\]
Thus, \textit{the polynomial \(\mathbb{I}\) is the identity element
for the \(M\)-product}\\
It is worth mentioning that
\begin{equation}%
\label{Pon}%
t^{(\Mm{}k)}\stackrel{\textup{\tiny{}def}}{=}
\underbrace{t\Mm{}t\Mm{}\,\cdots\,\Mm{}t}_{k}=
\frac{(\sqrt{\pi}/2)^k}{\Gamma(\frac{k}{2}+1)}t^k.
\end{equation}%
\begin{rem} %
M-multiplication
by the polynomial \(\mathbb{T}\)
 is related to the transformation \eqref{MTP}:
\begin{subequations}
\label{UInt}
\begin{align}
\textup{If }\hspace{22.0ex} A(t)&=\phantom{2^{-p}t^p}\sum\limits_{0\leq k\leq n} a_kt^k,\\
\textup{then }\hspace{5.5ex}%
(\underbrace{\mathbb{T}\Mm{}\,\cdots\,\Mm{}\mathbb{T}}_{p}\Mm{}A)(t)
&=2^{-p}t^p\sum\limits_{0\leq k\leq n} a_k\gamma_{\,\,k}^{(p)}t^k,
\end{align}
\end{subequations}
where the `multipliers' \(\gamma_{\,\,k}^{(p)}\) are defined by \eqref{MuSe}.\\
\end{rem}
\begin{lem} %
\label{InReMP}
The M-product \(A\Mm{}B\) of polynomials \(A\) and \(B\) admits
the integral\footnote{The integrals on the right hand sides of \eqref{IRTP}
are Stieltjes integrals.} representation\footnote{At least, for \(t>0\).}
\begin{subequations}
\label{IRTP}
\begin{equation}
\label{IRTP1}
(A\Mm{}B)(t)=A(0)B(t)+\int\limits_{0}^{t}A\big((t^2-\tau^2)^{1/2}\big)\,dB(\tau)\,,
\end{equation}
as well as
\begin{equation}
\label{IRTP2}
(A\Mm{}B)(t)=A(t)B(0)+\int\limits_{0}^{t}B\big((t^2-\tau^2)^{1/2}\big)\,dA(\tau)\,.
\end{equation}
\end{subequations}
\end{lem}
\begin{proof}
 First of all, the expressions on the right hand sides of
\eqref{IRTP} are equal: Integrating by parts and replacing the
variable \(\tau\to(t^2-\tau^2)^{1/2}\), we obtain
\[A(0)B(t)+\int\limits_{0}^{t}A\big((t^2-\tau^2)^{1/2}\big)\,dB(\tau)=
A(t)B(0)+\int\limits_{0}^{t}B\big((t^2-\tau^2)^{1/2}\big)\,dA(\tau).\]
The expressions on the right hand sides
of \eqref{IRTP}, which, at the first glance, are asymmetric with respect
 to \(A\) and \(B\), are therefore actually symmetric. Let
 \begin{equation*}
 \label{ETC}
 A(t)=\sum\limits_{0\leq k\leq n_{{}_{1}}}a_kt^k,\ \ \
B(t)=\sum\limits_{0\leq l\leq n_{{}_{2}}}b_lt^l
\end{equation*}
be the expressions for the polynomials \(A\) and \(B\) in terms of their coefficients.
Let us substitute these polynomials into the right hand side
of \eqref{IRTP1}:
\begin{gather*}
A(0)B(t)+\int\limits_{0}^{t}A\big((t^2-\tau^2)^{1/2}\big)\,dB(\tau)=\\
a_0\sum\limits_{0\leq l\leq n_{{}_{2}}}b_lt^l+
\int\limits_0^t\bigg(\sum\limits_{0\leq k\leq n_{{}_{1}}}%
a_k(t^2-\tau^2)^{k/2}\bigg)\cdot %
\bigg(\sum\limits_{1\leq l\leq n_{{}_{2}}}l\,b_{l}\tau^{l-1}\bigg)\,d\tau=
\notag\\
\sum\limits_{0\leq l\leq n_{{}_{2}}}a_0b_lt^l+
\sum\limits_{\substack{0\leq k\leq{}n_1\\
1\leq{}l{}\leq{}n_2}}a_k{}b_l\cdot{}l\int\limits_0^t(t^2-\tau^2)^{k/2}\tau^{l-1}d\tau\,.
\notag
\end{gather*}
Changing variable \(\tau\to{}t\tau^{1/2}\), we get
\begin{multline*}%
l\int\limits_0^t(t^2-\tau^2)^{k/2}\tau^{l-1}d\tau
=t^{k+l}(l/2)\int\limits_0^1(1-\tau)^{k/2}\tau^{l/2-1}d\tau
=\\
t^{k+l}(l/2) B\Big(\frac{k}{2}+1;\frac{l}{2}\Big)\,.
\end{multline*}%
Now, according to Euler,
\[(l/2) B\Big(\frac{k}{2}+1;\frac{l}{2}\Big)=
\frac{\Gamma\big(\frac{k}{2}+1\big)\frac{l}{2}\Gamma\big(\frac{l}{2}\big)}%
{\Gamma\big(\frac{k+l}{2}+1\big)}=
\frac{\Gamma\big(\frac{k}{2}+1\big)\Gamma\big(\frac{l}{2}+1\big)}%
{\Gamma\big(\frac{k+l}{2}+1\big)}\,.\] Thus, the right hand side
of \eqref{IRTP1} can be transformed into the right hand side of
\eqref{MMult2}.
\end{proof}
\begin{thm}
\label{MPCarP}
Given the compact convex sets \(V_1\) and \(V_2\),
\(V_1\subset{}\mathbb{R}^{n_1},  V_2\subset{}\mathbb{R}^{n_2},\)
let \(S_{\,V_1}^{\mathbb{R}^{n_1}}(t),\,S_{\,V_2}^{\mathbb{R}^{n_2}}(t)\)
be the Steiner polynomials for the sets
\(V_1\) and \(V_2\). Then the Steiner polynomial
\(S_{\,\,V_1\times{}V_2}^{\mathbb{R}^{n_1+n_2}}\) of the Cartesian product \(V_1\times{}V_2\)
 is equal to the M-product
of  the polynomials \(S_{V_1}^{\mathbb{R}^{n_1}}\) and \(S_{V_2}^{\mathbb{R}^{n_12}}\):
\begin{equation}
\label{MPCP}%
S_{\,\,V_1\times{}V_2}^{\mathbb{R}^{n_1+n_2}}
=S_{\,V_1}^{\mathbb{R}^{n_1}}\Mm{}S_{\,V_2}^{\mathbb{R}^{n_2}}\,.
\end{equation}%
\vspace{1.0ex}
\end{thm}
\noindent
A sketch of the proof for this theorem can be found in \cite[Chapter VI, Section 6.1.9]{Had}.
A detailed proof is presented below.
\begin{rem} Let \(S\) be the subset of \(\mathbb{R}^1\) containing only
the origin point, i.e.,
 \(S=\{t\in\mathbb{R}^1:t=0\}\).
Then \(S_S(t)=2t\), that is,
\begin{equation}%
\label{MOPS}
S_S(t)=2\,\mathbb{T}(t).
\end{equation}%
Let \(V\) be a compact convex set embedded into \(\mathbb{R}^{n}\).
 The Cartesian product \(V\times{}\underbrace{S\,\times\,\cdots\,\times\,S}_p\)
 can be identified with the
 convex set
\(V\times{}0^p,\,\, V\times{}0^p\subset\mathbb{R}^{n+p}\).
Thus,
 \[S_{V\times{}\underbrace{\scriptstyle{}S\,\times\,\cdots\,\times\,S}_p}(t)
 =S_{\,\,V\times{}0^p}^{\mathbb{R}^{n+p}}(t)\,,\]
 or
 \begin{equation}
 \label{OInt}%
 2^p(\underbrace{\mathbb{T}\Mm\,\cdots\,\Mm{}\mathbb{T}}_p)\Mm{}S_{\,V}^{\mathbb{R}^{n}}
 =S_{\,\,V\times{}0^p}^{\mathbb{R}^{n+p}}.
 \end{equation}
In view of \eqref{UInt} and \eqref{MOPS}, the equality \eqref{OInt}
 is another form of the equality \eqref{MPn2}.\\
\end{rem}
\begin{proof}[Proof of Theorem \ref{MPCarP}] Let
\[V=V_1\times{}V_2\,.\]
 According to the
identification
\(\mathbb{R}^{n_1+n_2}=\mathbb{R}^{n_1}\times{}\mathbb{R}^{n_2}\),
we present a point \(x\in\mathbb{R}^{n_1+n_2}\) as a pair
\(x=(x_1,\,x_2)\), where \(x_1\in\mathbb{R}^{n_1},\
x_2\in\mathbb{R}^{n_2}\). It is clear that
\begin{equation}%
\label{PiThe}%
 \dist_{\mathbb{R}^{n_1+n_2}}^2(x,\,V)=
\dist_{\mathbb{R}^{n_1}}^2(x_1,\,V_1)+\dist_{\mathbb{R}^{n_2}}^2(x_2,\,V_2).
\end{equation}
For \(\tau>0,\,\tau^{\prime}>0,\,\tau^{\prime\prime}>0\), let
\(V(\tau), V_1(\tau^{\prime})\) and \(V_2(\tau^{\prime\prime})\)
be the \(\tau\)-neighborhood of \(V\) with respect to
\(\mathbb{R}^{n_1+n_2}\), the \(\tau^{\prime}\)-neighborhood of
\(V_1\) w.\,r.\,t. \(\mathbb{R}^{n_1}\) and
\(\tau^{\prime\prime}\)-neighborhood of \(V_2\) w.\,r.\,t.
\(\mathbb{R}^{n_2}\) respectively:
\begin{gather*}
V(\tau)=V+\tau{}B^{n_1+n_2},\ \
V_1(\tau^{\prime})=V_1+{\tau}^{\prime}B^{n_1},\ \ %
V_2(\tau^{\prime\prime})=V_2+\tau^{\prime\prime}B^{n_2}; \\
V,\,B^{n_1+n_2}\subset\mathbb{R}^{n_1+n_2};\quad{}
V_1,\,B^{n_1}\subset\mathbb{R}^{n_1};\quad{}
V_2,\,B^{n_2}\subset\mathbb{R}^{n_2}.
\end{gather*}
Here \(B^{n_1+n_2}\), \(B^{n_1}\), \(B^{n_2}\) denote
the Euclidean balls of the radius one in
\(\mathbb{R}^{n_1+n_2}\), \(\mathbb{R}^{n_1}\), \(\mathbb{R}^{n_2}\),
 respectively.

Given a number \(t,\,t>0\), consider the \(t\)-neighborhood
\(V(t)\) of \(V=V_{1}\times{}V_2\), and let
\[0=\tau_0<\tau_1<\,\dots\,<\tau_{N-1}<\tau_N=t\]
be a subdivision of the interval \([0,\,t]\). From \eqref{PiThe} it
follows that
\begin{align}%
\label{DoIn}%
\big(V_1(0)\times{}V_2(t)\big)%
\cup&\Big(\bigcup_{1\leq{}k\leq{}N}\big(V_1(\tau_k)%
\setminus{}V_1(\tau_{k-1})\big)\times{}V_2\big((t^2-\tau_{k}^2)^{1/2
}\big)\Big)
\notag\\
&\subseteq
V(t)\subseteq\\
\big(V_1(0)\times{}V_2(t)\big)%
\cup&\Big(\bigcup_{1\leq{}k\leq{}N}\big(V_1(\tau_k)
\setminus{}V_1(\tau_{k-1})\big)\times{}V_2\,%
\big((t^2-\tau_{k-1}^2)^{1/2}\big)\Big)\notag\,.
\end{align}
Since \(V_1(\tau_k)\supseteq{}V_1(\tau_{k-1})\),
\[\textup{Vol}_{n_1}\big(V_1(\tau_k)
\setminus{}V_1(\tau_{k-1}\big)=
{\textup{Vol}}_{n_1}\big(V_1(\tau_k)\big)-\textup{Vol}_{n_1}\big(V_1(\tau_{k-1})\big),\]
thus
\begin{multline*}
\textup{Vol}_{\,n_1+n_2}\Big(\big(V_1(\tau_k)
\setminus{}V_1(\tau_{k-1})\big)\times{}V_2\,%
\big((t^2-\tau_{\,l}^2)^{1/2}\big)\Big)=\\
\Big({\textup{Vol}}_{n_1}\big(V_1(\tau_k)\big)-
\textup{Vol}_{n_1}\big(V_1(\tau_{k-1})\big)\Big)\cdot
\textup{Vol}_{n_2}\Big(V_2\big((t^2-\tau_{\,l}^2)^{1/2}\big)\Big),\
\\
 l=k-1\textup{ or }l=k.
\end{multline*}
Moreover,
\[\textup{Vol}_{\,n_1+n_2}\Big(V_1(0)\times{}V_2(t)\Big)=
\textup{Vol}_{\,n_1}\big(V_1(0)\big)\cdot{}\textup{Vol}_{\,n_2}\big(V_2(t)\big)\,.\]
In the notation of Steiner polynomials, the latter equalities take the form
\begin{subequations}
\label{PrFo}
\begin{gather}
\textup{Vol}_{\,n_1+n_2}\Big(\big(V_1(\tau_k)
\setminus{}V_1(\tau_{k-1})\big)\times{}%
V_2\,\big((t^2-\tau_{\,l}^2)^{1/2}\big)\Big)=\notag{}\\
\Big(S_{V_1}(\tau_k)-S_{V_1}(\tau_{k-1})\Big)\cdot{}S_{V_2}\big((t^2-\tau_{\,l}^2)^{1/2}\big),
\quad{}l=k-1\textup{ or }l=k,\\
\textup{Vol}_{\,n_1+n_2}\Big(V_1(0)\times{}V_2(t)\Big)=S_{V_1}(0)\cdot{}S_{V_2}(t)\,,
\end{gather}
and
\begin{equation}
\textup{Vol}_{\,n_1+n_2}\big(V(t)\big)=S_{V}(t).
\end{equation}
\end{subequations}
Since the sets \(V_1(\tau_k)\setminus{}V_1(\tau_{k-1})\) do not intersect
for pairwise different \(k\), and none of these sets intersects with the set
\(V_1(0)\),
it follows from \eqref{DoIn} and \eqref{PrFo}
that
\begin{gather}
S_{V_1}(0)\cdot{}S_{V_2}(t)+
\sum\limits_{1\leq{}k\leq{}N}\Big(S_{V_1}(\tau_k)-S_{V_1}(\tau_{k-1})\Big)%
\cdot{}S_{V_2}\big((t^2-\tau_{k}^2)^{1/2}\big)\notag{}\\
\label{FBoS}%
\leq{}S_{V}(t)\leq{}\\
S_{V_1}(0)\cdot{}S_{V_2}(t)+
\sum\limits_{1\leq{}k\leq{}N}\Big(S_{V_1}(\tau_k)-S_{V_1}(\tau_{k-1})\Big)%
\cdot{}S_{V_2}\big((t^2-\tau_{k-1}^2)^{1/2}\big)\notag{}\,.
\end{gather}
Passing to the limit as \(\max{(\tau_{k}-\tau_{k-1})}\to{}0\) in
the latter inequality, we express the Steiner polynomial
\(S_V(t)\) as the Stieltjes integral
\begin{equation}
\label{IRepr}
S_V(t)=S_{V_1}(0)\cdot{}S_{V_2}(t)+%
\int\limits_{0}^tS_{V_2}\big((t^2-\tau^2)^{1/2}\big)\,dS_{V_1}(\tau)\,.
\end{equation}
According to Lemma \ref{InReMP}, the expression on the right hand
side of \eqref{IRepr} is equal to
\(\big(S_{V_1}\Mm{}S_{V_2}\big)(t)\).\end{proof}
%%%%%%%%%%%%%%%%%%%%%%%%%%%%%%%%%%%%%%%%%%%%%%%%%%%%%%%%%%%%%%%%%%%%%%%%

\section{
 Properties of entire functions generating
  the Steiner and Weyl
polynomials for the degenerate convex sets
\(\boldsymbol{B^{n+1}\times{}0^q}\). \label{LocRoot}}
%%%%%%%%%%%%%%%%%%%%%%%%%%%%%%%%%%%%%%%%%%%%%%%%%%%%%%%%%%%%%%%%%%%%%%%%
%%%%%%%%%%%%%%%%%%%%%%%%%%%%%%%%%%%%%%%%%%%%%%%%%%%%%%%%%%%%%%%
In this section we investigate the locations of roots for entire
functions generating the Steiner and Weyl polynomials related to
the `degenerate' convex sets \(B^{n+1}\times{}0^q\). Entire
functions of this kind are
\begin{itemize}
\item
The entire functions \(\mathcal{M}_{B^{\infty}\times{}0^q}\) which
appear in \eqref{JPMEb}, in particular for \(q=1\) in
\eqref{LiEFMPS2} .
\item
The entire function
\(\mathcal{W}_{\partial{}(B^{\infty}\times{}0)}^{\,p},\ \
1\leq{}p<\infty\), which appears in \eqref{ReWP3}\,.
\item
The entire function
\(\mathcal{W}_{\partial{}(B^{\infty}\times{}0)}^{\,\infty}\) which
appears in \eqref{ReWP4}\,.
\end{itemize}
The functions \(\mathcal{M}_{B^{\infty}\times{}0^q}\),
\(\mathcal{W}_{\partial{}(B^{\infty}\times{}0)}^{\,p},\ \
1\leq{}p<\infty\) can not be determined explicitly (except for
some special values of the parameters \(p\) or \(q\)), but they
can be calculated asymptotically.

The above-mentioned functions admit integrable representations:
\begin{equation}%
\label{MInt}%
\mathcal{M}_{B^{\infty}\times{}0^q}(t)
=q\int\limits_{0}^{1}(1-\xi^2)^{\frac{q}{2}-1}\xi\,{}e^{\xi{}t}\,d\xi\,;
\end{equation}%
\begin{equation}%
\label{WInt}%
\mathcal{W}_{\partial{}(B^{\infty}\times{}0)}^{\,p}(t)= %
p\int\limits_{0}^{1}(1-\xi^2)^{\frac{p}{2}-1}\xi\cos{}t\xi\,d\xi\,;
\end{equation}%
These integral representations can be used for the asymptotic
calculation of the functions
\(\mathcal{M}_{B^{\infty}\times{}0^q}\),
\(\mathcal{W}_{\partial{}(B^{\infty}\times{}0)}^{\,p}\).

Another way to calculate the functions
\eqref{MInt},\,\eqref{WInt} asymptotically is to use the structure
of their respective Taylor series. The Taylor coefficients of each of these
functions are ratios of factorials. These functions belong to the
so-called class of \textit{Fox-Wright functions}, \cite{CrCs4}.

\textsf{The Fox-Wright class of functions} is defined as a class
of functions of the form
\begin{equation}
\label{FoWr}%
 {\sideset{_p}{_q}\Psii}%
\left\{\genfrac{}{}{0pt}{1}%
{\alpha_1\,\,\alpha_2\,\,\,\,\,\alpha_p}%
{\beta_1\,\,\beta_2\,\,\,\,\,\beta_p};%
\genfrac{}{}{0pt}{1}%
{\rho_1\,\,\rho_2\,\,\,\,\,\rho_q}%
{\sigma_1\,\,\sigma_2\,\,\,\,\,\sigma_q};%
z\right\}%
 \stackrel{\textup{\tiny
def}}{=}\sum\limits_{0\leq{}k<\infty}%
\frac{\prod_{j=1}^{p}\Gamma{}(\alpha_{j}\,k+\beta_j)}%
{\prod_{j=1}^{q}\Gamma{}(\rho_{j}\,k+\sigma_j)}%
\cdot\frac{z^k}{k!}\,.
\end{equation}

Comparing \eqref{FoWr} with the Taylor expansions
\eqref{JPMEb},\,\eqref{ReWP3},\,\eqref{ReWP4},\,we see that
\begin{equation}%
\label{FRMq}%
 \mathcal{M}_{B^{\infty}\times{}0^q}(t)=\Gamma\left(\frac{q}{2}+1\right)\cdot%
 {\sideset{_1}{_1}\Psii}%
\left\{\genfrac{}{}{0pt}{1}%
{\frac{1}{2}}%
{1};%
\genfrac{}{}{0pt}{1}%
{\frac{1}{2}}%
{1+\frac{q}{2}};%
t\right\}\,,\ \ 1\leq{}q<\infty\,,%
\end{equation}%
\begin{equation}
\label{As1to}
\mathcal{W}^{\,p}_{\partial{}(B^{\infty}\times{}0^{1})}(t)=
\Gamma({\textstyle{}\frac{1}{2}})\cdot\Gamma({\textstyle{}\frac{p}{2}+1})\cdot%
{\sideset{_1}{_2}\Psii}%
\left\{\genfrac{}{}{0pt}{1}%
{1}%
{1};%
\genfrac{}{}{0pt}{1}%
{1\,\,\,\,\,\,\,\,1\,\,\,\,\,}%
{\frac{1}{2}\,\,\frac{p}{2}+1};%
-\frac{t^2}{4}\right\}\,,\ \ 1\leq{}p<\infty\,,%
\end{equation}
E.\,Barnes, \cite{Bar}, G.N.\,Watson, \cite{Wat}, G.\,Fox, cite{Fox},
as well as E.M.\,Wright, \cite{Wr1}, \cite{Wr2}, have
all studied the asymptotic behavior of the function \({\sideset{_p}{_q}\Psii}(z)\).\\

\paragraph{Analysis of the function
\(\mathcal{M}_{B^{\infty}\times{}0^q}(t)\):} \ We would like to
determine for which \(q\) this function belongs to the Hurwitz
class \(\mathscr{H}\). According to \eqref{FRMq}, we may reduce
the question to the proportional function \({\sideset{_1}{_1}\Psii}%
\left\{\genfrac{}{}{0pt}{1}%
{\frac{1}{2}}%
{1};%
\genfrac{}{}{0pt}{1}%
{\frac{1}{2}}%
{1+\frac{q}{2}};%
z\right\}\). From the Taylor expansion it is clear that this
function is an entire function of exponential type. Since the
Taylor coefficients of the function are positive, its defect\,%
\footnote{We recall that the defect of an entire function \(H\) of
exponential type is defined by \eqref{DefH}.} %
 is non-negative. \textit{The function \(\mathcal{M}_{B^{\infty}\times{}0^q}(t)\)
 is therefore in the Hurwitz class \(\mathscr{H}\)  only if all roots
 of the function %
 \({\sideset{_1}{_1}\Psii}%
\left\{\genfrac{}{}{0pt}{1}%
{\frac{1}{2}}%
{1};%
\genfrac{}{}{0pt}{1}%
{\frac{1}{2}}%
{1+\frac{q}{2}};%
z\right\}\) %
are situated in the open left half plane.} To investigate the
locations of roots for the latter function, we use the following
asymptotic approximation, which can be derived from  results
in \cite{Wr1}, \cite{Wr2}:\\
For any \(\varepsilon: \ 0<\varepsilon<\dfrac{\pi}{2}\),\\[-3.0ex]
\begin{multline}
\label{Asymp1}
{\sideset{_1}{_1}\Psii}%
\left\{\genfrac{}{}{0pt}{1}%
{\frac{1}{2}}%
{1};%
\genfrac{}{}{0pt}{1}%
{\frac{1}{2}}%
{1+\frac{q}{2}};%
z\right\}=\\
=\begin{cases}
2^{\frac{q}{2}}z^{-\frac{q}{2}}e^z\left(1+r_1(z)\right)\,,&|\arg{}z|\leq{}\frac{\pi}{2}-\varepsilon\,;\\[1.0ex]
\frac{2}{\Gamma(\frac{q}{2})}z^{-2}+r_2(z),\,&|\arg{}z-\pi|\leq{}\frac{\pi}{2}-\varepsilon\,;\\[1.0ex]
2^{\frac{q}{2}}z^{-\frac{q}{2}}e^z+
\frac{2}{\Gamma(\frac{q}{2})}z^{-2}+r_3(z)\,,&|\arg{}z\mp\frac{\pi}{2}|\leq{}\varepsilon\,.
\end{cases}
\end{multline}
The remainders can be estimated as follows:
\begin{multline}%
\label{Rema1}%
 |r_1(z)|\leq{}C_1(\varepsilon)|z|^{-1},
|\arg{}z|\leq\frac{\pi}{2}-\varepsilon\,;\\
 |r_2(z)|\leq{}C_2(\varepsilon)|z|^{-3},
|\arg{}z-\pi|\leq\frac{\pi}{2}-\varepsilon\,;\hspace*{15.0ex}\\
|r_3(z)|\leq{}C_3\left(|z|^{-3}+|e^z||z|^{-(\frac{q}{2}+1)}\right)\,,\
\ |\arg{}z\mp\frac{\pi}{2}|\leq{}\varepsilon\,,
\end{multline}
where the values \(C_1(\varepsilon)<+\infty, \
C_2(\varepsilon)<+\infty, \ C_3(\varepsilon)<+\infty\) do not
depend on \(z\).

From \eqref{Asymp1}, \eqref{Rema1} it follows
that for any \(\varepsilon>0\), not more than finitely many roots
of the function %
\({\sideset{_1}{_1}\Psii}%
\left\{\genfrac{}{}{0pt}{1}%
{\frac{1}{2}}%
{1};%
\genfrac{}{}{0pt}{1}%
{\frac{1}{2}}%
{1+\frac{q}{2}};%
z\right\}\) %
lie  outside of the
angular domain
\(\{z:\,|\arg{}z\mp{}\frac{\pi}{2}|\leq{}\varepsilon\).  These roots are
asymptotically  close to the roots of the approximating function \(f_q(z)\):
\[f_q(z)=2^{\frac{q}{2}}z^{-\frac{q}{2}}\left(e^z+
\frac{2^{1-\frac{q}{2}}}{\Gamma(\frac{q}{2})}z^{\frac{q}{2}-2}\right).\]
One should distinguish several cases when examining the roots of the
approximating function \(f_q(z)\). \\ %
\(\boldsymbol{q=4}\). In this case, the equation \(f_q(z)=0\) becomes
\(e^z+\frac{1}{2}=0\), so, the roots of the
approximating function can be found explicitly. These root form an
arithmetical progression located on the straight line
\(\{z=x+iy:\,x=-\ln{}2,\,-\infty<y<\infty\}\). From this and
\eqref{Asymp1}-\eqref{Rema1} it follows that the roots of the
the analyzed function %
\({\sideset{_1}{_1}\Psii}%
\left\{\genfrac{}{}{0pt}{1}%
{\frac{1}{2}}%
{1};%
\genfrac{}{}{0pt}{1}%
{\frac{1}{2}}%
{1+\frac{q}{2}};%
z\right\}\) %
are asymptotically close to the above-mentioned straight
line. Thus, \textit{for \(q=4\) all  but finitely many
 roots of the function
\({\sideset{_1}{_1}\Psii}%
\left\{\genfrac{}{}{0pt}{1}%
{\frac{1}{2}}%
{1};%
\genfrac{}{}{0pt}{1}%
{\frac{1}{2}}%
{1+\frac{q}{2}};%
z\right\}\) are located in the open left half
plane.} In fact, for \(q=4\) \textsf{\textit{all}} roots of this
function are located in the open left half plane. To establish
this, further considerations are needed. We follow up on this later.\\
\(\boldsymbol{q\not=4}\). In this case, the equation \(f_q(z)=0\)
becomes the equation %
\[e^z+
c_q\,z^{\frac{q}{2}-2}=0\,,\ \ \
c_q={\frac{2^{1-\frac{q}{2}}}{\Gamma(\frac{q}{2})}}\,,\] where the
exponent \(\frac{q}{2}-2\) is non-zero. The latter
equation has infinitely many roots which have no finite
accumulation points and which are asymptotically close to the
`logarithmic parabola'
\begin{equation}
\label{LogPar}%
 x=({\textstyle\frac{q}{2}}-1)\ln{}(|y|+1)+ln|c|,\ -\infty<y<\infty\,,\ \
 (z=x+iy).
\end{equation}
From this and from \eqref{Asymp1}-\eqref{Rema1} it follows that
the roots of the
function\\
 \({\sideset{_1}{_1}\Psii}%
\left\{\genfrac{}{}{0pt}{1}%
{\frac{1}{2}}%
{1};%
\genfrac{}{}{0pt}{1}%
{\frac{1}{2}}%
{1+\frac{q}{2}};%
z\right\}\)  are asymptotically close to the logarithmic parabola
\eqref{LogPar}. Now we should distinguish the cases \(q<4\) and
\(q>4\).\\
\(\boldsymbol{q<4}\). In this case, the part of the logarithmic
parabola \eqref{LogPar} lying outside some compact set is located
inside the left half plane. Since the roots of the analyzed
function are asymptotically close to this parabola, all but
finitely many roots are located in
the left half plane. \\
\(\boldsymbol{q>4}\). In this case,  the part of the logarithmic
parabola \eqref{LogPar} lying outside some compact set is located
inside the right half plane. Therefore all but finitely many roots
of the function
\({\sideset{_1}{_1}\Psii}%
\left\{\genfrac{}{}{0pt}{1}%
{\frac{1}{2}}%
{1};%
\genfrac{}{}{0pt}{1}%
{\frac{1}{2}}%
{1+\frac{q}{2}};%
z\right\}\) are located in the right half
plane.\\

We formulate this result as
\begin{lem}
\label{NonHq} If \(q>4\), then the entire function
\(\mathcal{M}_{B^{\infty}\times{}0^q}\) has infinitely many roots
within the right half plane. In particular, this function does not
belong to the Hurwitz class \(\mathscr{H}\).
\end{lem}
Claim 2 of Lemma \ref{HPEGM} is a consequence of Lemma
\ref{NonHq}.
\begin{lem}%
\label{HurCl}%
For \(q:\,0\leq{}q\leq{}2\), the function
\(\mathcal{M}_{B^{\infty}\times{}0^q}\) belongs to the Hurwitz
class \(\mathscr{H}\).
\end{lem}
\begin{proof}
 For \(q=0\), the assertion is clear: the function in
question is equal to \(e^z\). For \(q>0\), Lemma \ref{HurCl} is a
consequence of Lemma \ref{NZRHP}. \end{proof}

To investigate the case \(q>0\), we use the integral
representation
\begin{equation}
\label{IRPs1}
{\sideset{_1}{_1}\Psii}%
\left\{\genfrac{}{}{0pt}{1}%
{\frac{1}{2}}%
{1};%
\genfrac{}{}{0pt}{1}%
{\frac{1}{2}}%
{1+\frac{q}{2}};%
z\right\}=\frac{q}{\Gamma(\frac{q}{2}+1)}\,I_q(z)\,,
\end{equation}
where
\begin{equation}
\label{IRPs2}
I_q(z)=\int\limits_{0}^{1}(1-\xi^2)^{\frac{q}{2}-1}\xi\,{}e^{\xi{}t}\,d\xi\,.
\end{equation}
The defect of the entire function \({\sideset{_1}{_1}\Psii}%
\left\{\genfrac{}{}{0pt}{1}%
{\frac{1}{2}}%
{1};%
\genfrac{}{}{0pt}{1}%
{\frac{1}{2}}%
{1+\frac{q}{2}};%
z\right\}\) is non-negative, so it is enough to prove that this
function has no roots in the closed right half plane. The function
\(I_q(z)\) is of the form
\begin{equation}
\label{IRPs3}
I_q(z)=\int\limits_{0}^{1}\varphi_q(\xi)e^{\xi{}z}d\xi\,,
\end{equation}
where
\begin{equation}
\label{IRPs4} \varphi_q(\xi)=(1-\xi^2)^{\frac{q}{2}-1}\xi,\ \
0\leq{}\xi\leq{}1\,.
\end{equation}
The crucial case is: \\
\textit{For \(q:\ 0\leq{}q\leq{}2\), the function
\(\varphi_q(\xi)\) is positive and strictly increasing on the
interval \((0,1).\)}
\begin{lem}
\label{NZRHP}
\textsf{\textup{[P\'olya]}} %
\label{MoImHu}%
If \(\varphi(\xi)\) is a non-negative increasing function on the
interval \([0,1]\), then the entire function
\begin{equation}
\label{IRPs5}%
 I(z)=\int\limits_0^1\varphi(\xi)e^{\xi{}z}d\xi
\end{equation}
has no zeros in the closed right half plane.
\end{lem}%
 This lemma is a continuous generalization of a theorem by
S.\,Kakeya. Proof of this Lemma and the reference to the paper of
S.\,Kakeya could be found in \cite[\S \,1]{Po1}. We give another
proof, based on an approach suggested by \cite[Lemma 4]{OsPe}.
\begin{proof}[Proof of Lemma \ref{NZRHP}] Let \(z=x+iy\). Since \(I(x)>0\)
for \(x\geq{}0\), \(I(x)\) has no zeros for \(0\leq{}x<\infty\).
Let us show that \(I(z)\) has no zeros for
\(0\leq{}x<\infty,\,y\not=0\). It is enough to consider the case
\(y>0\). We prove that \(\textup{Im}\,(e^{-z}I(z))<0\) for
\(z=x+iy,\,\,x\geq{}0,\,y>0\), thus \(I(z)\not=0\) for
\(z=x+iy,\,\,x\geq{}0,\,y>0\). To prove this, we use the integral
representation
\begin{equation}
\label{NIPa}
e^{-z}I(z)=\int\limits_0^{\infty}\psi(\xi)e^{-i\xi{}y}d\xi\,,
\end{equation}
where
\begin{equation}
\label{IRPoIP} \psi(\xi)=
\begin{cases}
\varphi(1-\xi)e^{-x\xi}, & 0\leq{}\xi\leq{}1\,,\\
0,&1<\xi<\infty\,.
\end{cases}
\end{equation}
In particular,
\begin{equation}
\label{IRPoIP1}
-(\textup{Im}\,e^{-z}I(z))=\int\limits_0^{\infty}\psi(\xi)\sin{}\xi{}y{}\,d\xi\,,
\ \ z=x+iy,\ \ x\geq{ }0,\,y>0\,,
\end{equation}
where the function \(\psi(\xi)\) is non-negative and decreasing on
\([0,\infty)\),  strictly decreasing on some non-empty open
interval, and \(\psi(\infty)=0\). Further,
\begin{multline}
\label{Fur}
\int\limits_0^{\infty}\psi(\xi)\sin{}\xi{}y{}\,d\xi=\sum\limits_{k=0}^{\infty}\,
\int\limits_{\frac{k\pi}{y}}^{\frac{(k+1)\pi}{y}}\psi(\xi)\sin{}\xi{}y{}\,d\xi=\\
\int\limits_{0}^{\frac{\pi}{y}}\left(\sum\limits_{k=0}^{\infty}
(-1)^k\psi(\xi+{\textstyle\frac{k\pi}{y}})\right)\sin{}\xi{}y{}\,d\xi>0\,:
\end{multline}
The series in the latter integral is a Leibnitz type series. Thus
the sum of this series is non-negative on the interval of
integration and is strictly positive on some subinterval.
\end{proof}
\begin{lem}%
\label{QEqFour}%
For \(q=4\), the function \(\mathcal{M}_{B^{\infty}\times{}0^q}\)
 belongs to the Hurwitz class \(\mathscr{H}\).
\end{lem}%
\begin{proof}%
 For \(q=4\), the integral in \eqref{IRPs2} can be calculated
explicitly:
\begin{equation}
\label{ExpCal4}%
 I_4(z)=\frac{(2z^2-6z+6)e^z+(z^2-6)}{z^4}\,\cdot
\end{equation}
Our goal is to prove that the function
\(\frac{(2z^2-6z+6)e^z+(z^2-6)}{z^4}\) has no roots in the closed
right half plane. Instead of considering of this function, we first
focus on the function
\begin{equation}
\label{AuxFunc1}%
 f(z)=(2z^2-6z+6)+(z^2-6)e^{-z}\,.
\end{equation}
We will now prove that \(f(z)\) has no roots in the
closed right half plane, other than the root at the point \(z=0\) of
multiplicity four. The function \(f\) is of the form
\begin{equation}
\label{AuxFunc2}%
 f(z)=g(z)+h(z)\,, \ \textup{where} \ \ g(z)=(2z^2-6z+6),\
 h(z)=(z^2-6)e^{-z}\,.
\end{equation}
In the right half plane the function \(h\) is subordinate to the
function \(g\) in the following sense. For \(R>0\), let us
consider the contour \(\Gamma_R\), which consists of the interval
\(I_R\) of the imaginary axis and  the semicircle \(C_R\)
located in the right half plane:
\begin{equation}
\label{AuxFunc3}%
 \Gamma_R=I_R\cup{}C_R\,, \textup{ where }I_R=[-iR,iR], \
 C_R=\{z:\,|z|=R,\,\textup{Re}\,z\geq{}0\}\,.
\end{equation}
It is clear that \(|g(z)|\geq{}1.75\,|z|^2,\ \
|h(z)|\leq{}1.25\,|z|^2 \)  if  \(\ z\in{}C_R\) and \(R\) is large
enough. In particular, \(|g(z)|>|h(z)|\)  if \(z\in{}C_R\) and
\(R\) is large enough. On the imaginary axis, we have that
\begin{equation}
\label{BDiCa}%
 |g(iy)|^2=36+12y^2+4y^4,\ \
|h(iy)|^2=36+12y^2+y^4\,,\ -\infty<y<\infty,
\end{equation}
 In particular,
\(|g(z)|\geq{}|h(z)|\) for \(z\in{}I_R\), and the inequality is
strict for \(z\not=0\). Thus,
\begin{equation}
\label{AuxFunc4}%
|g(z)|\geq|h(z)|\ \ \textup{ for } \ \ z\in\Gamma_R \ \ \textup{
and \(R\)  large enough.}
\end{equation}
For \(0<\varepsilon<1\), consider the function
\begin{equation}
\label{AuxFunc5}%
f_{\varepsilon}(z)=g(z)+(1-\varepsilon)h(z)\,.
\end{equation}
The polynomial \(g\) has two simple roots:
\(z_{1,2}=\frac{3\pm{}i\sqrt{3}}{2}\). They are located in the
open right half plane. In view of \eqref{AuxFunc5} and Rouche's
theorem, for \(\varepsilon>0\) the function \(f_{\varepsilon}(z)\)
has precisely two roots \(z_1(\varepsilon),\,z_2(\varepsilon)\) in
the open right half plane. For \(\varepsilon\) positive and very
small, the roots \(z_1(\varepsilon),\,z_2(\varepsilon)\) are
located very close to the boundary point \(z=0\). This can be
shown by the asymptotic calculation. Since \(f_{\varepsilon}(z)
=6\varepsilon+{\textstyle\frac{1-\varepsilon}{4}}z^4+o(|z|^4)\) as
\(z\to{}0\),  the equation \(f_{\varepsilon}(z)=0\) has the roots
\(z_1(\varepsilon),\,z_2(\varepsilon)\), for which
\[z_{1,2}(\varepsilon)=\varepsilon^{\frac{1}{4}}24^{\frac{1}{4}}e^{\pm\frac{\pi}{4}}
(1+o(\varepsilon))\ \ \textup{as} \ \varepsilon\to{}+0\,.\] Since
the function \(f_{\varepsilon}\) has only
two roots in the open right half plane for \({\varepsilon}>0\),
\(z_1(\varepsilon),\,z_2(\varepsilon)\) are the only roots located in this
half plane.
 Since
\(f(z)=\lim_{\varepsilon\to{}+0}f_{\varepsilon}(z)\), the function
\(f(z)\) has no roots in the right upper half plane. (We apply
Hurwitz's Theorem.) From \eqref{BDiCa} it follows that
 the function \(f\) does not
vanish on the imaginary axis except at \(z=0\). At this
point the function \(f\) has a root of multiplicity four. Thus,
for \(q=4\) the function \(I_q(z)\) is in the Hurwitz class.
\end{proof}

 Claim 1 of Lemma \ref{HPEGM} is a consequence of Lemma \ref{HurCl}
 and Lemma \ref{QEqFour}.

\begin{rem}
\label{Agr}%
 From \eqref{IRPs1} and \eqref{ExpCal4} we obtain:
 \begin{equation}
 \label{ExAg}
 {\sideset{_1}{_1}\Psii}%
\left\{\genfrac{}{}{0pt}{1}%
{\frac{1}{2}}%
{1};%
\genfrac{}{}{0pt}{1}%
{\frac{1}{2}}%
{1+\frac{q}{2}};%
z\right\}=2\frac{(2z^2-6z+6)e^z+(z^2-6)}{z^4}\ \ \textup{for} \ \
q=4\,.
 \end{equation} This expression agrees with the asymptotic relation
 \eqref{Asymp1}.
\end{rem}
\paragraph{Analysis of the function
\(\mathcal{W}_{\partial{}(B^{\infty}\times{}0)}^{\,p}\):}  \ \ \ \
We may
calculate the function\\
\(\mathcal{W}_{\partial{}(B^{\infty}\times{}0)}^{\,p}\) \ \
asymptotically expressing it in terms of the appropriate
Fox-Wright function, \eqref{As1to}, and then make use of the
asymptotic expansion of this Fox-Wright function. However, we
derive the asymptotic approximation of the function\\ %
\(\mathcal{W}_{\partial{}(B^{\infty}\times{}0)}^{\,p}\) \,from the
asymptotic approximation of the function
\(\mathcal{M}_{B^{\infty}\times{}0^p}(t)\). From \eqref{MInt} and
\eqref{WInt} it follows that
\begin{equation}%
\label{RelMW}%
 \mathcal{W}_{\partial{}(B^{\infty}\times{}0)}^{\,p}(t)=
\frac{1}{2}\big(\mathcal{M}_{B^{\infty}\times{}0^p}(it)
+\mathcal{M}_{B^{\infty}\times{}0^p}(-it)\big)\,.
\end{equation}%
Comparing \eqref{RelMW} with \eqref{Asymp1}, we see that
\begin{multline}%
\label{AsyW}%
 \mathcal{W}_{\partial{}(B^{\infty}\times{}0)}^{\,p}(t)=\\
=%
 \begin{cases}%
\hspace*{4.0ex}2^{\frac{p}{2}}\hspace*{0.8ex}
t^{-\frac{p}{2}}\hspace*{1.0ex} \cos{(t-\frac{\pi{}p}{4}})-
\frac{2}{\Gamma(\frac{p}{2})}t^{-2}+r_1(t), &|\arg{}t|\leq\varepsilon\,,\\
2^{\frac{q}{2}}
(te^{-i\pi})^{-\frac{p}{2}}\cos{(t+\frac{\pi{}p}{4}})-
\frac{2}{\Gamma(\frac{p}{2})}t^{-2}+r_2(t),\ \ &|\arg{}t-\pi|\leq\varepsilon\,,\\
2^{\frac{p}{2}}(t\,e^{\mp{}\frac{i\pi}{2}})^{-\frac{p}{2}}\,e^{\mp{}it}(1+r_3(t))\,,
&|\arg{}t\pm{}\frac{\pi}{2}|\leq{}\frac{\pi}{2}-\varepsilon\,,
\end{cases}%
\end{multline}%
where the remainders \(r_1(t),\,r_2(t),\,r_3(t)\) can be
estimated as follows
\begin{subequations}
\label{EstRema}
\begin{equation}%
\label{EstRema1}%
|r_1(t)|\leq{}C_1(\varepsilon)\left(|t|^{-(1+\frac{p}{2})}e^{|\textup{Im}\,t|}+|t|^{-3}\right)\,,
\ \ \ \ \ \ \ |\arg{}t|\leq\varepsilon\,,
\end{equation}%
\begin{equation}%
\label{EstRema2}%
|r_2(t)|\leq{}C_2(\varepsilon)\left(|t|^{-(1+\frac{p}{2})}e^{|\textup{Im}\,t|}+|t|^{-3}\right)\,,
\ \    |\arg{}t-\pi|\leq\varepsilon\,,
\end{equation}%
\begin{equation}%
\label{EstRema3}%
|r_3(t)|\leq{}C_3(\varepsilon)\,|t|^{-1}\,,\hspace*{9.0ex}
 |\arg{}t\mp{}\frac{\pi}{2}|\leq{}\frac{\pi}{2}-\varepsilon\,,
\end{equation}%
\end{subequations}%
 and
\(C_1(\varepsilon)<\infty,\,C_2(\varepsilon)<\infty,\,C_3(\varepsilon)<\infty\)
for every \(\varepsilon:\,0<\varepsilon<\frac{\pi}{2}.\) Moreover,
the function
\(\mathcal{W}_{\partial{}(B^{\infty}\times{}0)}^{\,p}(t)\) is an even
function of \(t\) and takes real values for real \(t\).

From \eqref{AsyW}, \eqref{EstRema} it follows that for every
\(\varepsilon>0\) the function
\(\mathcal{W}_{\partial{}(B^{\infty}\times{}0)}^{\,p}(t)\) can
not have more that finitely many roots within the angular domains
\(\{t:\,|\arg{}(t\mp{}\frac{\pi}{2}|\leq{}\frac{\pi}{2}-\varepsilon\}\,.\)
Within the angular domain \(\{t:\,|\arg{}t|\leq{}\varepsilon\}\), the
function
\(\mathcal{W}_{\partial{}(B^{\infty}\times{}0)}^{\,p}(t)\) has
infinitely many roots, and these roots are asymptotically close to
the roots of the approximating function
\[f_p(t)=2^{\frac{p}{2}} t^{-\frac{p}{2}}
\cos{(t-\textstyle{\frac{\pi{}p}{4}}})-
\frac{2}{\Gamma(\frac{p}{2})}t^{-2}\,,\ \ \
|\arg{}t|\leq{}\varepsilon\,.\]
 (Since the function
\(\mathcal{W}_{\partial{}(B^{\infty}\times{}0)}^{\,p}(t)\) is
even, there is no need to study its behavior within the angular domain
\(\{t:\,|\arg{}t-\pi|\leq{}\varepsilon\}\).) The behavior of thr roots
for the approximating equation \(f_p(t)=0\), that is, the equation
\begin{equation}
\label{ApprEqp}%
 \cos{(t-\textstyle{\frac{\pi{}p}{4}}})-
\frac{2^{1-\frac{p}{2}}}{\Gamma(\frac{p}{2})}\,t^{\frac{p}{2}-2}=0\,,
\ \ \ |\arg{}t|\leq{}\varepsilon\,,
\end{equation}
depends on \(p\).%
\begin{itemize}
\item[]%
 If \(0<p<4\), then all but finitely many roots of the
equation \eqref{ApprEqp} are real and simple, and these roots are
asymptotically close to the roots of the equation
\(\cos{}(t-\textstyle{\frac{\pi{}p}{4}})=0\).  %
\item[]%
If \(p=4\), then all but finitely many roots of the equation
\eqref{ApprEqp} are real and simple and these roots are
asymptotically close to the roots of the equation
\(\cos{}(t-\pi)-\frac{1}{2}=0\).  %
\item[]%
If \(p>4\), then all but finitely many roots of the equation
\eqref{ApprEqp} are \textsf{non}-real and simple,  located
symmetrically with respect to the real axis, and are
asymptotically close to the `logarithmic parabola'
\[|y|=({\textstyle\frac{p}{2}}-1)
\ln(|x|+1)+\ln{}c_p\,,\ \
c_p=\frac{2^{1-\frac{p}{2}}}{\Gamma(\frac{p}{2})}\,,\ \ \
0\leq{}x<\infty\,.\]
\end{itemize}
Thus we prove the following
\begin{lem}
\label{AnWp}%
 For each \(p:\,0\leq{}p<\infty\), the function
 \(\mathcal{W}_{\partial{}(B^{\infty}\times{}0)}^{\,p}(t)\) has
infinitely many roots.
 All but finitely many these roots are
 simple. They lie symmetric with respect to the point \(z=0\).
 \begin{enumerate}
 \item
 If \(0\leq{}p\leq{}4\), then all but finitely many these
 roots are real\,;
 \item
 If \ \(4<p\), then all but finitely many these
 roots are non-real.
 In particular, the function
\(\mathcal{W}_{\partial{}(B^{\infty}\times{}0)}^{\,p}(t)\) does
not belong to the Laguerre-P\'olya class
\(\mathscr{L}\text{-}\mathscr{P}\).
 \end{enumerate}
\end{lem}
\begin{lem}
\label{WPBLP}%
 For \(0<p\leq{}2\), as well as for \(p=4\), the
function \(\mathcal{W}_{\partial{}(B^{\infty}\times{}0)}^{\,p}\)
belongs to the Laguerre-P\'olya class.
\end{lem}
\begin{proof} The equality \eqref{RelMW} will serve as a starting point for
our considerations. If the function
\(\mathcal{M}_{B^{\infty}\times{}0^p}(t)\) is in the Hurwitz class
\(\mathscr{H}\), then the function
\begin{equation}
\label{RHC} %
\Omega(t)=\mathcal{M}_{B^{\infty}\times{}0^p}(it)
\end{equation}
is in the class \(\mathscr{P}\) (in the sense of \cite{Lev1}).
\begin{defn}\textup{\cite[Chapter VII,
Section 4]{Lev1}. } \label{ClP} An entire function \(\Omega(t)\) of
exponential type belongs to the class
\(\mathscr{P}\) if:%
 \begin{enumerate}
 \item
\(\Omega(t)\) has no roots in the closed lower half plane
\(\{t:\,\textup{Im}\,t\leq{}0\}.\)
 \item
 The defect \(d_{\Omega}\) of the function \(\Omega\) is
 non-negative, where
 \[2d_{\Omega}=\varlimsup\limits_{r\to+\infty}\frac{\ln|\Omega(-ir)|}{r}-
\varlimsup\limits_{r\to+\infty}\frac{\ln|\Omega(ir)|}{r}\,. \]
 \end{enumerate} \ \\
\end{defn}

 In the book \cite{Lev1} of B.Ya.Levin, the following version of the Hermite-Biehler
Theorem is proved:
\begin{nthm}\textup{\,\cite[Chapter VII,
Section 4, Theorem 7]{Lev1}} If an entire function \(\Omega(t)\) is in
class \(\mathscr{P}\), then its real and imaginary parts,
\({}^{\mathscr{R}}\Omega(t)\) and \({}^{\mathscr{I}}\Omega(t)\) respectively:
\[{}^{\mathscr{R}}\Omega(t)=\frac{\Omega(t)+\overline{\Omega(\overline{t})}}{2},
\quad{}{}^{\mathscr{I}}\Omega(t)=\frac{\Omega(t)-\overline{\Omega(\overline{t})}}{2i},\]
possess the following properties:
\begin{enumerate}
\item
The roots of each of the functions \({}^{\mathscr{R}}\Omega(t)\)
and \({}^{\mathscr{I}}\Omega(t)\) are real and simple;
\item
The zero loci of the functions \({}^{\mathscr{R}}\Omega(t)\) and
\({}^{\mathscr{I}}\Omega(t)\) interlace.
\end{enumerate}
\end{nthm}
Let us apply this theorem to the function \(\Omega(t)\) defined by
\eqref{RHC}:
\[\Omega(t)=\mathcal{M}_{B^{\infty}\times{}0^p}(it)\,.\]%
Taking into account that the function
\(\mathcal{M}_{B^{\infty}\times{}0^p}\) is
real:\\
\(\mathcal{M}_{B^n\times{}0^p}(t)\equiv%
\overline{\mathcal{M}_{B^{\infty}\times{}0^p}(\overline{t})}\), or
equivalently
\(\mathcal{M}_{B^n\times{}0^p}(-it)
\equiv\overline{\mathcal{M}_{B^{\infty}\times{}0^p}(i\overline{t})}\).
Hence,
\[{}^{\mathscr{R}}\Omega(t)=\frac{1}{2}
\big(\mathcal{M}_{B^n\times{}0^p}(it)+\mathcal{M}_{B^{\infty}\times{}0^p}(-it)\big)\,,\]
or because of  \eqref{RelMW}:
\[{}^{\mathscr{R}}\Omega(t)= \mathcal{W}_{\partial{}(B^{\infty}\times{}0)}^{\,p}(t)\,.\]
Thus, the following result holds:
\begin{lem}
\label{CondL}%
 If the function \(\mathcal{M}_{B^n\times{}0^p}(t)\) belongs to
 the Hurwitz class \(\mathscr{H}\), then the function
 \(\mathcal{W}_{\partial{}(B^{\infty}\times{}0)}^{\,p}(t)\)
 belongs to the Laguerre-P\'olya class \(\mathscr{L}\text{-}\mathscr{P}\).
\end{lem}
Combining Lemma \ref{CondL} with Lemmas \ref{HurCl} and
\ref{QEqFour}, we obtain Lemma \ref{WPBLP}.%
\end{proof}%
\begin{rem}%
\label{Never}%
The real part \({}^{\mathscr{R}}\Omega(t)\) for  \(\Omega(t)\)
from \eqref{RHC},
has infinitely many non-real roots if \(p>4\). Nevertheless,  all
roots of the imaginary part \({}^{\mathscr{I}}\Omega(t)\) are real
for every \(p\geq{}0\).
\end{rem}%
 Indeed, according to
\eqref{MInt} and \eqref{RHC},
\begin{equation}
\label{ImPaO}%
 {}^{\mathscr{I}}\Omega(t)=p\int\limits_{0}^{1}
(1-\xi^2)^{\frac{p}{2}-1}\xi\,{}\sin{\xi{}t}\,d\xi\,.
\end{equation}
According to A.Hurwitz (see \textup{\cite[section 15.27]{Wat}}),
 all roots of the entire function
\[\Big(\frac{t}{2}\Big)^{-\nu}J_{\nu}(t)=
\sum\limits_{l=0}^{\infty}\frac{(-1)^l}{l!\,\Gamma(\nu+l+1)}\Big(\frac{t^2}{4}\Big)^l\]
 are real for every \(\nu\geq{}-1\). (\(J_{\nu}(t)\) is the
Bessel function of index \(\nu\).) For \(\nu>-\frac{1}{2}\), the
function \((\frac{t}{2})^{-\nu}J_{\nu}(t)\) admits the integral
representation
\begin{equation}
\label{InReBe} \frac{1}{2}\Big(\frac{t}{2}\Big)^{-\nu}J_{\nu}(t)=
\frac{1}{\Gamma(\nu+\frac{1}{2})\,\Gamma(\frac{1}{2})}
\int\limits_0^1(1-\xi^2)^{\nu-\frac{1}{2}}\cos{t\xi}\,d\xi\,.
\end{equation}
Thus, for \(\nu>-\frac{1}{2}\) all roots of the entire function
\(\int\limits_0^1(1-\xi^2)^{\nu-\frac{1}{2}}\cos{t\xi}\,d\xi\) are
real.  If all roots of a real entire function of exponential type
are real, then all roots of its derivative are real as well. Thus,
for \(\nu>-\frac{1}{2}\) all roots of the function
\(\int\limits_0^1(1-\xi^2)^{\nu-\frac{1}{2}}\xi\sin{t\xi}\,d\xi\)
are real. However, for \(\nu=\frac{p-1}{2}\), the latter function
coincides with the function
\(\frac{1}{p}\,\,{}^{\mathscr{I}}\Omega(t)\).

 For \(p=3\), we do not know whether the function
 \(\mathcal{W}_{\partial{}(B^{\infty}\times{}0)}^{\,p}(t)\)
 belongs to the Laguerre-P\'olya class or not. Our conjecture is
 that YES.
 Let us formulate our conjectures in terms of the Fox-Wright
 functions.\\

 \noindent
 \textbf{Conjecture 1.} \textit{For \(0\leq{}\lambda\leq{}2\), all roots of
 the Fox-Wright function
 \begin{equation}
 \label{FRHur}
 {\sideset{_1}{_1}\Psii}%
\left\{\genfrac{}{}{0pt}{1}%
{\frac{1}{2}}%
{1};%
\genfrac{}{}{0pt}{1}%
{\frac{1}{2}}%
{1+\lambda};%
t\right\}=\sum\limits_{0\leq{}k<\infty}
\frac{\Gamma(\frac{k}{2}+1)}{\Gamma(\frac{k}{2}+1+\lambda)}\frac{t^k}{k!}
 \end{equation}
 lie in the open left half plane}.\\

 We proved that Conjecture 1 holds for
 \(0\leq{}\lambda\leq{}1\) and for \(\lambda=2\).\\

 \noindent
 \textbf{Conjecture 2.}
 \textit{For \(0\leq{}\lambda\leq{}2\), all roots of
 the Fox-Wright function
 \begin{equation}
 \label{FRLP}
 {\sideset{_1}{_2}\Psii}%
\left\{\genfrac{}{}{0pt}{1}%
{1}%
{1};%
\genfrac{}{}{0pt}{1}%
{1}%
{\frac{1}{2}\,}\,%
\genfrac{}{}{0pt}{1}%
{1}%
{1+\lambda};%
t\right\}=\sum\limits_{0\leq{}l<\infty}
\frac{\Gamma(l+1)}{\Gamma(l+\frac{1}{2})\Gamma(l+1+\lambda)}\,\frac{t^l}{l!}
\end{equation}
 are negative and simple.}\\

 We proved that Conjecture 2 holds for
 \(0\leq{}\lambda\leq{}1\), and for \(\lambda=2\).\\

From Hermite-Biehler theorem it follows that if  Conjecture 1 holds
for some \(\lambda\), then also  Conjecture 2 holds for this \(\lambda\).

The conjectures 1 and 2 are related to  deep questions related
to  `meromorphic multiplier sequences'. (See \cite[Problem
1.1]{CrCs4}.)

%%%%%%%%%%%%%%%%%%%%%%%%%%%%%%%%%%%%%%%%%%%%%%%%%%%%%%%%%%%%%%%%%%%%%
\section{Concluding remarks. \label{ConRem}}
\textbf{1.} In the present paper, little use was made of
geometric methods. The only general geometric tools used were
the Alexandrov-Fenchel inequalities. We did not use the
monotonicity properties of the coefficients belonging to the Steiner
polynomials. If \(V_0,\,V_1,\,V_2\) are convex sets, such that
\[V_1\subseteq{}V_0\subseteq{}V_2\]
 and
 \(S_{V_0}(t),\,S_{V_1}(t),\,S_{V_2}(t)\) are their Steiner
 polynomials,
 \[S_{V_j}(t)=\sum\limits_{0\leq{}k\leq{}n}s_k(V_j)t^k,\quad{}j=0,\,1,\,2,\]
then for the coefficients of these polynomials the inequalities
\[s_k(V_1)\leq{}s_k(V_0)\leq{}s_k(V_2)\,,\quad{}0\leq{}k\leq{}n.\]
hold. In this context, the use of the \textit{Kharitonov
Criterion for Stability} may be helpful. (Concerning the Kharitonov
Criterion see Chapters 5 and 7 of the book \cite{BCK} and the
literature quoted there.) The Kharitonov Criterion deals with the
`interval stability' of polynomials. In its simplest form, this
criterion allows one to determine whether the polynomial
\begin{equation}%
\label{IntPO}
 f(t)=\sum\limits_{0\leq{}k\leq{}n}a_kt^k
\end{equation}
with real coefficients \(a_k\) is stable, given
that these coefficients belong to some intervals:
\begin{equation}%
\label{CoefInt}%
 a_k^{-}\leq{}a_k\leq{}a_k^{+}\,,\quad{}0\leq{}k\leq{}n\,.
\end{equation}
Applying this criterion, one needs to construct certain polynomials
 from the given numbers \(a_k^{-},\,a_k^{+},\
0\leq{}k\leq{}n\,.\) (There are finitely many such polynomials.)
If all these polynomials are stable, then the arbitrary polynomial
\(f(t)\), \eqref{IntPO}, whose coefficients satisfy the
inequalities \eqref{CoefInt}, is stable.

\noindent%
\textbf{2.} In the example of a convex set \(V\), whose Steiner
polynomial is not dissipative, the set \(V\) is very `flat'
in some direction. (See Theorem \ref{NMR}.)

\begin{defn}
\label{DefIso} The solid convex set \(V\),
\(V\subset{}\mathbb{R}^n\), is said to be isotropic (with respect
to the point \(0\)), if the integral
\[\int\limits_{V}|\langle{}x,e\rangle|^2dv_n(x)\]
is constant with respect to \(e\), for every vector
\(e\in\mathbb{R}^n\) such that \(\langle{}e,e\rangle=1\,.\) Here
\(\langle{}\,.\,,\,.\,\rangle\) is the standard scalar product in
\(\mathbb{R}^n\), and \(dv_n(x)\) is the standard
\(n\)-dimensional differential volume element.
\end{defn}

\noindent%
\textbf{Question.} \textit{What can one say about Steiner
polynomials associated with a convex set \(V\) which is isotropic?
}
%%%%%%%%%%%%%%%%%%%%%%%%%%%%%%%%%%%%%%%%%%%%%%%%%%%%%%%%%%%%%%%%%%%%
%%%%%
%%%%%%%%%%%%%%%%%%%%%%%%%%%%%%%%%%%%%%%%%%%%%%%%%%%%%%%%%%%%%%%%%%%%%

\end{document}